\documentclass{amsart}

\usepackage{amsfonts,amssymb,verbatim,amsmath,amsthm,latexsym,textcomp,amscd, hyperref}
\usepackage{epsfig}
\usepackage{amsmath,amsthm,graphics,float}
\usepackage{graphicx}
\usepackage{xcolor}
\usepackage{url}
\usepackage{thm-restate}

\begin{document}

\newtheorem{theorem}{Theorem}[section]
\newtheorem{prop}[theorem]{Proposition}
\newtheorem{lemma}[theorem]{Lemma}
\newtheorem{cor}[theorem]{Corollary}
\newtheorem{cond}[theorem]{Condition}
\newtheorem{ing}[theorem]{Ingredients}
\newtheorem{conj}[theorem]{Conjecture}
\newtheorem{claim}[theorem]{Claim}
\newtheorem{constr}[theorem]{Construction}
\newtheorem{rem}[theorem]{Remark}

\newtheorem*{theorem*}{Theorem}
\newtheorem*{modf}{Modification for arbitrary $n$}
\newtheorem{qn}[theorem]{Question}
\newtheorem{condn}[theorem]{Condition}
\newtheorem*{BGIT}{Bounded Geodesic Image Theorem}
\newtheorem*{BI}{Behrstock Inequality}
\newtheorem*{QCH}{Wise's Quasiconvex Hierarchy Theorem}

\theoremstyle{definition}
\newtheorem{defn}[theorem]{Definition}
\newtheorem{eg}[theorem]{Example}
\newtheorem{rmk}[theorem]{Remark}

\newcommand{\map}{\rightarrow}
\newcommand{\boundary}{\partial}
\newcommand{\C}{{\mathbb C}}
\newcommand{\integers}{{\mathbb Z}}
\newcommand{\natls}{{\mathbb N}}
\newcommand{\ratls}{{\mathbb Q}}
\newcommand{\reals}{{\mathbb R}}
\newcommand{\proj}{{\mathbb P}}
\newcommand{\lhp}{{\mathbb L}}
\newcommand{\tr}{{\operatorname{Tread}}}
\newcommand{\rs}{{\operatorname{Riser}}}
\newcommand{\tube}{{\mathbb T}}
\newcommand{\cusp}{{\mathbb P}}
\newcommand\AAA{{\mathcal A}}
\newcommand\BB{{\mathcal B}}
\newcommand\CC{{\mathcal C}}
\newcommand\ccd{{{\mathcal C}_\Delta}}
\newcommand\DD{{\mathcal D}}
\newcommand\EE{{\mathcal E}}
\newcommand\FF{{\mathcal F}}
\newcommand\GG{{\mathcal G}}
\newcommand\HH{{\mathcal H}}
\newcommand\II{{\mathcal I}}
\newcommand\JJ{{\mathcal J}}
\newcommand\KK{{\mathcal K}}
\newcommand\LL{{\mathcal L}}
\newcommand\MM{{\mathcal M}}
\newcommand\NN{{\mathcal N}}
\newcommand\OO{{\mathcal O}}
\newcommand\PP{{\mathcal P}}
\newcommand\QQ{{\mathcal Q}}
\newcommand\RR{{\mathcal R}}
\newcommand\SSS{{\mathcal S}}

\newcommand\TT{{\mathcal T}}
\newcommand\ttt{{\mathcal T}_T}
\newcommand\tT{{\widetilde T}}
\newcommand\UU{{\mathcal U}}
\newcommand\VV{{\mathcal V}}
\newcommand\WW{{\mathcal W}}
\newcommand\XX{{\mathcal X}}
\newcommand\YY{{\mathcal Y}}
\newcommand\ZZ{{\mathcal Z}}
\newcommand\CH{{\CC\HH}}
\newcommand\TC{{\TT\CC}}
\newcommand\EXH{{ \EE (X, \HH )}}
\newcommand\GXH{{ \GG (X, \HH )}}
\newcommand\GYH{{ \GG (Y, \HH )}}
\newcommand\PEX{{\PP\EE  (X, \HH , \GG , \LL )}}
\newcommand\MF{{\MM\FF}}
\newcommand\PMF{{\PP\kern-2pt\MM\FF}}
\newcommand\ML{{\MM\LL}}
\newcommand\mr{{\RR_\MM}}
\newcommand\tmr{{\til{\RR_\MM}}}
\newcommand\PML{{\PP\kern-2pt\MM\LL}}
\newcommand\GL{{\GG\LL}}
\newcommand\Pol{{\mathcal P}}
\newcommand\half{{\textstyle{\frac12}}}
\newcommand\Half{{\frac12}}
\newcommand\Mod{\operatorname{Mod}}
\newcommand\Area{\operatorname{Area}}
\newcommand\ep{\epsilon}
\newcommand\hhat{\widehat}
\newcommand\Proj{{\mathbf P}}
\newcommand\U{{\mathbf U}}
 \newcommand\Hyp{{\mathbf H}}
\newcommand\D{{\mathbf D}}
\newcommand\Z{{\mathbb Z}}
\newcommand\R{{\mathbb R}}
\newcommand\bN{\mathbb{N}}
\newcommand\s{{\Sigma}}
\renewcommand\P{{\mathbb P}}
\newcommand\Q{{\mathbb Q}}
\newcommand\E{{\mathbb E}}
\newcommand\til{\widetilde}
\newcommand\length{\operatorname{length}}
\newcommand\BU{\operatorname{BU}}
\newcommand\gesim{\succ}
\newcommand\lesim{\prec}
\newcommand\simle{\lesim}
\newcommand\simge{\gesim}
\newcommand{\simmult}{\asymp}
\newcommand{\simadd}{\mathrel{\overset{\text{\tiny $+$}}{\sim}}}
\newcommand{\ssm}{\setminus}
\newcommand{\diam}{\operatorname{diam}}
\newcommand{\pair}[1]{\langle #1\rangle}
\newcommand{\T}{{\mathbf T}}
\newcommand{\inj}{\operatorname{inj}}
\newcommand{\pleat}{\operatorname{\mathbf{pleat}}}
\newcommand{\short}{\operatorname{\mathbf{short}}}
\newcommand{\vertices}{\operatorname{vert}}
\newcommand{\collar}{\operatorname{\mathbf{collar}}}
\newcommand{\bcollar}{\operatorname{\overline{\mathbf{collar}}}}
\newcommand{\I}{{\mathbf I}}
\newcommand{\tprec}{\prec_t}
\newcommand{\fprec}{\prec_f}
\newcommand{\bprec}{\prec_b}
\newcommand{\pprec}{\prec_p}
\newcommand{\ppreceq}{\preceq_p}
\newcommand{\sprec}{\prec_s}
\newcommand{\cpreceq}{\preceq_c}
\newcommand{\cprec}{\prec_c}
\newcommand{\topprec}{\prec_{\rm top}}
\newcommand{\Topprec}{\prec_{\rm TOP}}
\newcommand{\fsub}{\mathrel{\scriptstyle\searrow}}
\newcommand{\bsub}{\mathrel{\scriptstyle\swarrow}}
\newcommand{\fsubd}{\mathrel{{\scriptstyle\searrow}\kern-1ex^d\kern0.5ex}}
\newcommand{\bsubd}{\mathrel{{\scriptstyle\swarrow}\kern-1.6ex^d\kern0.8ex}}
\newcommand{\fsubeq}{\mathrel{\raise-.7ex\hbox{$\overset{\searrow}{=}$}}}
\newcommand{\bsubeq}{\mathrel{\raise-.7ex\hbox{$\overset{\swarrow}{=}$}}}
\newcommand{\tw}{\operatorname{tw}}
\newcommand{\base}{\operatorname{base}}
\newcommand{\trans}{\operatorname{trans}}
\newcommand{\rest}{|_}
\newcommand{\bbar}{\overline}
\newcommand{\UML}{\operatorname{\UU\MM\LL}}
\renewcommand{\d}{\operatorname{diam}}
\newcommand{\hs}{{\operatorname{hs}}}
\newcommand{\EL}{\mathcal{EL}}
\newcommand{\tsum}{\sideset{}{'}\sum}
\newcommand{\tsh}[1]{\left\{\kern-.9ex\left\{#1\right\}\kern-.9ex\right\}}
\newcommand{\Tsh}[2]{\tsh{#2}_{#1}}
\newcommand{\qeq}{\mathrel{\approx}}
\newcommand{\Qeq}[1]{\mathrel{\approx_{#1}}}
\newcommand{\qle}{\lesssim}
\newcommand{\Qle}[1]{\mathrel{\lesssim_{#1}}}
\newcommand{\simp}{\operatorname{simp}}
\newcommand{\vsucc}{\operatorname{succ}}
\newcommand{\vpred}{\operatorname{pred}}
\newcommand\fhalf[1]{\overrightarrow {#1}}
\newcommand\bhalf[1]{\overleftarrow {#1}}
\newcommand\sleft{_{\text{left}}}
\newcommand\sright{_{\text{right}}}
\newcommand\sbtop{_{\text{top}}}
\newcommand\sbot{_{\text{bot}}}
\newcommand\sll{_{\mathbf l}}
\newcommand\srr{_{\mathbf r}}
\newcommand\geod{\operatorname{\mathbf g}}
\newcommand\mtorus[1]{\boundary U(#1)}
\newcommand\A{\mathbf A}
\newcommand\Aleft[1]{\A\sleft(#1)}
\newcommand\Aright[1]{\A\sright(#1)}
\newcommand\Atop[1]{\A\sbtop(#1)}
\newcommand\Abot[1]{\A\sbot(#1)}
\newcommand\boundvert{{\boundary_{||}}}
\newcommand\storus[1]{U(#1)}
\newcommand\Momega{\omega_M}
\newcommand\nomega{\omega_\nu}
\newcommand\twist{\operatorname{tw}}
\newcommand\SSSS{{\til{\mathcal S}}}
\newcommand\modl{M_\nu}
\newcommand\MT{{\mathbb T}}
\newcommand\dw{{d_{weld}}}
\newcommand\dt{{d_{te}}}
\newcommand\Teich{{\operatorname{Teich}}}
\renewcommand{\Re}{\operatorname{Re}}
\renewcommand{\Im}{\operatorname{Im}}
\newcommand{\mc}{\mathcal}
\newcommand{\ccs}{{\CC(S)}}
\newcommand{\mtdw}{{(\til{M_T},\dw)}}
\newcommand{\tmtdw}{{(\til{M_T},\dw)}}
\newcommand{\tmldw}{{(\til{M_l},\dw)}}
\newcommand{\mtdt}{{(\til{M_T},\dt)}}
\newcommand{\tmtdt}{{(\til{M_T},\dt)}}
\newcommand{\tmldt}{{(\til{M_l},\dt)}}
\newcommand{\trvw}{{\tr_{vw}}}
\newcommand{\ttrvw}{{\til{\tr_{vw}}}}
\newcommand{\but}{{\BU(T)}}
\newcommand{\ilkv}{{i(lk(v))}}
\newcommand{\pslc}{{\mathrm{PSL}_2 (\mathbb{C})}}
\newcommand{\tttt}{{\til{\ttt}}}
\newcommand{\bcomment}[1]{\textcolor{blue}{#1}}

\newcommand{\Star}{\operatorname{star}}
\newcommand{\jfmchange}[1]{{\color{purple}{#1}}}

\newcommand{\defstyle}[1]{\textbf{#1}}
\newcommand{\emphstyle}[1]{\emph{#1}}

\title{Cubulating Surface-by-free Groups}

\author[Mahan Mj]{%
	Mahan Mj\\ with an appendix by Jason Manning, Mahan Mj, and Michah Sageev}

\address{Department of Mathematics, Cornell University, 573 Malott Hall, 
Ithaca, NY  14853,  U.S.A.}
\email{jfmanning@cornell.edu}
\thanks{JM supported by Simons Collaboration Grants \#524176 and \#942496.}

\address{School of Mathematics, Tata Institute of Fundamental Research, Mumbai-40005, India}
\email{mahan@math.tifr.res.in}
\thanks{ MM partly supported by a DST JC Bose Fellowship, a grant from the Infosys Foundation, Matrics research project grant  MTR/2017/000005  and CEFIPRA  project No.\ 5801-1.}

\address{Mathematics Department,
	Technion - Israel Institute of Technology,
	Haifa, 32000,
	Israel   }
 \email{sageevm@technion.ac.il}
 \thanks{ MS partly supported by the Israel Science Foundation (grant 1026/15).}


\date{\today}

\begin{abstract}  Let $$1 \to H \to G \to Q \to 1$$ be an exact sequence where $H= \pi_1(S)$ is the fundamental 
group of a closed surface $S$ of genus greater than one, $G$ is hyperbolic and $Q$ is finitely generated free.
The aim of this paper is to provide sufficient conditions to prove that $G$ is cubulable and construct examples satisfying these conditions. The main result may be thought of as a
  combination theorem for virtually special hyperbolic groups when the amalgamating  subgroup is not quasiconvex. Ingredients include the theory of tracks, the quasiconvex hierarchy theorem of Wise,  the distance estimates in the mapping class group from subsurface projections due to Masur-Minsky 
 and the model geometry for doubly degenerate Kleinian surface groups used in the proof of the ending lamination theorem.
 
 An appendix to this paper by Manning, Mj, and Sageev proves a reduction theorem by showing that cubulability of $G$ follows from the existence of an essential incompressible quasiconvex track
 in a surface bundle over a graph with fundamental group $G$.

\end{abstract}

\maketitle

\tableofcontents

\section{Introduction} This paper lies at the interface of two themes in geometric group theory that have attracted a lot of attention of late: convex cocompact subgroups of mapping class groups, and cubulable hyperbolic groups. 
Let 
$  1 \to H \to G \to Q \to 1$ be an exact sequence with $H$ a closed surface group and $Q$ a convex cocompact subgroup of $MCG(S)$. It follows that $G$ is 
hyperbolic. In fact convex cocompactness of 
$Q$ is equivalent to hyperbolicity of $G$ \cite{farb-coco,hamen} {(see also \cite[Theorem 1.2]{kl-coco} where other equivalent notions of convex cocompactness are given)}.
The only known examples of convex cocompact subgroups $Q$ of $MCG(S)$ are virtually free. 
Cubulable groups, by which we mean  groups acting freely, properly discontinuously and cocompactly by isometries (cellular isomorphisms) 
on a CAT(0) cube complex, have been objects of much attention over the last few years particularly due to path-breaking work of Agol and Wise.
In this paper, we shall address the following  question that lies at the interface of these two themes:
\begin{qn}\label{splqn} Let 
\begin{equation}\label{exactsequence}  1 \to H \to G \to Q \to 1
\end{equation}
  be an exact sequence of  groups, where $H= \pi_1(S)$ is the fundamental 
	group of a closed surface $S$ of genus greater than one, $G$ is hyperbolic and $Q$ is a finitely generated free group of rank $n$.
	\begin{enumerate}
		\item[(i)] Does $G$ have a quasiconvex hierarchy? 
		Equivalently (by Wise's  Theorem 
		\ref{wise-hierarchy} below), is $G$ virtually special  cubulable? 
		\item[(ii)] In particular, is $G$ linear?
	\end{enumerate}
\end{qn}
{Question \ref{splqn} makes sense even when $Q$ is not  free. However, in this paper we shall only address the case where $Q$ is free,
 providing} sufficient conditions on the exact sequence \eqref{exactsequence} guaranteeing an affirmative answer to Question \ref{splqn}. We shall also  construct examples satisfying these conditions. A somewhat surprising consequence, using work of Kielak \cite{kielak}, is the existence of groups $G$ as in Question \ref{splqn} that surject to $\Z$ with finitely generated kernel  (Section \ref{vfiber}).
Note that an affirmative answer to the first question  in Question \ref{splqn} implies an affirmative answer to the second. To the best of our knowledge, when the rank of $Q$ is greater than one, there was no known example of a linear $G$ as above, and the answer is not known in general. This is
perhaps not too surprising as  linearity even in the case $n=1$ really goes back to Thurston's {hyperbolization of atoroidal fibered 3-manifolds} \cite{thurston-hypstr2} and the latter feeds into the cubulability of these 3-manifold groups. 

The main theorem of this paper may also be looked upon as evidence for a combination theorem of cubulable groups along non-quasiconvex subgroups. Let us specialize to the  case $n=2$ in Question \ref{splqn} for the time being. Let $A$ (resp.\ $B$) be the fundamental group of a closed hyperbolic 3-manifold $M_1$ (resp.\ $M_2$)  fibering over the circle with fiber a closed surface $S$ of genus at least 2. Let $C = \pi_1(S)$ be the fundamental group of the fiber and we ask if  $A\ast_CB$ is cubulable. We point out a preliminary caveat. Since the distortion 
of the fiber subgroup $C$ in $A$ is exponential, the double of $A$ along $C$ given by $G_0 = A \ast_C A$ has an exponential isoperimetric inequality.  Since $CAT(0)$ groups satisfy a quadratic isoperimetric inequality, $G_0$ cannot be a $CAT(0)$ group; in particular $G_0$ is not cubulable. It therefore makes sense to demand that the group $G$ resulting from the combination is hyperbolic.
Unlike in the existing literature (see  \cite{hw-annals, hw-inv, wise-hier} for instance),
the amalgamating subgroup $C$  is not quasiconvex in $A$ or $B$.

We briefly indicate the broader framework in which our results sit.
The starting point of this work is Wise's quasiconvex hierarchy theorem \cite{wise-hier} for hyperbolic cubulable groups:

\begin{theorem} \cite{wise-hier} \label{wise-hierarchy}  Let $G$ be a finite graph of hyperbolic groups  so that  $G$ is hyperbolic, { the vertex groups are virtually special cubulable and the edge groups are quasiconvex} in $G$.  Then $G$ is virtually special cubulable.  \end{theorem}

Further, a celebrated Theorem of  Agol \cite{agol-vh} proves a  conjecture due to Wise \cite{wise-hier, wise-cbms} and establishes:

\begin{theorem} \cite{agol-vh} \label{spl} Let $G$ be hyperbolic and cubulable. Then $G$ is virtually special.  \end{theorem}

The sufficient conditions we provide are a first attempt at relaxing the quasiconvexity hypothesis in Theorem \ref{wise-hierarchy}: is there a 
combination theorem for  cubulated groups along \emphstyle{non-quasiconvex} subgroups? We explicitly state the general question  below:

\begin{qn}\label{genlqn}
  Let $G$ be a finite graph of hyperbolic groups (e.g. $G=A\ast_CB$ or $G=A\ast_C$) so that the vertex groups are virtually special cubulable and $G$ is hyperbolic.  Is $G$ virtually special cubulable?
\end{qn}

 Related questions have been raised by Wise \cite[Problems 13.5, 13.15]{wise-icm}, for  instance when each of $A, B$ are hyperbolic free-by-cyclic groups of the form 
$$1 \to F_k \to G \to \Z \to 1$$
and $C$ is the normal subgroup $F_k$.

\subsection{Motivation and context} The base case of Question \ref{splqn} is when $Q=\Z$ and $G$ is the fundamental group of a 3-manifold $M$ fibering over the circle with fiber $S$. We briefly recall what goes into the proof  \cite{agol-vh,wise-hier} of the virtually special cubulability of such $G$. By Thurston's {hyperbolization theorem for atoroidal fibered 3-manifolds} \cite{thurston-hypstr2} $M$ admits a hyperbolic structure. Then, by work of Kahn-Markovic \cite{km-surf} there are many immersed quasiconvex surfaces in $M$. These are enough to separate pairs of points on $\partial G = S^2$. Hence by work of Bergeron-Wise \cite{bw}, $G$ is cubulable. Finally, by Agol's theorem \ref{spl} \cite{agol-vh}, $G$ is virtually special. 
In the restricted case that the first Betti number $b_1(G) > 1$, {an
embedded surface representing a class in the boundary of a fibered face of the
unit ball in the Thurston norm
must be quasi-Fuchsian
 (see \cite[p. 278]{clr}, \cite{fried}).} Starting with such an {\it embedded} quasi-Fuchsian surface in $M$ and using Wise's quasiconvex hierarchy theorem \ref{wise-hierarchy}, it follows that $G$ is virtually special. 

Yet another approach to the cubulation of $G$ when $Q=\Z$ was given by Dufour \cite{dufour} where the cross-cut surfaces of Cooper-Long-Reid \cite{clr} were used to manufacture enough codimension one quasiconvex subgroups. Dufour's approach essentially used the fact that the cross-cut surfaces of 
\cite{clr}  can be isotoped to be transverse to the suspension flow in $M$ and are hence incompressible. Replacing $H$ by a free group in Question \ref{splqn}, Hagen and Wise \cite{hagenwise1,hagenwise2} prove cubulability of hyperbolic $G$ with $Q=\Z$. Their proof again uses a replacement of the suspension flow (a semi-flow). {This was extended to hyperbolic hyperbolic-by-cyclic groups in very recent work by Dahmani-Krishna-Mutanguha \cite{dkm}.}

Thus, in the general context of 3-manifolds fibering over the circle with pseudo-Anosov monodromy, there are two methods of proving the existence of codimension one quasiconvex subgroups:

\begin{enumerate}
	\item Work of Cooper-Long-Reid \cite{clr} that is special to fibered manifolds.
	\item The general theorem of Kahn-Markovic \cite{km-surf} for hyperbolic 3-manifolds. This uses real hyperbolicity of $M$ in an essential way.
\end{enumerate}

We do not know an answer to the following in this generality:

\begin{qn}
	Let $G$ be as in Question \ref{splqn}. Does $G$ have a quasiconvex codimension one subgroup? 
\end{qn}

When $Q$ has rank greater than one, we do not have an analog of Thurston's {hyperbolization theorem for atoroidal fibered 3-manifolds} (or the geometrization theorem of Perelman) and hence we do not have an  analog of the  Kahn-Markovic theorem providing sufficiently many codimension one quasiconvex subgroups. We are thus forced to use softer techniques from the coarse geometry of hyperbolic groups, e.g.\ the Bestvina-Feighn combination theorem \cite{BF} giving necessary and sufficient conditions for the Gromov-hyperbolicity of $G$. 
{ A particular case of Question \ref{genlqn} arises when $A, B$ are fundamental groups of 3-manifolds fibering over the circle such that the fiber group is $C$.
The fiber group $C$ is clearly not quasiconvex.	We mention as an aside that the {turning} construction of Hsu-Wise \cite{hw-inv} requires quasiconvexity of the amalgamating subgroup. } We pose a general problem in this context
{ seeking a combination theorem for quasiconvex codimension one subgroups when the amalgamating subgroups are not necessarily quasiconvex}. This would help in addressing Question \ref{genlqn}:

\begin{qn}\label{qn-comb}
	Let $G$ be a finite graph $\GG$ of hyperbolic groups (e.g. $G=A\ast_CB$ or $G=A\ast_C$) so that the vertex groups are virtually special cubulable and $G$ is hyperbolic. {We do not assume that the edge groups are quasiconvex in $G$.
Find sufficient conditions on a finite family $\HH$ of quasiconvex codimension one subgroups  of vertex and edge groups of $G$, such that the subgroup $H$ of $G$ generated by $H_v \in \HH$ is quasiconvex and codimension one.} A case of particular interest is
 $G$ as in Question \ref{splqn}.
\end{qn}

{A basic test case of Question \ref{qn-comb} can be formulated as follows. Let $G=G_1 \ast_{G_{12}} G_2$. Let $H_i < G_i$, $i=1,2$, and $H_{12} = H_1 \cap H_2$. Assume further that $H =\langle H_1, H_2\rangle$
	is given by $H=H_1 \ast_{H_{12}} H_2$. Given that $H_i < G_i$
	is  quasiconvex and codimension one for $i=1,2$ and similarly for 
	 $H_{12} < {G_{12}}$, when is  $H$ 
 quasiconvex and codimension one in $G$? An answer to  Question \ref{qn-comb} would allow us
to construct quasiconvex and codimension one subgroups in $G$ and thus take a first step towards answering Question \ref{genlqn}.}

The boundary of  a $G$ as in Question \ref{splqn} is somewhat intractable. Abstractly, it may be regarded as a quotient of the circle (identified with the boundary $\partial H$ of $H=\pi_1(S)$) under the Cannon-Thurston map \cite{mitra-endlam,mahan-rafi} that collapses a Cantor set's worth of ending laminations, where the Cantor set is identified with the boundary $\partial Q$ of the quotient free group $Q$. It thus seems difficult to apply Bergeron-Wise's criterion for cubulability \cite{bw}. Further there is no natural replacement for the suspension flow: a flowline would have to be replaced by a tree and transversality breaks down, preempting any straightforward generalization of the techniques of \cite{dufour,hagenwise1,hagenwise2}. 

This forces us to find sufficient conditions guaranteeing the existence of a quasiconvex hierarchy. The replacement of embedded incompressible surfaces in our context are tracks. Our main theorem \ref{maineg0} gives sufficient conditions to ensure the existence of embedded tracks.
{To prove theorem \ref{maineg0}, we draw
	liberally from the model geometries that went into the proofs of the ending lamination theorem and the existence of Cannon-Thurston maps for Kleinian surface groups \cite{minsky-jams,minsky-cdm,minsky-elc1,minsky-elc2,mahan-split} as also the  hierarchy machinery of subsurface projections in the mapping class group \cite{masur-minsky,masur-minsky2}. These techniques were originally developed  to address problems of infinite covolume surface group representations  into $\pslc$ (see \cite[Problems 6-14]{thurston-bams} for instance).}
 In the interests of readability, the material that goes into proving {the existence of} geometric models is treated in the companion paper \cite{mahan-hyp}.

\subsection{Statement of Results} {In the special case of a hyperbolic 3-manifold $M$ fibering over the circle, our techniques yield  monodromies $\Phi$ and 
a fairly explicit construction  of embedded quasiconvex surfaces in the associated $M$ that cannot in general be made transverse to the suspension flow corresponding to $\Phi$} (see Remark \ref{rmk-euler}).
Thus these surfaces need not
realize the Thurston norm in their homology class (as in \cite{clr})
{and so incompressibility must be proven by different methods}.

The curve graph $\ccs$ of a closed surface of genus at least two \cite{masur-minsky} is a graph whose vertices are given by isotopy classes of simple closed curves, and whose edges are given by distinct isotopy classes of simple closed curves that can be realized disjointly on $S$.
 An element $\psi \in MCG(S)$ is said to be a pseudo-Anosov homeomorphism in the complement of
a simple closed curve $\alpha$ if it fixes $\alpha$,   restricts to a pseudo-Anosov on $(S\setminus \alpha)$, and further, the powers $\psi^n$  are renormalized by Dehn twists $\tw_\alpha^{k_n}$ so that the renormalized powers 
$\overline{\psi^n}:=\tw_\alpha^{k_n} \, \circ \, \psi^n$ do not twist about $\alpha$ (see Definitions \ref{def-notwist} and \ref{def-renpa} for details). The action of such a $\psi$ on the curve complex $\ccs$ fixes the vertex $\alpha$. Thus renormalized large powers of $\psi$ may be thought of as ``large rotations'' about $\alpha$ in $\ccs$. {Following \cite{masur-minsky2}, we say that a sequence of simple closed curves 
$\cdots, \sigma_{i-1}, \sigma_i, \sigma_{i+1}, \cdots$ on $S$ is a
{\bf tight geodesic} in $\ccs$ if
\begin{enumerate}
\item $\cdots, \sigma_{i-1}, \sigma_i, \sigma_{i+1}, \cdots$ is a geodesic in $\ccs$,
\item for all $i$, $\sigma_{i-1}, \sigma_{i+1}$ fill $S \setminus \sigma_i$.
\end{enumerate}}
A sequence of simple closed curves on $S$ on a tight geodesic of  length
at least one in $\ccs$ is called a {\bf tight sequence}.
Informally, Proposition \ref{qchierarchyin3m} below says: The composition
of large powers of pseudo-Anosovs in the respective complements of a pair of disjoint  homologous curves gives, via the mapping torus construction, a 3-manifold fibering over the circle with an embedded geometrically finite surface. Alternately, the composition
of large rotations (cf.\ \cite[Chapter 5]{dgo}) about a pair of homologous curves gives the monodromy of a 3-manifold fibering over the circle with an embedded geometrically finite surface.

{
	In Proposition \ref{qchierarchyin3m} below, and in the rest of this paper, whenever 
	we refer to a sequence of \emph{homologous} simple closed  curves, we shall mean
	a sequence of simple closed  curves that are homologous \emph{up to a choice of orientation.}}
Then
(see Proposition \ref{qchierin3m}, {Remark~\ref{rmk-qchierin3m}} and Remark \ref{qchierin3msep}):
\begin{prop}\label{qchierarchyin3m}
	Let $v_1,v_2 \in \ccs$ be a pair of adjacent vertices in $\ccs$ such that $v_1, v_2$ correspond to homologous simple closed non-separating curves $\sigma_1, \sigma_2$. For $i = 1, 2$, let $\psi_i:S \to S$ be  a pseudo-Anosov  homeomorphism in the complement of $\sigma_i$. 
{	Then there exists $p_0 \geq 1 $ (depending on $S$, $\psi_i$), such that the following holds.}\\
	Let
	$\Phi(p_1, p_2)=\overline{\psi_2^{p_2}}.   \overline{\psi_1^{p_1}}$ and let $M(p_1,p_2)$ be the 3-manifold fibering over the circle with fiber $S$ and monodromy $\Phi(p_1, p_2)$. 
{	For all $p_1, p_2$ with $|p_1|, |p_2|\geq p_0$,} $M(p_1,p_2)$  
 admits an embedded incompressible geometrically finite surface.  
\end{prop}

As a consequence, {we obtain the following.}
{\begin{cor}\label{cor-qchier}
Let 	$v_1,v_2 \in \ccs$, $\psi_i:S \to S$, and  $M(p_1, p_2)$ be  
as in Proposition~\ref{qchierarchyin3m}. 	Then there exists $p_0 \geq 1 $ (depending on $S$, $\psi_i$), such that for all $p_1, p_2$ with $|p_1|, |p_2|\geq p_0$, $M(p_1, p_2)$ admits a quasiconvex hierarchy (in the terminology of Wise's Theorem \ref{wise-hierarchy}).
\end{cor} }
Of course, Agol's Theorem \ref{spl} shows that the manifolds $M(p_1, p_2)$ in Proposition \ref{qchierarchyin3m} are virtually special cubulable and hence a finite cover of any such $M(p_1, p_2)$ does admit a quasiconvex hierarchy. When the first Betti number $b_1(M(p_1, p_2))$ is at least $2$,  $M(p_1, p_2)$ itself admits an embedded geometrically finite surface by an    argument involving the Thurston norm \cite{thurston-mams}. However,  Proposition \ref{qchierarchyin3m} furnishes a new sufficient condition on the monodromy $\Phi(p_1, p_2)$ to guarantee
the existence of an embedded incompressible geometrically finite surface in the 3-manifold $M(p_1, p_2)$ even when $b_1(M(p_1, p_2))=1$.  When $b_1(M(p_1, p_2))=1$, the surfaces we construct are necessarily separating. 
{We also mention work of Brock-Dunfield \cite{brock-dunfeld} and Sisto \cite{sisto-jta}  in a similar spirit that uses model geometries of degenerate ends to extract information about closed manifolds.}

Proposition \ref{qchierarchyin3m} becomes an ingredient for the next theorem which provides some of the main new examples of this paper (see Theorems \ref{maintech} and \ref{maineg1}).  We first provide a  statement  using the terminology of hierarchies \cite{masur-minsky2} before giving an alternate description.
{(Theorem \ref{maineg0} below follows from 
	Theorems \ref{maintech} and \ref{eiqimpliescube}: {Theorem \ref{maintech} proves the existence of an EIQ track in the sense of Definition~\ref{def-eiq} and Theorem~\ref{eiqimpliescube} uses this to furnish a quasiconvex hierarchy.})}

\begin{theorem}\label{maineg0}
	Let 
	$Q$ be a  subgroup of $MCG(S)$ isomorphic to the free group $F_n$, and $\sigma$  a non-separating simple closed curve on $S$ satisfying the following conditions:
	\begin{enumerate}
		\item {\it Tight tree:} The orbit map $q \to q.\sigma$, $q \in Q$ extends to  a $Q-$equivariant  isometric embedding $i$ of  a tree $T_Q$ into $\CC(S)$ such that $T_Q/Q$ is a finite graph;
		\item {\it Large links:}  $d_{\CC(S\setminus i(v))} (i(v_1), i(v_2)) \gg 1$,  for any vertex $v$ of $T$ and distinct neighbors $v_1, v_2$ of $v$ in $T$.
		\item  {\it Homologous curves:}   All vertices of $i(T)$ are homologous to each other.
		\item {\it subordinate hierarchy paths small:} Hierarchy geodesics  subordinate to the geodesics in  $\CC(S\setminus i(v))$ (Item (2) above)  are uniformly bounded.
	\end{enumerate}
	Then $Q$ is convex cocompact.
	For $$1 \to \pi_1(S) \to G \to Q \to 1$$ the induced exact sequence of hyperbolic groups, $G$ admits a quasiconvex hierarchy and hence is cubulable and virtually special.
\end{theorem}

{
In   Theorem \ref{maineg0} above, the conclusion that $Q$ is convex cocompact follows just from hypothesis (1) and is due to Kent-Leininger \cite{kl-coco} and Hamenstadt \cite{hamen}.}
We now describe fairly explicitly a way of constructing groups $Q$ (and hence $G$) as in {Theorem \ref{maineg0}}.  We shall use the notion of subsurface projections from \cite{masur-minsky2} (the relevant material is summarized in \cite[Section 2.2]{mahan-hyp}). We also restrict ourselves here to 
the case $n=2$ for ease of exposition. Let $\gamma_1, \gamma_2$ be two
tight geodesics of homologous non-separating curves stabilized by pseudo-Anosov homeomorphisms $\Phi_1, \Phi_2$ constructed as in Proposition \ref{qchierarchyin3m}.
Further assume without loss of generality that both $\gamma_1, \gamma_2$
pass through a common vertex $v$ (this can be arranged after conjugating $\Phi_2$ by a suitable element of $MCG(S)$ for instance). Note that 
$\gamma_1= \gamma_2$ and $\Phi_1 = \Phi_2$ are allowed in the construction below. Thus, the data of a single 3-manifold constructed as in Proposition \ref{qchierarchyin3m} above allows the construction below to go through. 
For $j=1,2$, we denote the vertex sequence of $\gamma_j$ by $v_{ij}$, $i \in \Z$.

{
Let $MCG(S,v)$ denote the subgroup of $MCG(S)$ stabilizing the curve $v$ on $S$ and preserving its co-orientation in $S$. Then, after choosing a representative curve for $v$, each element of  $MCG(S,v)$ has a representative fixing it pointwise. We assume for now that such choices have been made.}
We think of the elements
$\Psi \in MCG(S,v)$ as  rotations about $v$ in the curve graph $\ccs$ (cf.\ \cite[Chapter 5]{dgo}). { Given $L, R >0$, an element $\Psi \in MCG(S,v)$ is said to be 
an $(L,R)-${\bf large rotation}  about $v$ sending $\gamma_1$ to $ \gamma_2$ (see Definition \ref{def-largerot} where a more general definition is given) if $\Psi (\gamma_1)=\gamma_2$ and for any  distinct $u,z, w \in \gamma_1\cup \gamma_2  $,
with $d(u,z) \leq 2, d(w,z)\leq 2$, we have the following:
\begin{enumerate}
	\item  $d_{\CC(S\setminus z)} (u,w ) \geq L,$
	\item further subsurface projections (including annular projections) of any  geodesic in $\CC(S\setminus z)$ joining $u,w$  are at most $R$.
\end{enumerate} }

We are now in a position to state a special case of one of the main theorems (see Theorem \ref{maineg1}) of the paper. Informally, Theorem \ref{maineg} below says that a pair of pseudo-Anosov homeomorphisms
constructed as in Proposition \ref{qchierarchyin3m} having axes passing through a common vertex $v \in \ccs$ generate a convex cocompact free subgroup of $MCG(S)$ such that the resulting surface-by-free group is virtually special cubulable so long as the `angle' between the axes at $v$ is large (Theorem \ref{maineg} below specializes Theorem \ref{maineg1} to the case $n=2$, see Proposition~\ref{prop-lrrotexist}).
More precisely,

\begin{theorem}\label{maineg} There exist $L, R > 0$ such that if
	\begin{enumerate}
	\item  $\gamma_1, \gamma_2$ are two
	$L-$tight $R-$thick geodesics of homologous non-separating curves stabilized by pseudo-Anosov homeomorphisms $\Phi_1, \Phi_2$ constructed as in Proposition \ref{qchierarchyin3m},
	\item $\gamma_1, \gamma_2$
	pass through a common vertex $v$, 
	 {\item  $\Psi$  is  an $(L,R)-$ large rotation  about $v$ taking $\gamma_1$ to $\gamma_2$
	\item the fundamental domain of the $\Phi_i$ action on $\gamma_i$ has length at least 3,
	\item $\Psi \Phi_1 \Psi^{-1}=\Phi_2$}
	\end{enumerate}   then 
	the group $Q$ generated by $\Phi_1,  \Phi_2 $ is a free convex cocompact subgroup of rank 2 in $MCG(S)$. 
	For $$1 \to \pi_1(S) \to G \to Q \to 1$$ the induced exact sequence of hyperbolic groups, $G$ admits a quasiconvex hierarchy and hence is cubulable and virtually  special.
\end{theorem}

The hypothesis on the non-separating nature of the curves in  $\gamma_1, \gamma_2$ can be relaxed; however the modified version of Theorem \ref{maineg} becomes more technically involved to state. We refer the reader to Theorem \ref{maintech2} for the analog in this more general situation.

\subsection{Scheme of the paper}
The study of  embedded incompressible surfaces has a long history in the study of 3-manifolds \cite{hempel-book}. Tracks (see \cite{sageev-th, wise-cbms} for instance) are the natural generalization of these to arbitrary cell complexes and form the background and starting point of this paper. 
 Since our main motivation is to cubulate hyperbolic groups, we are interested primarily in quasiconvex tracks leading us naturally to the study of EIQ 
(essential incompressible quasiconvex)  tracks. An EIQ track in a 3-manifold is simply an embedded incompressible geometrically finite surface.
{ The main content of
Appendix~\ref{sec-app}, where tracks are dealt with, is a reduction theorem, Theorem \ref{eiqimpliescube}. It says that if a hyperbolic bundle $M$ over a finite graph $\GG$ with fiber a closed surface $S$ admits an EIQ track, then it admits a quasiconvex hierarchy in the sense of Theorem \ref{wise-hierarchy}. Theorem~\ref{eiqimpliescube} thus reduces the problem of cubulating {$\pi_1(M)$} to one of finding an EIQ track.
This part of the paper, written jointly with Jason Manning and Michah Sageev, forms the appendix  to the paper. The \emph{techniques}  
	of Appendix~\ref{sec-app} are largely orthogonal to those used in
	the main body of the paper.
The rest of the paper crucially uses the output, i.e.\ Theorem \ref{eiqimpliescube},
and proceeds to construct an EIQ track using  techniques that come largely from the model manifold technology of \cite{minsky-elc1}.}

 Section \ref{sec-tt} provides the background, where we start by recalling the notions of
 graphs of spaces, surface bundles over graphs, and metric bundles.
We also state Theorem~\ref{eiqimpliescube} from the Appendix~\ref{sec-app} for convenience of the reader.
 We then recall some of the essential features of the geometry of a hyperbolic bundle $M$ over a finite graph $\GG$ from \cite{mahan-hyp}.  An essential tool that is recalled in Section \ref{sec-tightree} is the notion of a tight tree $T$ of non-separating curves in $\ccs$ generalizing the notion of a tight geodesic. We also equip  $M_T = S \times T$ with a metric $\dw$  using the model geometry of doubly degenerate hyperbolic 3-manifolds \cite{minsky-elc1,minsky-elc2,mahan-hyp}. Further, the construction of an auxiliary `partially electrified' pseudo-metric $\dt$  on $M_T$ is also recalled from \cite{mahan-hyp}.

{
Section \ref{sec-geolts} deals with an essential technical tool of this paper: geometric limits. The section culminates in a  quasiconvexity result (Lemma \ref{lem-cutsurfacegf}) for certain subsurfaces of the fiber. This is the main tool for proving a uniform quasiconvexity result in Section \ref{sec-treads}.}

The tight tree  $T$ is then used in Section \ref{sec-stairstep} to construct a track $\ttt$ in $M_T$. The track we construct
 is of a special kind--a  `stairstep'. This is fairly easy to describe in $S \times [0,n]$: it consists of essential  horizontal subsurfaces  called {\it treads}, denoted $\trvw$, in $S\times \{ i \}$ with boundary consisting of  curves $v\times \{ i \}, w\times \{ i \}$ connected together by vertical annuli  $v\times [i, i+1], w\times [i-1, i]$ called {\it risers} corresponding to the curves $v$ and $w$. The sequence of simple closed curves thus obtained on $S$ is required to be a tight geodesic in the curve graph $\ccs$. Section \ref{sec-stairstep} concludes with the statement of the main technical theorem \ref{maintech} 
of the paper, whose proof is deferred to  the later sections.

Section \ref{sec-construction} then applies Theorem \ref{maintech} to construct the main examples of the paper (Theorems \ref{maineg0} and \ref{maineg}) already described.

 The next two sections prove that the track $\ttt$ is $\pi_1$--injective in $M_T$ and that any
elevation $\tttt$ to the universal cover $\widetilde{M_T}$ is quasiconvex with respect to either the $\dw$ metric or the $\dt$ pseudo-metric.
Gromov-hyperbolicity (with constant of hyperbolicity depending only on genus of $S$, the maximal valence of $T$, and a parameter $R$ as in Theorem \ref{maineg})
of $\tmtdt$ was established in \cite{mahan-hyp}. Quasiconvexity of treads (with constant having the same dependence above) is established in Section \ref{sec-treads} using the structure of Cannon-Thurston maps.    In Section \ref{sec-qctrack},  the treads are pieced together via risers
using a version of the local-to-global principle for quasigeodesics in $\delta-$hyperbolic spaces to complete the proof of Theorem \ref{maintech}. Section \ref{sec-genlzns} generalizes the main theorem by
allowing 
tight trees of homologous separating curves.

\section{Graphs of spaces, tight trees and models}\label{sec-tt}

\subsection{Graphs of Spaces, bundles, tracks}\label{sec-graphs-of-spaces}

For all of the discussion below, spaces are assumed to be connected and path connected.
A \defstyle{graph} $\GG$ we take to be a tuple $(V,E,i,t)$ where $V$ and $E$ are sets (the \defstyle{vertex set} and \defstyle{edge set}, respectively), and $i: E\to V$ and $t: E\to V$ give the \defstyle{initial} and \defstyle{terminal} vertex of each edge.  (Strictly speaking this is a \emph{directed} graph.)

A graph of spaces is constructed from the following data \cite{scott-wall}:

\begin{itemize}
	\item A connected graph $\GG=(V,E,i,t)$;
	\item a  \defstyle{vertex space} $X_v$ for each $v\in V$ and an \defstyle{edge  space} $X_e$ for each $e\in E$; and
	\item continuous maps $f_e^-: X_e\to X_{i(e)}$ and $f_e^+:X_e\to X_{t(e)}$ for each $e\in E$.
\end{itemize}

Given this data, we construct a space
\[ X_\GG = \bigcup_v X_v \cup \bigcup_e X_e\times I/\sim, \]
where $(x,0)\sim f_e^-(x)$ and $(x,1)\sim f_e^+(x)$ for each $x\in X_e$, $e\in E$.  We say $X_\GG$ is a \defstyle{graph of spaces}.  We say a homeomorphism $X\to X_\GG$ is a \defstyle{graph of spaces structure on $X$}.

\begin{eg}
	The trivial example in which every vertex and edge space is a single point yields a $1$--complex, the geometric realization of $\GG$.  We abuse notation in the sequel and refer to this $1$--complex also as $\GG$.
\end{eg}

When the maps $f_e^\pm$ are all $\pi_1$-injective then the fundamental group of the space $X_\GG$ inherits a graph of groups structure.  In particular, when the graph has a single edge, we obtain a decomposition of $\pi_1(X_\GG)$ as a free product with amalgamation or as an HNN-extension, depending on whether or not the graph is an edge or a loop. 
(See \cite{scott-wall} for more on graphs of spaces and graphs of groups.)

Conversely, if we are given a space $X$ and a subspace $Y\subset X$ such that $Y$ has a closed neighborhood $N$ homeomorphic to $Y\times [0,1]$, then we obtain a graph of spaces structure for $X$, namely the components of $Y$ as the edge spaces and the components of $X\setminus \mathring{N}$ as the vertex spaces, where $\mathring{N}\cong Y\times (0,1)$. \\

\noindent {\bf Surface bundles over a graph}\\
In this paper the main objects of study will be surface bundles over graphs.  Let $\GG$ be a connected graph, thought of as a $1$--complex, and consider a bundle $E\to \GG$ with fiber $S$ a surface.  Then it is easy to see that $E$ has the structure of a graph of spaces (with graph $\GG$) where every edge and vertex space is homeomorphic to $S$ and every edge-to-vertex map is a homeomorphism.

In particular, {for $\GG$ finite} we can describe the fundamental group of $E$ as follows.  Let {$\TT\subset \GG$ be a (necessarily finite)} maximal tree.  In the graph of spaces structure coming from the bundle, we may assume that for any edge {$e\subset \TT$}, the gluing maps $f_e^\pm$ are the identity map on $S$.  Let $e_1,\ldots,e_n$ be the edges in {$\GG\setminus \TT$}.  For each $i\in \{1,\ldots,n\}$, let $f_i = f_{e_i}^+$, and let $\phi_i = (f_i)_*: \pi_1S\to \pi_1S$.  We can describe $\pi_1E$ as a multiple HNN extension  {	\[ \pi_1E \cong \langle \pi_1S,t_1,\ldots,t_n\mid t_i^{-1} st_i=\phi_i(s),\forall s\in \pi_1S, i\in\{1,\ldots,n\}\rangle.\]}

We are particularly interested in the case that  {$G=\pi_1E$} is Gromov hyperbolic.  It is a theorem of Farb and Mosher \cite{farb-coco} that such groups exist.
Our approach to cubulating such $\pi_1E$ will be to find a quasiconvex subgroup of $G$ over which $G$ splits as an amalgam or HNN extension.
The fiber group $\pi_1(S)$ is normal in $\pi_1E$, in particular not quasiconvex, so we will need to look at other ways of expressing $E$ as a graph of spaces.\\

\noindent {\bf Tracks and a reduction theorem}\\ We refer the reader to
 Appendix ~\ref{sec-app} for background material on tracks, especially the notion of \emph{essential, incompressible quasiconvex} (EIQ) tracks (Definition~\ref{def-eiq}). We also state below, for convenience
of the reader, the main output of Appendix ~\ref{sec-app}:

\begin{theorem*}[Theorem~\ref{eiqimpliescube}]
	Let $M$ be a closed surface bundle over a finite graph $\GG$, so that $\pi_1M$ is hyperbolic. Suppose that $M$ contains an EIQ freely indecomposable surface bundle track $\TT$. Then $\pi_1M$ {admits a quasiconvex hierarchy and is therefore} cubulable.
\end{theorem*}

Theorem~\ref{eiqimpliescube} says that in order to cubulate a hyperbolic surface bundle over a graph, it suffices to 
construct an EIQ track.\\

\noindent {\bf Metric surface bundles}\\
If $M$ is a  surface bundle over a graph $\GG$ with fiber $S$, then 
the cover of $M$ corresponding to $\pi_1(S)$ is again  a   surface bundle $M_T$ over a tree, $T$, where $T=\til \GG$ is the universal cover of $\GG$. 

We shall also have need to equip such surface bundles over graphs with a metric structure.  Here, the underlying graph $\GG$ or tree $T$ will be a metric tree.

\begin{defn}\label{def-mbdl} Let $(X,d)$ be a path-metric space equipped with
	the structure of a bundle  $P:X \to \GG$  over a graph $\GG$  with fiber a surface $S$ (here we allow $\GG$ to be a tree $T$). Then  $P:X \to \GG$  will be called a \defstyle{metric surface bundle} if 
	\begin{enumerate}
		\item  There exists a metric $h$ on $S$ and $\lambda\geq 1$ such that for all $x \in \GG$,
		$P^{-1}(x) =S_x$ equipped with the induced path-metric induced from $(X,d)$ is $\lambda-$bi-Lipschitz to $(S,h)$.
		\item  Further,  for any isometrically embedded interval $I \subset \GG$, with
		$I=[0,1]$,
		$P^{-1}(I)$ is $\lambda-$bi-Lipschitz to $(S,h) \times [0,1]$ by a $\lambda-$bi-Lipschitz fiber-preserving homeomorphism that is $\lambda-$bi-Lipschitz on the fibers.
	\end{enumerate}
\end{defn}

Definition \ref{def-mbdl} above is a special case of the more general notion of metric bundles introduced in \cite{mahan-sardar}.\\

\subsection{Tight trees of non-separating curves}\label{sec-tightree}
 Let $M \to \GG$ be a surface bundle over a  graph  as in Section \ref{sec-graphs-of-spaces} where the edge and vertex spaces are all homeomorphic to a closed surface $S$. Then the  cover  of $M$ corresponding to $\pi_1(S)$ is again a surface bundle over a  graph with base graph the universal cover $\til \GG$ of $\GG$. In what follows in this section, we shall  denote the tree $\til \GG$  by $T$. {Note that unlike the tree $\TT$ in Section \ref{sec-graphs-of-spaces}, $T$ is infinite.}

The \defstyle{curve graph} $\CC(S)$ of an orientable finite-type surface $S$ is a graph whose vertices consist of free homotopy classes of 
simple closed curves and edges consist of pairs of distinct free homotopy classes of 
simple closed curves that can be realized by curves having minimal number of intersection points (2 for $S_{0,4}$, 1 for $S_{1,1}$ and 0 for all other surfaces of negative Euler characteristic). A fundamental theorem of Masur-Minsky\cite{masur-minsky} asserts that  $\CC(S)$ is Gromov-hyperbolic.
{In fact, Aougab \cite{aougab}, Bowditch \cite{bowditch}, Clay-Rafi-Schleimer \cite{crs},} and Hensel-Przytycki-Webb \cite{slim-unicorns} establish that all curve graphs are uniformly hyperbolic.  The Gromov boundary $\partial \CC(S)$ may be identified with the space   $\EL(S)$ of ending laminations \cite{klarreich-el}.
We shall be interested in surface bundles coming from trees $T$ embedded in $\CC(S)$. 
We will now briefly recall from \cite{mahan-hyp} the construction of a geometric structure on such  surface bundles. The ingredients of this construction are as follows:
\begin{enumerate}
	\item A sufficient condition to ensure an isometric embedding of $T$ into $\CC(S)$.
	\item The construction of an auxiliary metric tree $\but$ from $T$ where each vertex $v$ is replaced by a finite {\it metric} tree $T_v$  called the tree-link of $v$ (see Definition \ref{def-treelink}).
	We refer to $\but$ as the blown-up tree. 
	\item The construction of  a surface bundle  $M_T$  over $\but$. 
	The metric tree $\but$ captures the {\it geometry} of the base space of the bundle $M_T$, while the tree $T$ only captures the  topological features.
	\item An effective construction of a metric on $M_T$  such that the universal cover $\til M_T$ is  $\delta-$hyperbolic, with $\delta$ depending only on some
	properties of $T$ (see Theorems \ref{model-str} and \ref{mainprel} below for precise statements).
\end{enumerate}

 We refer the reader to \cite{masur-minsky2} for   details on subsurface projections  (the necessary material is summarized in  \cite[Section 2.1]{mahan-hyp}).

\begin{defn}\label{def-tighttree}   \cite[Section 2.2]{mahan-hyp} For any $L \geq 1$,
	an {\bf $L-$tight tree of non-separating curves} in the curve graph $\CC(S)$ consists of a  simplicial tree $T$ of bounded valence and a simplicial map $i: T \to \CC (S)$ such that for every vertex $v$ of $T$ 
	\begin{enumerate}
	\item {$i(v)$ is non-separating,}
	\item for every pair of distinct vertices $u \neq w$ adjacent to $v$ in $T$,
	$$d_{\CC(S \setminus i(v))} (i(u), i(w)) \geq L.$$
	\end{enumerate}

	An $L-$tight tree of non-separating curves for some $L\geq 3$ will simply be called a tight tree of non-separating curves. Such a tree is called a 
	tight tree of homologous non-separating curves if, further, the curves $\{i(v):v \in T\}$ are homologous {(up to orientation)}.
\end{defn}

{We shall need the following  condition  guaranteeing that tight trees give isometric embeddings.}

\begin{prop}\label{isometrictighttree} \cite[Proposition 2.12]{mahan-hyp}
Let $S$ be a closed surface of genus at least $2$.	There exists $L\geq 3$, such that the following holds.  Let $i: T\to \CC(S)$ define an $L$--tight tree of non-separating curves.  Then $i$ is an isometric embedding.
\end{prop}

Chris Leininger told us a proof of the main technical Lemma that went into a proof of Proposition \ref{isometrictighttree}. A more general version (Proposition \ref{isometrictighttree-sep}) due to Ken Bromberg will be given  later.

\subsection{Topological building blocks from links: non-separating curves} \label{sec-topbb} We recall the structure of building blocks from  \cite[Section 2.3]{mahan-hyp}.
{The {\bf weak hull} of a subset $Y$ of a Gromov-hyperbolic space $(X,d)$ consists of the union of all geodesics in $X$ joining pairs of points in $Y$.}  Let $i: T \to \CC(S)$ be  a tight tree of non-separating curves and let  $v$ be a vertex of $T$. The link of $v$ in $T$  is denoted as $lk(v)$.
Let $S_v = S \setminus i(v)$. Then $i(lk(v))$ consists of a uniformly bounded number (depending only on the maximal valence of $T$) of vertices in $\CC(S_v)$. Hence the weak  hull 
$CH(i(lk(v)) $ of 
$i(lk(v))$ in  $\CC(S_v)$ 
admits a uniform approximating tree $T_v$ { (see
\cite{cdp} Chapter 8 Theorem 1 and \cite[p. 155]{gromov-hyp})}. More precisely,
\begin{lemma}\label{lem-tlink}
	Let  $i: T \to \CC(S)$ be a tight tree of non-separating curves.
	There exists $k\geq 1$, depending only on the valence of vertices of $T$, such that for all $v \in T$ there exists a finite metric tree $(T_v,d_{T_v})$  and a surjective $(1,k)-$quasi-isometry
$$\P_v :CH(i(lk(v))  \to T_v.$$ 
	Further, $\P_v$ maps the vertices of $i(lk(v))$ to the terminal vertices (leaves) of $T_v$ such that for any pair of  vertices $x, y \in i(lk(v)$,
	$$d_{T_v} (\P_v(x),\P_v(y)) \leq d_{\CC (S_v)} (x,y) \leq 
	d_{T_v} (\P_v(x),\P_v(y)) + k.$$
\end{lemma}

{
In the case of interest in this paper, $T$ will be an $L-$tight tree with $L \geq k+1$. This will ensure that  the distance between any pair of terminal vertices/leaves of $T_v$ is at least one.}

\begin{defn} \label{def-treelink}
	The finite tree $T_v$ is called the {\bf tree-link} of $v$.
\end{defn}

\begin{defn}\label{def-topbb} For $i: T \to \CC(S)$  a tight tree of non-separating curves and   $v$  any vertex of $T$,
	the {\bf topological building block corresponding to $v$ } is   \[M_v =S \times T_v.\]
	
	The block
	$M_v$ contains a distinguished subcomplex $i(v)\times  T_v$ denoted as $\RR_v$ which we call the {\bf Margulis riser} in $M_v$ or  the Margulis riser corresponding to $v$. 
\end{defn}
{
	Note that $T$ in Definition~\ref{def-topbb} above need not be regular.}
{Concretely, we shall see in Theorem \ref{model-str} below that
		$M_v$ carries a path metric, where $S$ is equipped with a fixed auxiliary hyperbolic metric and $i(v)$ is realized as a geodesic in this metric.
	For now, the reader may assume that such an auxiliary metric on $S$ has been fixed, and 
        curves are identified with their geodesic representatives in this metric.}

Note that in the definition of the  Margulis riser,
$i(v)$ is identified with a non-separating simple closed curve on $S$. The Margulis riser will take the place of Margulis tubes in doubly degenerate hyperbolic 3-manifolds. 
See also Definitions \ref{def-stairstep} and \ref{def-tstairstep} below clarifying the use of the word ``riser.''

Let $i: T \to \CC(S)$ be  a tight tree of non-separating curves.
The {\bf blow-up} $\BU (T)$ of $T$ is a {\it metric tree} obtained from $T$ by replacing the $\half-$neighborhood of each $v \in T$ by the tree-link $T_v$ (see \cite[Section 2.3]{mahan-hyp} for a more detailed description).
{The condition $L \geq k+1$ after Lemma \ref{lem-tlink} guarantees that any two terminal vertices/leaves of $\BU (T)$ are at a distance at least one from each other.}

Assembling the topological building blocks $M_v$   according to the combinatorics of $\BU (T)$, we get the following: 

\begin{defn}\label{def-topmodeltree}
	The {\bf topological model corresponding to a tight tree $T$ } of non-separating curves is \[M_T =S \times \BU(T).\]
	$\Pi: M_T  \to \BU(T)$ will denote the natural projection giving $M_T$ the structure of a  surface bundle over the tree $\but$. 
\end{defn}

For every $v$, the tree-link $T_v$ occurs as a natural subtree of $\BU(T)$ and $M_v$ occurs naturally as the induced bundle $\Pi: M_v\to T_v$ after identifying  $\Pi^{-1}(T_v)$ with $M_v$.  The intersection of the tree-links $T_v, T_w \subset \BU(T)$ of adjacent vertices $v,w$ of $T$ will be called a {\bf mid-point vertex} $vw$. The pre-image $\Pi^{-1}(vw)$  of  a mid-point vertex will be denoted as  $S_{vw}$ and referred to  as a {\bf mid-surface}.

\subsection{Model geometry}\label{sec-bb} We now recall from \cite[Section 3 ]{mahan-hyp} the essential aspects of the geometry of $M_T$. To do this we need an extra hypothesis.

\begin{defn}\label{def-rthick}
	An $L-$tight tree $i: T\to  \CC (S)$ of non-separating curves is said to be \defstyle{$R-$thick} if
	for any vertices $u,v, w$ of $T$ and any proper essential subsurface $W$ of $S \setminus  i(v)$ (including essential annuli), \[ d_W (i(u), i(w)) \leq R,\] where $d_W(\cdot\ , \cdot)$ denotes distance between subsurface projections onto $W$.
\end{defn}

{For the rest of this section, $T$ will refer to an $L-$tight $R-$thick tree.}

\begin{rmk}\label{rmk-local}
	The condition on $R-$thickness is really a local condition. By the Bounded Geodesic Image Theorem \cite{masur-minsky2}, it follows that 
	there exists a universal constant $M$ such that if {$d_W (i(u), i(w))\leq R-M$ whenever $u, w$ are within distance 2 from  $v$, then the conclusion of Definition \ref{def-rthick} holds for all triples $u,v, w$.}
\end{rmk}

Here, as elsewhere in this paper, a {\bf Margulis tube} in a hyperbolic 3-manifold $N$ will refer to a maximal solid torus $\T \subset N$, with $\inj_x \leq \ep_M$ for all $x \in \T$, where $\ep_M$ is a Margulis constant for $\Hyp^3$ fixed  for the rest of the paper. In particular, all Margulis tubes are closed and embedded.
For $l$ a bi-infinite geodesic in $T$, let $l_\pm$ denote the ending laminations given by the ideal end-points of $i(l) \subset \CC(S)$ and let $\VV(l)$ denote the vertices of $T$ occurring along $l$. Let $N_l$ denote the doubly degenerate hyperbolic 3-manifold with ending laminations  $l_\pm$. {Note that $N_l$ is unique by the ending lamination theorem \cite{minsky-elc2}.}
Then $l$ 
gives a bi-infinite geodesic in $\ccs$ which is an $L-$tight $R-$thick tree with underlying space $\R$.

\begin{defn}\label{def-ddsplgeoltightrthick}
	The manifold $N_l$ will be called a doubly degenerate manifold of {\bf special split geometry corresponding to the $L-$tight $R-$thick tree $l$}.
\end{defn}
The reason for the terminology in Definition \ref{def-ddsplgeoltightrthick} will be explained in Proposition \ref{prop-splsplit}.
For $L$ large enough, if $T$ is $L-$tight, each vertex $v\in \VV(l)$ gives a Margulis tube in $N_l$ \cite[Lemma 3.7]{mahan-hyp}. 
Let $\T_v$ denote the Margulis tube in $N_l$ corresponding to $i(v)$ and $N_l^0 = N_l \setminus \bigcup_{v \in \VV(l)} \T_v$.
Let $\BU(l)$ denote the bi-infinite geodesic in $\but$ after blowing up $l$ in $T$. 
Also let $M_l$ denote the bundle over $\BU(l)$ induced from $\Pi: M_T \to \but$. {Note that for $v \neq w \in \VV(l)$, $\RR_v \cap \RR_w = \emptyset$.}
Let $M_l^0 = M_l \setminus \bigcup_{v \in \VV(l)} \RR_v$ denote the complement of the risers in $M_l$. 

\begin{theorem}{\cite[Theorem 3.35]{mahan-hyp}} \label{model-str} Given $R \geq 0$ and $V_0 \in \natls$, there exist $K \geq 1$ such that for an $L-$tight $R-$thick tree with $L \geq 3$ and valence bounded by $V_0$,  there exists a  metric $\dw$ on $M_T$ such that $\Pi: M_T \to \but$ is a  metric bundle of surfaces (cf.\ Definition \ref{def-mbdl})  over the metric tree $\but$ satisfying the following:
	\begin{enumerate}
		\item The induced metric for every Margulis riser $\RR_v$ is the metric product $S^1_e \times T_v$, where $S^1_e$ is a unit circle. 
		\item  For any bi-infinite geodesic $l$ in $T$, $N_l^0$ and $M_l^0$ are $K-$bi-Lipschitz homeomorphic by a homeomorphism that extends to their path-metric completions.
		\item Further, if there exists a subgroup $Q$ of $MCG(S)$ acting  geometrically on $i(T)$, then this action can be lifted to an isometric fiber-preserving isometric action of $Q$ on $(M_T,\dw)$.
	\end{enumerate}
The  bi-Lipschitz constant  (in Definition \ref{def-mbdl}) of the metric bundle $\Pi: M_T \to \but$ is also at most $K$.
\end{theorem}  

{The interested reader is referred to \cite[Definition 3.15]{mahan-hyp} and \cite[p. 1198]{mahan-hyp} for further details about the metric $\dw$.
We shall return to it in Section \ref{sec-ghlts}.}
Henceforth we shall assume that a bi-Lipschitz homeomorphism as in Theorem \ref{model-str} between 
$N_l^0$ and $M_l^0$ has been fixed.
Theorem \ref{model-str} establishes a bijective correspondence between the risers $R_v \cap M_l$ occurring along the geodesic $l \subset \ccs$ and the Margulis tubes $\T_v$ in $N_l$.
We shall describe  features of the hyperbolic geometry of the special split geometry manifold $N_l$ in Proposition \ref{prop-splsplit} below. For now, we dwell instead on the geometry of $M_l$.
See Figure \ref{schematic} below for a schematic representation of $M_l$ in the special case that $T= \R$ with vertices at $\Z$. Also see Figure \ref{schematic2} for a description of the geometry of individual blocks. The edges of $\BU(l)$ are at least $L-k$ in length, where $k$ depends only on the maximal valence of a vertex of $T$ (see Lemma \ref{lem-tlink}). This is where the parameter $L$ (for an $L-$tight tree) shows up in the model geometry.

\begin{figure}[H]
	
	\centering
	
	\includegraphics[height=6cm]{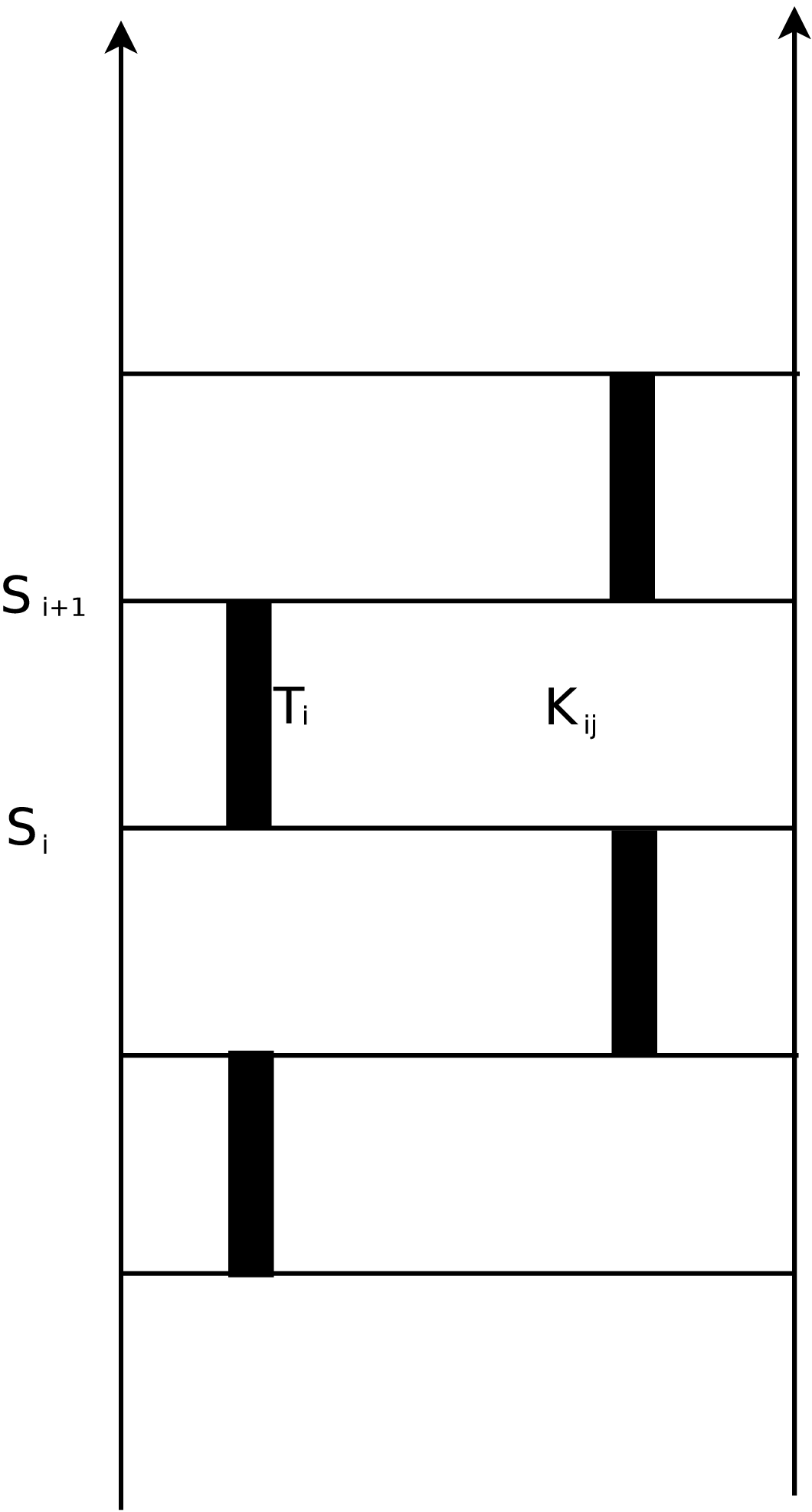}

	\bigskip
	
	\caption{Model geometry for $T$ a line}
	\label{schematic}
\end{figure}

We draw the reader's attention to the fact
that the topological building block  $M_i$ between $S_i$ and $S_{i+1}$ is a topological product (corresponding to the vertex $i$ on the underlying tree $T= l= \R$ with vertices at $\Z$). Further, each such block contains a unique Margulis riser homeomorphic to $S^1 \times I$. Theorem \ref{model-str} above shows that the complement $M_l^0$ of the Margulis risers in $M_l$ and the complement $N_l^0$ of the Margulis tubes in the doubly degenerate hyperbolic manifold $N_l$ are bi-Lipschitz homeomorphic. The place of building blocks $M_i$ in $M_l$ will be taken by   split blocks in $N_l$ (see Definition \ref{def-splitblock} and Proposition \ref{prop-splsplit} below). { For the next Lemma, we extract out the property of $R-$thickness from Definition \ref{def-rthick}.}

{
\begin{defn}\label{def-frthick}
 We shall say that a tree
$T$ equipped with a qi-embedding $i: T \to \CC(S)$ is \defstyle{fully  $R-$thick} if
for any vertices $u, w$ of $T$ and any proper essential subsurface $W$ of $S $ (including essential annuli), \[ d_W (i(u), i(w)) \leq R,\] where $d_W(\cdot\ , \cdot)$ denotes distance between subsurface projections onto $W$.
\end{defn}
}
{
The following was established by Minsky \cite{minsky-bddgeo} in the case of a single pseudo-Anosov and Kent-Leininger \cite{kl-coco} in general.}

\begin{lemma}\label{lem-pA}
	For a surface $S$, let $\phi$ be a pseudo-Anosov homeomorphism. Then there exists $R>0$ such that any  geodesic $\gamma$ in $\ccs$ preserved 
	by $\phi$ is {fully $R-$thick.}
	
	More generally, let $\phi_1, \cdots, \phi_k$ freely generate a free convex cocompact subgroup $Q=F_k$.
		There exists $R$ such that the following holds. Let $\partial Q \subset \partial \CC(S)$ denote
	the associated embedding of  $\partial Q$ in the boundary of the curve complex.
	Let $\HH(\partial Q)$ denote the weak hull of $\partial Q$ in $\CC(S)$.
{Then any bi-infinite geodesic  $\gamma \subset \HH(\partial Q)$ is fully $R-$thick.}
\end{lemma}

\begin{proof}{ 
    Let $\gamma^\pm \subset \partial Q$. Then, \cite[Theorem 7.4]{kl-coco} shows that there exists $D$, independent of $\gamma$ such that subsurface projections $d_W(\gamma^+, \gamma^-) \leq D$ for all proper subsurfaces $W$ of $S$. 
    According to the distance formula of Masur-Minsky (\cite[Lemma 6.2]{masur-minsky2} and \cite[Theorem 6.12]{masur-minsky2}) there exists $D_1$ (independent of $\gamma$),
    such that for any proper essential subsurface
    $W$ of $S$, the length of any geodesic in the hierarchy
    joining  $\gamma^+, \gamma^-$ and supported in $W$ is at most $D+D_1$.}
\end{proof}
{
Kent and Leininger  (see the discussion towards the end of p. 1275 of \cite{kl-coco})  show that for $Q$ convex cocompact,
there exists a unique associated embedding of  $\partial Q$ in the boundary of $\Teich(S)$.
As a consequence of Lemma \ref{lem-pA}, we have:}

\begin{cor}\label{cor-pA}
	{Let $T_Q \subset \ccs$ be a tree preserved by $Q$.
		Then, for any $R>0$,  there exists $\ep >0$ such that if $T_Q$ is  fully $R-$thick, then the following holds.
		Let   $\partial_T Q \subset \partial \Teich(S)$ denote 	the associated embedding of  $\partial Q$ in the boundary of $\Teich(S)$.  Let $\HH_T(\partial Q)$ denote the weak hull of $\partial_T Q$ in $ \Teich(S)$. Then for every $x \in \HH_T(\partial Q)$, the injectivity radius of $S$ equipped with the hyperbolic metric corresponding to $x$ is at least $\ep$.}
	
\end{cor}

\begin{proof}
{
	As noted in \cite[Theorem 2.7]{kl-coco}, Rafi's proof of \cite[Theorem 1.5]{rafi} shows that uniform boundedness of subsurface projections is equivalent to thickness of any bi-infinite Teichm\"uller geodesic $l \subset \HH_T(\partial Q)$.}
\end{proof}

{
 The lower bound on injectivity radius explains the terminology fully $R-$thick. It is worth pointing out to the reader that in Definition \ref{def-rthick}, all curves \emph{other than the
ones on the $L-$tight tree $T$} are required to be thick in the associated model manifold \cite[Section 9]{minsky-elc1}, whereas in 
Definition \ref{def-frthick}, all curves \emph{without exception} are required to be thick in the associated model manifold.}

\subsubsection{Tube electrified metric}\label{sec-tel} It would  be nice if $\tmtdw$ were $\delta-$hyperbolic with a constant $\delta$ independent of $L$. This is simply not true as the Margulis risers $\RR_v$ lift to flat strips of the form $\R \times T_v$ and so the constant $\delta$ depends on the length of the largest isometrically embedded interval in $T_v$. There are a couple of ways to get around it. One way is to use relative hyperbolicity \cite{farb-relhyp}. We shall use an alternate approach using pseudometrics and partial electrification  \cite{mahan-pal} that preserves the $T_v-$direction in $\R \times T_v$.

An auxiliary pseudometric on $M_T$ was defined in \cite{mahan-hyp} as follows.
Equip each Margulis riser $\RR_v= S^1 \times T_v$ with a product pseudometric that is zero on the first factor $S^1$ and agrees with the metric on $T_v$ on the second. 
Let $(S^1 \times T_v, d_0)$ denote the resulting pseudometric.

\begin{defn}\label{tubeel}\cite[p. 39]{mahan-split}
	The {\bf tube-electrified metric} $d_{te}$  on $M_v$ is  the path-pseudometric  defined as follows:\\
		Paths that lie entirely within $M_v \setminus \RR_v$ are assigned their
${d_{weld}}$-length.  Paths that lie  entirely within some $\RR_v$ are assigned their  $d_0$-length. The distance between any two points is now defined to be
the infimum of lengths of paths given as a concatenation of subpaths lying either entirely $\RR_v$ or entirely outside $\RR_v$ except for end-points.
	
	The {\bf tube-electrified metric} $d_{te}$  on $M_T$ is defined to be the path-pseudometric  that
	agrees with $d_{te}$  on $M_v$ for every $v$.
\end{defn}

The lift of the metric $\dw$ (resp. $\dt$) on $(M_T,\dw)$ (resp. $(M_T,\dt)$) to the universal cover $\til M_T$  is also denoted by $\dw$ (resp. $\dt$).
Let $\tmr$ denote the collection of lifts of Margulis risers to $\til M_T$.
The main theorem of \cite{mahan-hyp} states:

\begin{theorem}\label{mainprel} Given $R > 0$, $V_0 \in \natls$ there exists $\delta_0, L_0$ such that the following holds. Let $i:T \to \CC(S)$ be an $L-$tight  $R-$thick tight tree of non-separating curves with $L\geq L_0$ such that the valence of any vertex of $T$ is at most $V_0$. Then
	\begin{enumerate}
	\item	$(\til M_T,\dt)$ is $\delta_0-$hyperbolic.
	\item $\tmtdw$ is strongly hyperbolic relative to the collection 
	$\tmr$.
	\end{enumerate}
\end{theorem}

Note that $\delta_0$ in Theorem \ref{mainprel} depends on $R$ but not on $L$ provided $L$ is large enough.
\subsubsection{Split geometry from a hyperbolic point of view}\label{sec-splitgeo} For the purposes of this sub-subsection, let $l = T$ be a bi-infinite geodesic and $v \in l$ a vertex. Then the tree link $T_{v,l}$ corresponding to the vertex $v$ and tree $l=T$ is an interval of length $$h_v := d_{\CC(S\setminus \{i(v)\})} (i(u),i(w)),$$ where $u, w$ are the vertices on $l$ adjacent to $v$.

We fix a hyperbolic structure $(S,h)$ on $S$ for the rest of the discussion. 
\begin{defn}\label{def-dbddgeo}
	A \defstyle{$ D-$bounded geometry surface} in a hyperbolic manifold $N$ is the image of a $D-$bi-Lipschitz embedding of $(S,h)$ in $N$. \\
	Let $B$, homeomorphic to $S \times [0,1]$ be a hyperbolic manifold \emph{with boundary}, i.e.\ the interior $S \times (0,1)$ of $B$ has a metric of constant curvature $-1$, and $\partial B$ is equipped with the induced Riemannian metric. We say that  $\partial N$ has \defstyle{$ D-$bounded geometry}, if each component of
	 $\partial N$ (equipped with the induced Riemannian metric) is $D-$bi-Lipschitz 
	 homeomorphic to $(S,h)$.
\end{defn}

{For a hyperbolic manifold $N$ without boundary,
 the \defstyle{injectivity radius} at a point $x$ refers to
half the length of a shortest homotopically essential curve passing through $x$. If  $N$ has 
boundary,  the \defstyle{injectivity radius} at a point $x \in \partial N$ refers to
half the length of a shortest homotopically essential curve in $\partial N$ 
passing through $x$, where  $\partial N$ is equipped with the induced Riemannian metric.}

\begin{defn}\label{def-splitblock}{
Let $B$ denote a hyperbolic manifold with boundary, homeomorphic to $S \times [0,1]$. We say that $B$ is a
 \defstyle{split block} with parameters $D \geq 1$ and $\ep > 0$ if 
 \begin{enumerate}
 \item $B$ contains a {\em unique Margulis tube} $\T^0$ with core curve of length at most $\ep$, such that for all
 $x \in B \setminus  ( \T^0)$, the injectivity radius $inj_x (B) \geq \ep$.
 \item 	There exists a 
 solid torus neighborhood $\T$ of $\T^0$ contained in a $D-$neighbor\-hood of $\T^0$, called a {\bf splitting tube}, such that
 $\T \cap S \times \{0\}$ and $\T \cap S \times \{1\}$ are  annuli in $S \times \{0\}, \, S \times \{1\}$ respectively.   Further, the annuli 
 	$\T \cap S \times \{0\}$ and $\T \cap S \times \{1\}$ contain the geodesic representatives of their core-curves in $S \times \{0\}, \, S \times \{1\}$ respectively.
 \item $\partial B$ has $D-$bounded geometry.
 \end{enumerate}}
 \end{defn}

\begin{prop}\cite[Proposition 3.11]{mahan-hyp}\label{prop-splsplit}
	Given $R>0$, there exist {$L_0 \geq 3$}, $D \geq 1$ and $\ep > 0$ such that the following holds.\\ Let $N_l$ be
	a  doubly degenerate  manifold of special split geometry (see Definition \ref{def-ddsplgeoltightrthick}) corresponding to an $L-$tight, $R-$thick tree $l \subset \ccs$ with underlying space $\R$ {and $L \geq L_0$}. Then
	\begin{enumerate}
		\item
		there exists a sequence
		$\{ S_i \}, i \in \Z$  of   disjoint,  embedded, incompressible, $ D-$bounded geometry surfaces called  {\bf split surfaces}
		exiting  the ends $E_\pm$ as $ i \to \pm \infty$ respectively. 
		The surfaces are ordered so that   $i<j$ implies that
		$S_j$ is contained in the component of $N_l\setminus S_i$ representing $E_+$.
		\item {
		Let $B_i$ denote the topological product region  between $S_i$ and $S_{i+1}$. Then $B_i$ is a  split block  with parameters $D, \ep$.}
			\item Let $T_i$ denote the splitting tube of $B_i$. For $i \neq j$, $\T_i, \, \T_j$ are $\ep-$separated from each other.
			\item Let $v_i$ denote the core curve of $T_i$. Then $\cdots. v_{-1}, v_0, v_1, \cdots$ coincides with the vertices of $l$.
	\end{enumerate}
 \end{prop}

{The split surfaces $S_i$ in Proposition \ref{prop-splsplit} are intimately connected to
Theorem \ref{model-str}. Recall that Theorem \ref{model-str} furnishes a metric bundle
$\Pi: M_l \to BU(l)$, where $l$ is as in  Proposition \ref{prop-splsplit}. The split surfaces $S_i$ correspond precisely to $\Pi^{-1} (v_i)$, where $v_i$ is the $i$-th vertex on $BU(l)$.}

\begin{defn}\label{def-splsplit}
	The numbers $D \geq 1$ and $\ep > 0$ shall be called the parameters of special split geometry.
\end{defn}

{Note that in Definition \ref{def-splsplit},  $D$ absorbs two constants into one, and thus serves 2 purposes: 
\begin{enumerate}
\item as a bi-Lipschitz constant (Definition \ref{def-dbddgeo}) for surfaces,
\item as an upper bound on the distance of a split level surface from a Margulis tube
(Definition \ref{def-splitblock}, item (2)).
\end{enumerate}}

Thus the special split geometry manifold $N_l$ can be decomposed as a union $N_l= \bigcup_i B_i$ of split blocks. See Figure \ref{schematic2} below, where a split block with a splitting tube is given. A section of the splitting tube $T_i$ is drawn on the right side. Note the similarity between the block $B_i$ and the region between $S_i$ and $S_{i+1}$ in Figure \ref{schematic}.

\begin{figure}[H]
	
	\centering
	
	\includegraphics[height=8cm]{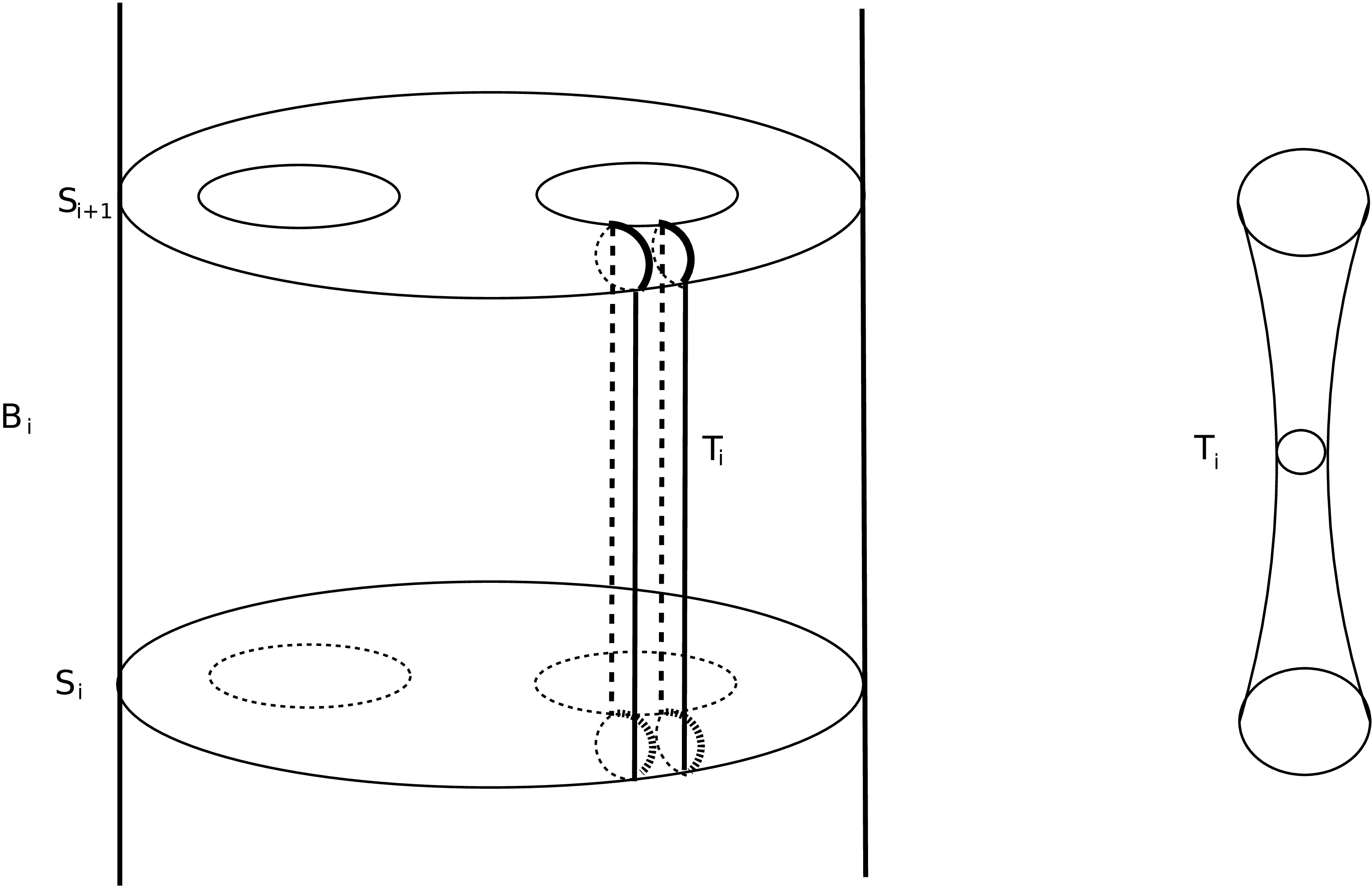}

	\bigskip
	
	\caption{Split block $B_i$ with splitting tube $\T_i$ with core curve of length at most $\ep$.  Split surfaces $S_i, S_{i+1}$ are of $D-$bounded geometry.}
	\label{schematic2}
\end{figure}

\begin{defn}\label{def-ht}
Ordering the vertices of $l$ by $i \in \Z$, if $v$ is the $i-$th vertex on $l$, we denote $h_v$ by $h_i$ and call it the {\bf combinatorial height} of the $i-$th split block. 

{Let $B$ be a split block, with splitting tube $\T$, and boundary components $S \times \{0\}$ and $S \times \{1\}$. 
The distance between $(S \times \{0\} \setminus \T)$ and $(S \times \{1\}\setminus \T)$
in the induced path-metric on $(B\setminus \T)$ will be called the {\bf geometric height} of the  split block $B$. }
\end{defn}

{The following is a consequence of the model manifold in \cite[Theorem 9.1]{minsky-elc1}:
\begin{lemma}\label{lem-heightscompare} Given $D\geq 1, \ep > 0$, there exists $C \geq 1$ such that the following holds. \\
Let $N_l$ denote a doubly degenerate hyperbolic manifold of special split geometry with
parameters $D, \ep$. Then the geometric height of the $i-$th split block lies in
$[\frac{h_v}{C}, {h_v}{C}]$.
\end{lemma}}

{ The lower bound $\ep > 0$ on injectivity radius away from splitting tubes is equivalent to $R-$thickness: it follows 
 from \cite[Theorem 9.1]{minsky-elc1} and $R-$thickness that  no  curves other than those on $l$ are too short. Again, from \cite[Theorem 9.1]{minsky-elc1}, it follows that the geometric height of a splitting tube corresponding to $v$ is comparable to the combinatorial height of $v$. }
 
{ Next, let $B$ be a split block and $\T$ the splitting tube in $B$. Note  that, in the presence of a lower bound on
injectivity radius, the diameter of $B\setminus \T$ is bounded in terms of height. }

\begin{rmk}\label{rmk-ht}{
Since the geometric and combinatorial heights are comparable by Lemma \ref{lem-heightscompare}, we shall, henceforth, simply refer to the height of a split block.}
\end{rmk}

\subsection{A criterion for quasiconvexity} We recall 
from \cite[Section 4.5]{mahan-hyp} a necessary and sufficient condition for promoting  quasiconvexity in vertex spaces $X_v$ to quasiconvexity in the total space  $\tmtdt$.
We refer the reader to \cite{BF,mahan-reeves,gautero} for the relevant background on trees of relatively hyperbolic spaces and the flaring condition. 
Let $P: \tmtdt \to \but$  denote the usual projection map.
For  $w\in \but$, $(X_w,d_w)$ will denote $P^{-1}(w)$ equipped with  the path metric induced by $d_{weld}$.
We recall some of the necessary notions from \cite{mahan-hyp}.  A \emph{$\rho$--qi-section} is a section of a metric bundle which is also a $\rho$--qi-embedding.  By \cite[Lemma 4.19]{mahan-hyp}, there is a $\rho_0$ depending only on the genus $g\ge 2$ and the maximum valence of $T$, so that there are $\rho_0$--qi-sections of the bundles $\tmtdw\to \but$ and $\tmtdt\to\but$ passing through any chosen point.
When we refer to qi-sections of these bundles below, we assume that they are  $\rho_0-$qi-sections.

\begin{defn}\label{def-qibhallway}
	A disk $f : [a,b]{\times}{I} \rightarrow 
	\tmtdt$ is a {\bf qi-section bounded hallway}  if:
	\begin{enumerate}
		\item for all $ v \in \but$, $f^{-1} ({X_v}) = \{t \}{\times}
		I$ for some $t \in [a,b]$. Further,
		$ f$ maps $t{\times}I$ to a geodesic in  $(X_v,d_v)$. The length of the geodesic {$ f (\{t \}{\times}I)$} in $X_v$ will be denoted by $\LL_t$. 
		\item 	For all {$s \in I$}, $P\circ f$ is an isometry of $[a,b] \times \{ s\}$ (with the Euclidean metric) onto a geodesic $\sigma \subset \but$.
		\item $f ([a,b] \times \{0\})$ and $f ([a,b] \times \{1\})$
		are contained in $\rho_0$-qi-sections; in particular, they are $\rho_0$-qi-sections  of $P\circ f: [a,b] \times \{0\} \to \but$ and $P\circ f: [a,b] \times \{1\} \to \but$.
	\end{enumerate}
	The {\bf girth} of such a hallway is $\min_t \LL_t$.
\end{defn}

\begin{defn}\label{flare}
	The  space $\tmtdt$, is said to satisfy the {\bf qi-section bounded hallways flare}
	condition with parameters $\lambda > 1$, $m \geq 1$ and $H \geq 0$ if
	for any qi-section bounded hallway $f : [a,b]{\times}{I} \rightarrow 
	\tmtdt$   of girth at least $H$ and with $b-a \geq m$,
	$$\lambda \, \LL_{\frac{a+b}{2}} \leq \, {\rm max} \ \{ \LL_{a}, \LL_b \}.$$
\end{defn}

\begin{rmk}
  In {\cite[Theorem 4.20, Corollary 4.23, 4.24]{mahan-hyp}} the equivalence of the flare condition above and hyperbolicity of $\tmtdt$ was established.  Thus, Theorem \ref{mainprel} implies the existence of constants $\lambda > 1$, $m \geq 1$, $H \geq 0$ as above such that $\tmtdt$ satisfies  qi-section bounded hallways flare
condition with parameters $\lambda, m, H$.
\end{rmk}

\begin{defn}\label{def-flareinall} Suppose that  $\tmtdt$ satisfies the  qi-section bounded hallways flare
	condition with parameters $\lambda > 1$, $m \geq 1$ and $H \geq 0$.
	A  subset $Y\subset (X_v,d_v)$ will be said to {\bf flare in all directions with parameter $K$} if for any geodesic segment $[c,d] \subset (X_v,d_v)$ with $c, d \in Y$ and any qi-section bounded hallway $f:[0,k] \times I \to \tmtdt$ of girth at least $H$ satisfying 
	\begin{enumerate}
		\item $f(\{0\} \times I) = [c,d]$,
		\item $\LL_0 \geq K$,
		\item $k \geq K$,
	\end{enumerate}
	the length $\LL_k$ of $f(\{k\} \times I)$ satisfies $\LL_k  \geq \,  \lambda \LL_0.$
\end{defn}

\begin{prop}\label{effectiveqc}  \cite[Proposition 4.27]{mahan-hyp}
	Given $K, C, \delta_0$, there exists $C_0$ such that the following holds.\\ Suppose that $\tmtdt$ is $\delta_0-$hyperbolic.
	Let $P: \tmtdt \to \but$ and $X_v$ be as above. If $Y$ is a $C-$quasiconvex subset of $(X_v,d_v)$ and flares in all directions with parameter $K$, then $Y$ is  $C_0-$quasiconvex in $\tmtdt$.
	
	Conversely, given $\delta_0, C_0$, there exist $K, C$ such that the following holds.\\ Suppose that $\tmtdt$ is $\delta_0-$hyperbolic.
	If 
	$\, Y \subset X_v$ is   $C_0-$quasiconvex in $\tmtdt$, then it is a $C-$quasiconvex subset in $(X_v,d_v)$ and flares in all directions with parameter $K$.
\end{prop}

\section{Geometric limits}\label{sec-geolts}
We shall need a few facts on geometric limits of doubly degenerate hyperbolic 3-manifolds $N_l$ of special split geometry (see for instance \cite[Chapters 8, 9]{thurstonnotes}, \cite[Chapter I.3]{CEG} and \cite[Chapters 8, 9]{kapovich-book} for details on geometric limits). {We refer especially to \cite{ohshika,os} for a detailed classification of geometric limits of Kleinian surface groups. In \cite[Section 4.5]{ohshika}, Ohshika discusses geometric limits of hierarchies. This sets up an exact dictionary between
	\begin{enumerate}
		\item Geometric limits of hierarchies in the Masur-Minsky marking complex \cite{masur-minsky2},
		\item Geometric limits of model manifolds constructed by Minsky \cite{minsky-elc1},
		\item Geometric limits of hyperbolic manifolds. The ending lamination theorem of Brock-Canary-Minsky \cite{minsky-elc2} establishes that a hyperbolic manifold with given end-invariants is bi-Lipschitz homeomorphic to the model manifold of  \cite{minsky-elc1} with constants depending only on the genus.
	\end{enumerate} In short, the dictionary between hierarchy paths \cite{masur-minsky2}, model manifolds  \cite{minsky-elc1}, and
	doubly degenerate hyperbolic manifolds established by the Ending Lamination Theorem and the model manifold technology that goes into it \cite{minsky-elc2} is extended to geometric limits in \cite{ohshika,os}. Here, we specialize this dictionary  to special split geometry manifolds.}

In the proof of Lemma \ref{qcMl}, we shall need to consider geometric limits of  a sequence of special split geometry manifolds $N_n$ with fixed parameters. {We refer the reader to Definition~\ref{def-ddsplgeoltightrthick} for the notion of special split geometry, and to Proposition~\ref{prop-splsplit} for the structure of manifolds having this geometry.}
For every $n$, fix a base split surface $S_{0,n} \subset N_n$ containing a base-point $x_n$. Since $N_n$ has special split geometry,  $S_{0,n}$ can be chosen to be of uniformly (independent of $n$) bounded geometry. Further, we may also assume that
$S_{0,n}$ does not intersect any of the Margulis tubes in $N_n$.
{For the purposes of this paper, a sequence $\{(Z_n, z_n)\}$ of geodesic metric spaces with base points \emph{Gromov-Hausdorff  converges}
	to a complete metric space $(Z_\infty,z_\infty)$ if
	\begin{enumerate}
	\item  there
	exist $K_n-$bi-Lipschitz homeomorphic embeddings  $H_n : (B_{R_n} (z_n ),z_n)
	\to (Z_\infty,z_\infty)$ of $R_n-$balls about $z_n$ into $(Z_\infty,z_\infty)$ with
	$K_n \to 1$ and $R_n \to \infty$, \\
	\item For every $z \in Z_\infty$, there exists $N\in \natls$, such that $z \in H_n (B_{R_n} (z_n ))$ for all $n \geq N$.
	\end{enumerate}
	When $\{(Z_n, z_n)\}$ are hyperbolic 3-manifolds, Gromov-Hausdorff  convergence specializes to geometric convergence. {We refer the reader to \cite[Definition 9.1.1]{thurstonnotes} for Thurston's definition of geometric convergence for hyperbolic 3-manifolds,
		and \cite[Chapter 3]{CEG} for the equivalence between pointed 
		Gromov-Hausdorff  convergence and geometric convergence.}
	
	Geometric convergence of hierarchies, on the other hand, represents a convergence of the encoding devices, the hierarchies. It is worth pointing out that this notion is arranged so that hierarchies converge geometrically if and only if the associated model manifolds converge in the Gromov-Hausdorff sense.  }

{
	\begin{defn}\label{def-geocgncetriples}
		A sequence of triples $\{(N_n,S_{0,n},x_n)\}$ of hyperbolic manifolds (resp.\ manifolds with piecewise Riemannian metrics)
		is said to converge geometrically (resp.\ Gromov-Hausdorff) to 
		$(N_\infty,S_{0,\infty},x_\infty)$ if
		\begin{enumerate}
			\item   $(N_n,x_n)$ converges to	$(N_\infty,x_\infty)$ geometrically (resp.\ Gromov-Hausdorff), and
			\item with induced path metrics $(S_{0,n},x_n)$ Gromov-Hausdorff converges to
			$(S_{0,\infty},x_\infty)$.
		\end{enumerate}
\end{defn}}

{
	\begin{rmk}\label{rmk-bddgeo}
		We note at the outset that if $N_n$ is a sequence of doubly degenerate hyperbolic manifolds with injectivity radius bounded below by
		$\ep>0$, then any geometric limit $N_\infty$ is also a doubly degenerate hyperbolic manifold with injectivity radius bounded below by
		$\ep>0$. One way to see this is via \cite{minsky-elc2} or \cite{rafi} where the injectivity radius bound translates to uniform bounds on subsurface projections.
		Further, by \cite[Section 4.5]{ohshika} (see the dictionary mentioned at the beginning of this subsection), the model manifold for the geometric limit is obtained via the model metric construction of \cite{minsky-elc1} applied to the 
		geometric limit of the hierarchies. The geometric limit of the hierarchies furnishes 
		the same uniform bound on subsurface projections.
		Hence any geometric limit $N_\infty$ is also a doubly degenerate hyperbolic manifold with injectivity radius bounded below. That the same $\ep$ suffices is evident from convergence of $\ep-$balls in the geometric limit.
\end{rmk}}

{Below, we shall be particularly interested in the following cases:
	\begin{enumerate}
		\item $\{(N_n,S_{0,n},x_n)\}$ is a sequence of split geometry hyperbolic manifolds equipped with base surfaces $S_{0,n}$.
		Here, geometric convergence will be the relevant notion of convergence.
		\item $\{(M_n,S_{0,n},x_n)\}$ is a sequence of model manifolds equipped with the
		welded metric $\dw$  and base surfaces $S_{0,n}$.  Here, Gromov-Hausdorff convergence will be the relevant notion of convergence. Note that $\dw$ is not necessarily a smooth metric. However, $\dw$ is smooth both when restricted to the risers, as well as away from the risers. Hence $\dw$ gives rise to a piecewise Riemannian metric on $(M_n,\dw)$.
\end{enumerate}}

\subsection{Geometric limits of $(N_n,x_n)$}\label{sec-geoltshyp} Let  $\{(N_n,S_{0,n},x_n)\}$ be a sequence of split geometry hyperbolic manifolds with parameters $D, \ep$.
After passing to a subsequence if necessary, we assume henceforth that the triples $(N_n,S_{0,n},x_n)$ converge geometrically to 
$(N_\infty,S_{0,\infty},x_\infty)$. This is possible, since $S_{0,n}$'s have been chosen to be of uniformly (independent of $n$) bounded geometry. Geometric convergence guarantees the existence of  $L_n-$bi-Lipschitz homeomorphisms $H_n:
(S_{0,\infty},x_\infty)\to (S_{0,n},x_n)$ with $L_n\to 1$ as $n \to \infty$.

{To describe geometric limits of $N_n$, we first describe geometric limits of individual split blocks.}

{
	\begin{lemma}\label{lem-geoltsplitb} Fix $D \geq 1, \ep >0$.
		Let $B_m$ be a sequence of split blocks with splitting tubes $\T_m \subset B_m$
		and parameters $D, \ep$.
		Let $S_m$ and $S_m'$ denote the boundary components of $B_m$.
		Let $h_m$ denote the height of $B_m$. If $h_m \leq h$ for all $m$, then 
		any limit of $\{B_m, S_m, x_m\}$ is of the form $\{B_\infty, S_\infty, x_\infty\}$,
		where $B_\infty$ is a split block of height at most $h$.
\end{lemma}}

\begin{proof} {In the proof below, we assume that we have passed to a subsequence whenever necessary to ensure convergence. Due to the structure of
		a splitting tube $\T$ in a split block $B$ (Definition \ref{def-splitblock}), the boundary $\partial \T$ is bi-Lipschitz homeomorphic to $S^1_e \times S^1_m$, where the first factor $S^1_e$ is  the unit circle. Further, $S^1_e$ is homotopic 
		to the core curve of $\T$ and, for any $p \in S^1_e$, $p \times S^1_m$ (the second factor) bounds a disk in $\T$. By $D-$boundedness, the two annuli comprising $\T \cap \partial B$ have core curves in $\partial B$ of length bounded above in terms of
		$D$.
	For any $C_0$,	since the number of curves on $\partial B_m$ of length bounded above by $C_0$ is uniformly bounded (independent of $m$), we can  pass to a subsequence so that a fixed curve on $S$ corresponds to the core curve of $\T_m \subset B_m$ for all $m$.
		Assume without loss of generality that the base-points $x_m \in \T_m \cap S_m$, i.e.\ the splitting tubes $\T_m$ intersect the boundary $S_m$ in annuli containing $x_m$ (this is possible by Proposition \ref{prop-splsplit}). Since the height of a block $B$ controls
		the diameter of $B \setminus \T$ in the presence of a lower bound on injectivity radius (see the paragraph preceding Remark \ref{rmk-ht}), the sequence $\{ (\T_m,  x_m)\}$ converges to a hyperbolic solid torus $(\T_\infty,x_\infty)$ with boundary.}
	
	{
		We next observe that, as in Remark \ref{rmk-bddgeo}, the sequence
		$\{((B_m\setminus \T_m), S_m, x_m)\}$ converges  to $\{(B_\infty\setminus \T_\infty'), S_\infty, x_\infty\}$ where $(B_\infty\setminus \T_\infty')$ has injectivity radius bounded below by $\ep$.}

	{
		Finally, observing that convergence forces convergence of $(\partial \T_m, x_m)$ to both $(\partial \T_\infty', x_\infty)$ and  $(\partial \T_\infty, x_\infty)$, it follows that the
		last two are isometric. Hence  any limit of $\{B_m, S_m, x_m\}$ is of the form $\{B_\infty, S_\infty, x_\infty\}$,
		where $B_\infty$ is a split block of height at most $h$.}
\end{proof}

{
	\begin{defn}\label{def-ltsplit} Let $B_m$ be a sequence of split blocks with splitting tubes $\T_m \subset B_m$
		and parameters $D, \ep$. Let $(B_m,x_m) \to (B_\infty,x_\infty)$ geometrically, where, as before, we assume that $x_m \in S_m$. Let $h_m$ denote the height of $B_m$.
		If  $h_m$ tends to infinity as $n$ tends to infinity, the geometric limit 
		$B_{\infty}$ shall be called a {\bf limiting split block}.
\end{defn}}

{
	\begin{lemma}\label{lemma-rk1cuspinlimitsplit} With setup as in Definition \ref{def-ltsplit},
		a limiting split block $B_{\infty}$ contains a rank one cusp $\T_\infty$ arising as a limit of the splitting tubes $\T_{m}$. Away from $\T_\infty$, the injectivity radius of
		$B_{\infty}$ is bounded below by $\ep$.
\end{lemma}}

{
	\begin{proof} We only  give a quick sketch as  the argument is 
		similar to Lemma \ref{lem-geoltsplitb}.
		Since $B_m \to B_\infty$, we can assume that for all $m$ large enough, the splitting tube $\T_m$ corresponds to a fixed curve $v$ in $\CC(S)$.
		Since $h_m \to \infty$ as $m \to \infty$, the subsurface projections on $S \setminus i(v)$ tend to infinity and the length of the core curve $\gamma_m$ in $\T_m$ tends to zero as $m \to \infty$.
		Further, since $x_m \in \T_m \cap S_m$, it follows that $d_m(x_m,\gamma_m)\to \infty$  as $m \to \infty$, where $d_m$ is the metric on $B_m$. Hence $(\T_m,x_m)$ converges to a rank one cusp
		$(\T_\infty,x_\infty)$. The last statement follows as before.
\end{proof}}

{
	Returning to special split geometry manifolds $N_n$,
	we shall now describe geometric limits for the positive and negative ends 
	$N_n^+$ and $N_n^-$ of $N_n$ as $n \to \infty$. We discuss the sequence of positive ends $N_n^+$ below. A similar discussion holds for  negative ends.
	Denote the $i-$th split surface (resp.  split block) of $N_n$ as
	$S_{i,n}$ (resp. $B_{i,n}$). Denote the $i-$th splitting tube, i.e. the splitting tube in $B_{i,n}$,  by
	$\T_{i,n}$. Let $h_{i,n}$ denote the height (see Definition \ref{def-ht}) of  $B_{i,n}$.
	Two cases arise for each $i$ (passing to a further subsequence if necessary): Either $h_{i,n}$ remains bounded as $n$ tends to infinity or  $h_{i,n}$ tends to infinity as $n$ tends to infinity.}

{
	If  $h_{i,n}$ remains bounded for all $i \geq 0$, then the positive ends $N_n^+$  converge to a degenerate end $N_\infty^+$ of special split geometry,
	since each sequence of split blocks $\{B_{i,n}\}$ does so by Lemma \ref{lem-geoltsplitb}. For all $j \geq 0$, the sequence of split surfaces $S_{j,n}$  converges to a $D-$bounded geometry surface $S_{j,\infty}$. We shall  refer to each such $S_{j,\infty}$ as a {\bf split surface in $N_\infty$.}}

{Else, let $i_0$ be the least positive integer such that $h_{i_0,n}$ tends to infinity as $n$ tends to infinity. By Lemma \ref{lem-geoltsplitb}, the split blocks $B_{i,n}$ (with base-points
	on $S_{i-1,n}$) converge to split blocks $B_{i,\infty}$  (with base-point
	on $S_{i-1,\infty}$) for $i<i_0$. Thus the union of the first $i_0$ blocks $\bigcup_{j=1}^{i_0} B_{j,n}$  (with base-points
	on $S_{0,n}$) converge to $\bigcup_{j=1}^{i_0} B_{j,\infty}$  (with base-point
	on $S_{0,\infty}$).
	Finally, the splitting tubes $\T_{i_0,n}$ converge to a rank one cusp in a limiting split block $B_{i_0,\infty}$ by Lemma \ref{lemma-rk1cuspinlimitsplit}. }

{
	Thus, when $h_{i,n}$ tends to infinity as $n$ tends to infinity,    the split blocks  $B_{j,n}$  converge in the  geometric limit  $N_\infty$, to split blocks  $B_{j,\infty}$ for $j<i$, while the split blocks  $B_{i_0,n}$ converge to a  limiting split block $B_{i_0,\infty}$ containing  a rank one cusp. For all $j>i_0$, the sequence of split blocks  $B_{j,n}$ satisfy
	$d_n(x_n, B_{j,n}) \to \infty$ for $x_n \in S_{0,n}$ and hence "vanish off to infinity".
	Hence, the sequence of split surfaces $S_{j,n}$, $j \leq i_0$ converges to a $D-$bounded geometry surface $S_{j,\infty}$. We shall also refer to such an $S_{j,\infty}$ for $j\leq i_0$ as a {\bf split surface in $N_\infty$.} Note that no such $S_{j,\infty}$ exists for $j>i_0$ in this case.} 

{\begin{rmk}\label{rmk-classfn-geolt}
We summarize the above discussion. Let  $\{N_n\}$ be  a sequence of special split geometry manifolds, with ends $\{N_n^\pm\}$.
There are two kinds of geometric limits 
$N_\infty^+$ possible:
\begin{enumerate}
\item Heights $h_{i,n}$ of split blocks remain bounded for all $i \geq 0$. In this case,
$N_\infty^+$ itself has special split geometry. This is illustrated in the top picture of Figure \ref{fig-splitlt}.
\item Else let $i_0 \geq 0$ be the least $i$ for which  $h_{i,n} \to \infty$ as $n\to \infty$. In this case, the split blocks $B_{i,n}$  converge to a split block $B_{i,\infty}$   for $i<i_0$. The split blocks  $B_{i_0,n}$ converge to a  limiting split block $B_{i_0,\infty}$ containing  a rank one cusp. Also $N_\infty^+$ is the union
of $B_{i,\infty}$   for $i\leq i_0$. This is illustrated in the bottom picture of Figure \ref{fig-splitlt}, where $i_0=1$. Compare Figure \ref{schematic2} for the picture of a single split block.
\end{enumerate}
A similar description occurs for $N_\infty^-$, giving four possible cases for the limit $N_\infty$. 
\end{rmk}}

\begin{figure}[H]
	
	\centering
	
	\includegraphics[height=10cm]{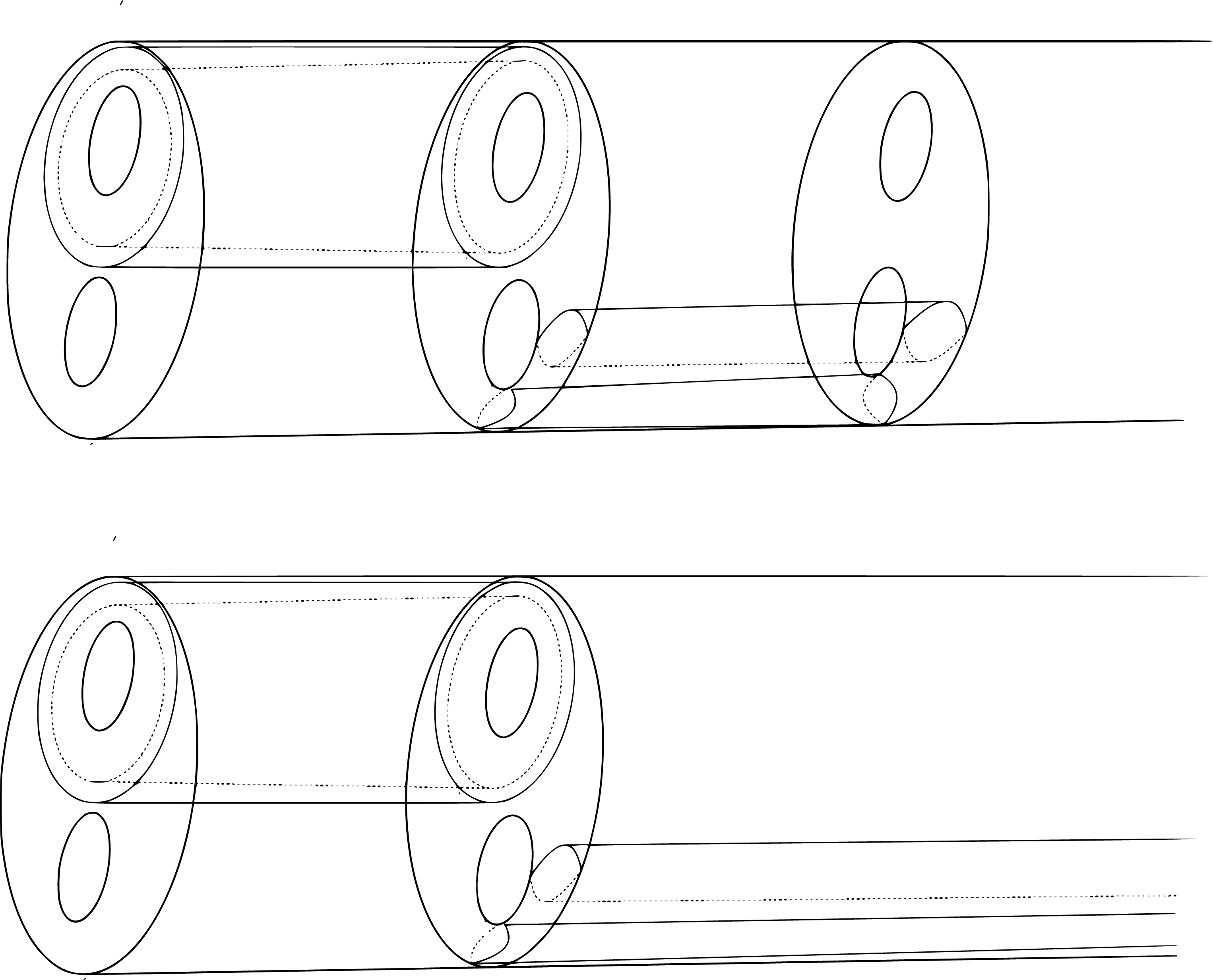}

	\bigskip
	
	\caption{Geometric limits of split geometry ends}
	\label{fig-splitlt}
\end{figure}
{
\subsection{ Gromov-Hausdorff limits of $(M_n, \dw)$}\label{sec-ghlts} Let $(M_n,\dw)$ denote 
	the metric graph bundle in Theorem \ref{model-str} corresponding to the doubly degenerate manifold $N_n$.
	We will now proceed to transfer the above geometric convergence statements to statements about Gromov-Hausdorff convergence of $(M_n, \dw)$.
	Let $(M_n, \dw)$ denote the  metric surface bundle (cf.\ Definition \ref{def-mbdl})  {$K-$}bi-Lipschitz to $N_n$ away from risers and Margulis tubes (Theorem \ref{model-str} Item (2)). We expand on the relation between the weld-metric  $(M_n, \dw)$ and the hyperbolic metric on the corresponding $N_n$ by extracting some details from the proof of \cite[Theorem 3.35]{mahan-hyp} as this is not explicit in the statement of Theorem \ref{model-str}. }

{
	It will suffice to describe $\dw$ for a split block $B$ with a splitting tube $\T$. Let $S, S'$ denote the boundary surfaces of $B$.
	Let $h$ denote the height of $B$. Let $B_0=
	(B\setminus  \T^o)$ (removing the interior of $\T$ from $B$). The boundary torus 
	$\partial \T$ consists of four parts:
	\begin{enumerate}
		\item $A^+= \partial \T \cap S$,  and  $A^-=\partial  \T \cap S'$, referred to as 
		\emph{small} annuli.
		\item $\partial \T \setminus (A^+ \cup A^-)$ consists of two 
		\emph{long} annuli $\AAA^+, \AAA^-$
		both $K-$bi-Lipschitz to the Margulis riser $\RR$ in $B$ (by Theorem \ref{model-str} Item (2)).
\end{enumerate}}

{
	Let $B_w$ denote the topological building block corresponding to $B$ (Definition \ref{def-topbb}). Topologically, $B_w$ may be obtained from $B$ by first removing 
	$\T^o$ and then constructing a surjective quotient map from $\partial \T$ to $\RR$ by collapsing $A^+, A^-$ to the boundary circles of $\RR$ and then diffeomorphically mapping the long annuli to $\RR$.
	We need to do this metrically to describe the weld metric $\dw$ on $B_w$.
	By Theorem \ref{model-str} (1), the Margulis riser $\RR$ in $(B_w,\dw)$ is isometric to $S^1_e \times [0,h]$. 
	Construct a map from $\partial \T$ to $\RR$ by
	\begin{enumerate}
		\item collapsing the small annuli $A^+$ (resp.\ $A^-$) to $S^1_e \times \{0\}$ (resp.\
		$S^1_e \times \{h\}$) via a $K-$Lipschitz map (sending the circle directions of 
		$A^\pm$ diffeomorphically to $S^1_e \times \{0\}$ or
		$S^1_e \times \{h\}$); and 
		\item sending the long annuli  $\AAA^+, \AAA^-$ to $\RR$ by 
		$K-$bi-Lipschitz diffeomorphisms.
	\end{enumerate}
	Note that $B_w \setminus \RR$ is naturally homeomorphic to $B \setminus \T$.
	The path metric $\dw$ is a \emph{singular Riemannian metric} on $ B_w$  obtained by minimizing over paths that consist of finitely many segments, each of whose interiors lie either entirely in $B_w\setminus \RR$ or entirely in $\RR$.  Paths in $B_w\setminus \RR$ are measured using a Riemannian metric $K-$bi-Lipschitz to the hyperbolic metric (see Theorem \ref{model-str}); paths in $\RR$ are measured using the product metric on $S^1_e \times [0,h]$.
	We shall refer to $(B_w,\dw)$ as the \emph{welded split block} associated to $B$ to emphasize the presence of the weld metric in the 
	topological building block of Definition \ref{def-topbb}.
	\begin{rmk}\label{rmk-alt}
		Alternately, {let $\psi_\pm:\bbar{\AAA_\pm} \to \RR$ denote the $K-$bi-Lipschitz diffeomorphisms in Item (2) above. Let 
		$\CC(\psi):= (\bbar{\AAA^+} \cup \bbar{\AAA^-}) \times [0,1] \cup_{\psi_\pm} \RR$ denote the mapping cylinder of $(\psi_+ \sqcup \psi_-) : 
		(\bbar{\AAA_+} \sqcup \bbar{\AAA_-})\to \RR$, where  $\sqcup$ denotes disjoint union. Now,
		\begin{enumerate}
		\item attach 	$\CC(\psi)$ to $\partial \T^o$ along 
		$(\bbar{\AAA^+} \cup \bbar{\AAA^-})$, and
		\item remove the short annuli altogether.
		\end{enumerate}
	The resulting object is an alternate (bi-Lipschitz equivalent) model for the \emph{welded split block} $(B_w,\dw)$.}
\end{rmk}}

{
	Let us now consider a sequence of split blocks $B_n$ 
	with $h_n \to \infty$. Let $(B_n^w,\dw)$ denote the  associated welded split block
	and $\RR_n \subset B_n^w$ denote its Margulis riser.
	Then, fixing a base-point $x_n \in S_n \cap \T_n$,
{	$(\partial \T_n,x_n) \to (\partial \T_\infty,x_\infty)$} and $(B_n,x_n) \to (B_\infty,x_\infty)$ where $B_\infty$ is a limiting split block by Lemma \ref{lemma-rk1cuspinlimitsplit}.
	Clearly, $S^1_e \times [0, h_n] \to S^1_e \times [0, \infty)=\RR_\infty$. 
	Passing to a subsequence if necessary, we can ensure that 
	\begin{enumerate}
		\item the metrics on $B_n^w \setminus \RR_n$  converge. Recall by Theorem \ref{model-str} that these metrics are $K-$bi-Lipschitz to the hyperbolic metrics on $B_n \setminus \T_n$.  
		\item the gluing  $K-$bi-Lipschitz diffeomorphisms used to define $(B_n^w,\dw)$ Gromov-Hausdorff converge to $K-$bi-Lipschitz maps from {$\partial \T_\infty$} to $S^1_e \times [0, \infty)$. Alternately, using Remark \ref{rmk-alt}, we can assume that the metric mapping tori
		Gromov-Hausdorff converge.
	\end{enumerate}
	Let $(B_\infty^w,\dw)$ denote the resulting 
	singular Riemannian manifold. 
	We refer to $(B_\infty^w,\dw)$ as the  \emph{limiting welded split block}. We record the following for later use.}

{
	\begin{lemma}\label{lem-ltbilip}
		$(B_\infty^w\setminus \RR_\infty,\dw)$ and $(B_\infty\setminus T_\infty)$ are
		$K-$bi-Lipschitz homeomorphic.
\end{lemma}}

{
	Note that $(B_w,\dw)$
	(resp.\  each $(B_n^w,\dw)$) is a metric surface bundle over $[0,h]$ (resp.\
	$[0,h_n])$ in the sense of Definition \ref{def-mbdl} by \cite[Theorem 3.35]{mahan-hyp}. The bundle structure passes to the limit, giving a metric bundle structure to $(B_\infty^w,\dw)$, where the base metric graph is $[0, \infty)$.}

{
	We turn now to the manifolds $(M_n, \dw)$ associated to the doubly degenerate manifolds $N_n$. Let $(M_n, \dw)$ denote a sequence of metric surface bundles as in Theorem \ref{model-str}. Note that each $(M_n, \dw)$ is built up of a union of welded split blocks whose metric structure has been described above.
	Let  $(M_\infty, \dw)$  denote the metric corresponding to the metric surface bundle   on the geometric limit  of the sequence $(M_n, \dw)$.  Let
	$N_\infty^0$ denote $N_\infty$ minus the union of Margulis tubes and rank one cusps. Similarly, let $M_\infty^0$ denote $M_\infty$ minus the union of limits of Margulis risers.}

From the way the metrics $\dw$ is constructed on $M_\infty$, we have the following  {Lemma. It    says  that the welding procedure and geometric limits essentially commute.}

{
	\begin{lemma}\label{weldandgeoltcommute}
		$N_\infty^0$ and $M_\infty^0$  are $K-$bi-Lipschitz homeomorphic.
\end{lemma}}

\begin{proof}
	{There are 2 cases as given by Lemmas \ref{lem-geoltsplitb} and \ref{lemma-rk1cuspinlimitsplit}. We deal with only the positive end as before.
		Suppose first that each split block has uniformly bounded height.
		Then  $N_\infty$ is a doubly degenerate manifold with split geometry. Further, after passing to a subsequence if necessary, we may assume that the metric mapping tori of
		Remark \ref{rmk-alt} converge. For a split block $B$ with splitting tube $\T$ in $N_n$ (for any $n\in \natls$), 
		let $(B_w,\dw)$  denote the corresponding welded split block, and let $\RR$ denote the Margulis riser in it. The metric on $(B_w\setminus \RR,\dw)$ is $K-$bi-Lipschitz to the
		hyperbolic metric 
		on $B \setminus \T$ (this is contained in  the proof of \cite[Theorem 3.35]{mahan-hyp} mentioned above). Any split block in $N_\infty$ is a limit of such split blocks, and the above property passes to limits. Concatenating all the split blocks of  $N_\infty$ together, we conclude that 	$N_\infty^0$ and $M_\infty^0$  are $K-$bi-Lipschitz homeomorphic.}
	
	{Next, suppose that there exists $i_0$ such that the heights $h_{i_0,n} \to \infty$ as in the discussion following Lemma \ref{lemma-rk1cuspinlimitsplit}. Assume that $i_0$ is the least such positive integer. For $j<i_0$, let $B_j$ denote the $j$th split block, $\T_j \subset B_j$ its splitting tube, and $(B_j^w,\dw)$ denote the associated welded split block, with Margulis riser $\RR_j$.
		By the argument above for the case where
		all heights are bounded, it follows that the hyperbolic metric on  $\bigcup_{j<i_0} (B_j \setminus \T_j)$ and the metric on $\bigcup_{j<i_0} (B_j^w \setminus \RR_j)$
		induced by $\dw$ are $K-$bi-Lipschitz. It therefore remains only to prove the 
		$K-$bi-Lipschitz property for limiting split blocks and associated 
		limiting welded split blocks. But this is the content of Lemma \ref{lem-ltbilip}.}
\end{proof}

	{
\subsection{Geometric limits of hierarchies}\label{sec-hierlts} To complete the dictionary mentioned at the beginning of this Section \ref{sec-geolts}, we describe briefly what  geometric limits of hierarchies associated to special split geometry manifolds look like. We have, in Section \ref{sec-geoltshyp} considered doubly degenerate manifolds of special split geometry corresponding to $L-$tight $R-$thick trees. By Proposition \ref{prop-splsplit} and Definition \ref{def-splsplit}, the parameters $D, \ep$ of special split geometry  depend only on $R$ for $L\geq 3$.
We start with the following, where we use the fact that $\partial \CC(S)=\EL(S)$
\cite{klarreich-el}.}

	{
\begin{lemma}\label{lem-injradforLtRt} There exist positive functions $R_0(\ep), L_0(\ep)$ satisfying $R_0(\ep) \to \infty$, $L_0(\ep)\to \infty$ as $\ep \to 0+$
such that the following holds.\\ 
Let $T=\{\cdots,v_{-1},v_0,v_1, \cdots\}$ be an $L-$tight $R-$thick tree of non-separating curves in $\CC(S)$ with underlying space $\R$ such that $L\geq 3$
and $R \leq R_0$.
Let $\LL^\pm$ denote the ending laminations given by the ideal end-points of $T$ in $\partial \CC(S)=\EL(S)$.
	 Let $N_T$ denote the doubly degenerate hyperbolic manifold
with end-invariants  $\LL^\pm$. Then the short curves of $N_T$ correspond to a subset
of the vertices $\{\cdots,v_{-1},v_0,v_1, \cdots\}$. More precisely, the $\ep-$thin part of $N$ consists of neighborhoods of closed geodesics corresponding to 
a subset
of the vertices $\{\cdots,v_{-1},v_0,v_1, \cdots\}$. Further, there exists $L_0$ such that if $L \geq L_0$, then the short curves of $N_T$ correspond to the
entire set of  vertices $\{\cdots,v_{-1},v_0,v_1, \cdots\}$ of $T$.
\end{lemma}}

\begin{proof}	{
The uniqueness of $N_T$ is an output of the ending lamination theorem \cite{minsky-elc2}.}

	{
The  collection of short curves in $N_T$ is determined precisely by the meridinal 
coefficients of tori in Minsky's model manifold
\cite[Theorem 9.11]{minsky-elc1}. $R-$thickness of $T$ now guarantees the first conclusion of the lemma: $\ep-$thickness away from neighborhoods of geodesics corresponding to 
$\{\cdots,v_{-1},v_0,v_1, \cdots\}$. }

	{
Finally, $L-$tightness of $T$ guarantees that the  meridinal 
coefficients of solid tori in Minsky's model manifold
\cite[Theorem 9.11]{minsky-elc1} are at least $L$. Taking $L_0$ sufficiently large
gives the second conclusion.}
\end{proof}

	{
 To describe geometric limits of hierarchies corresponding to 
 doubly degenerate manifolds of special split geometry, it suffices therefore to consider geometric limits of hierarchies corresponding to $L-$tight $R$-thick trees. As in the discussion following Lemma \ref{lemma-rk1cuspinlimitsplit} in Section \ref{sec-geoltshyp}, we shall describe only sequences of hierarchies for 
 positive ends. Let $\{v_{0,n}, v_{1,n}, v_{2,n}, \cdots\}$ be a sequence of 
 $L-$tight $R-$thick geodesics $\gamma_n$ with $d_{\CC(S \setminus v_{i,n})} (v_{i-1,n}, v_{i+1,n})=L_{i,n}$. Note that by \cite[Theorem 9.11]{minsky-elc1}, the heights 
 $h_{i,n}$ of split blocks occurring in Section \ref{sec-geoltshyp} and $L_{i,n}$ are comparable, i.e.\ there exists $c \geq 1$ (depending only on $S$) such that
 for all $i, n$
$ \frac{1}{c} L_{i,n} \leq h_{i,n} \leq cL_{i,n}.$ There are thus two cases
as in Section \ref{sec-geoltshyp}:\\
1) $L_{i,n}$ is bounded independent of all $i, n$. Then, after passing to a subsequence if necessary, $\gamma_n=\{v_{0,n}, v_{1,n}, v_{2,n}, \cdots\}$ converges to an 
$L-$tight $R-$thick geodesic $\gamma_\infty = \{v_{0,\infty}, v_{1,\infty}, v_{2,\infty}, \cdots\}$ corresponding to the geometric limit $N_\infty$ of $\{N_n\}$.\\
2) There exists a least $i_0$ such that $L_{i_0,n} \to \infty$ as $n \to \infty$.
This corresponds to a limiting split block (Definition \ref{def-ltsplit}).\\}

	{
We want to describe the hierarchy for a limiting split  block now. 
In this case the geometric limit of hierarchies is a finite 
$L-$tight $R-$thick geodesic $\gamma_\infty = \{v_{0,\infty}, v_{1,\infty}, v_{2,\infty}, \cdots, v_{i_0,\infty}\}$ along with a fully $R-$thick semi-infinite geodesic ray $r_{i_0} = \{v_{,i_0-1\infty}=w_0,w_1, w_2, \cdots\}$ in $\CC(S\setminus v_{i_0,\infty})$. Then the geometric limit of hierarchies consists of
\begin{enumerate}
\item the base geodesic  $\gamma_\infty$
\item a fully $R-$thick semi-infinite geodesic ray $r_{i_0}$  in $\CC(S\setminus v_{i_0,\infty})$ subordinate to $v_{i_0,\infty}
\in \gamma_\infty$
\item the remaining hierarchy paths in $\CC(S\setminus v_{j,\infty})$, for $j<i_0$ are all of bounded length and fully $R-$thick.
\end{enumerate}}

\subsection{Geometrically finite subsurfaces in geometric limits}\label{sec-gftread}
	{We shall now fix a geometric limit $N$ of special split geometry doubly degenerate manifolds with parameters $D, \ep$ as in Section \ref{sec-geoltshyp}. Further, let $S_0$ denote a distinguished split surface in $N$. Let $B_1, \cdots B_k, \cdots$ denote
the split blocks in the positive end $N^+$ in case $N^+$ itself has special split geometry. If $N^+$ has limiting special split geometry, then there exists $i_0 \geq 1$ such that $N^+$ has $(i_0-1)$   split blocks $B_1, \cdots B_{i_0-1}$, and one limiting split block $B_{i_0}$. It is possible that $i_0=1$, so that $B_1$ is itself a limiting split block, and $N^+$ has no split blocks. Similarly, let $B_{-1}, \cdots B_{-k}, \cdots$ denote
the split blocks in the negative end $N^-$ in case $N^-$ itself has special split geometry. If $N^-$ has limiting special split geometry, then there exists $j_0 \geq 1$ such that $N^-$ has $(j_0-1)$  split blocks $B_{-1}, \cdots B_{-(j_0-1)}$, and one limiting split block $B_{-j_0}$. }

	{Now, let $T=\{\cdots,v_{-1}, v_0,v_1, \cdots\}$ be an $L-$tight 
$R-$thick tree corresponding to $N$, so that $T^+ = \{v_0,v_1, \cdots\}$ 
and $T^-=\{
\cdots,v_{-1}, 
v_0\}$ 
correspond to $N^+, N^-$ respectively. Further (Section \ref{sec-hierlts})
\begin{enumerate}
\item $T^+$ (resp. $T^-$) is infinite if $N^+$ (resp. $N^-$) has special split geometry,
\item $T^+ = \{v_0,v_1, \cdots, v_{i_0}\}$ 
(resp.\  $T^-=\{v_{-j_0},\cdots,v_{-1},v_0\}$) is finite
if $N^+$ (resp. $N^-$) has limiting special split geometry.
\end{enumerate}}

	{Let $S_0$ denote the base split surface in $N$ so that $B_i$ corresponds to the vertex $v_i$ for  $i>0$,  and $B_{-j}$ corresponds to the vertex $v_{-j+1}$ for  $j>0$. Let $S_{01}$ denote a connected component of $S_0 \setminus (i(v_0) \cup i(v_1))$.}

	{\begin{lemma}\label{lem-cutsurfacegf}
$S_{01} \subset N$ is geometrically finite with  no accidental parabolics, i.e.\ it has parabolics possibly  only along $i(v_0), i(v_1))$.
Further,  $i(v_1))$ (resp.\ $i(v_0)$) is parabolic if and only if $i_0=1$
(resp.\ $j_0=0$).
\end{lemma}}

\begin{proof}	{
The end-invariant of the end $N^+$ (at infinity) is given by 
\begin{enumerate}
\item a lamination $\LL^+ \in \partial \CC(S)$ if $N^+$ is  of special split geometry.  In this case, $\LL^+$ is arational, minimal, and fills $S$.
\item a lamination of the form $\LL^+_0 \cup i(v_{i_0})$, if $N^+$ is  of limiting special split geometry, where 
$\LL^+_0 \in \partial \CC(S\setminus i(v_{i_0}))$. 
In this case, $\LL^+_0$ is arational, minimal, and fills $S\setminus i(v_{i_0})$.
\end{enumerate}}

	{In either case, no leaf of either $\LL^+$ or $\LL^+_0$ can be contained in 
$S_{01}$. Finally, $v_{i_0}$ has distance $i_0$ from $v_0$ in $\CC(S)$. Hence
$i(v_{i_0})$ is homotopic into $S_{01}$ if and only if $i_0=1$.}

	{Similarly, the end-invariant of the end $N^-$ (at infinity) is given by 
\begin{enumerate}
	\item a lamination $\LL^- \in \partial \CC(S)$ if $N^-$ is  of special split geometry. In this case, $\LL^-$ is arational, minimal, and fills $S$.
	\item a lamination of the form $\LL^-_0 \cup i(v_{-j_0+1})$, if $N^-$ is  of limiting special split geometry, where 
	$\LL^-_0 \in \partial \CC(S\setminus i(v_{-j_0}))$. 
	In this case, $\LL^-$ is arational, minimal, and fills $S\setminus i(v_{-j_0})$.
\end{enumerate}
Again,  no leaf of either $\LL^-$ or $\LL^-_0$ can be contained in 
$S_{01}$. Finally, as before, $i(v_{-j_0})$ is homotopic into $S_{01}$ if and only if $j_0=0$.}

	{The proof is completed by an application of the covering theorem  (see \cite[Theorem 9.2.2]{thurstonnotes} or \cite{canary-cover}) which asserts in this case that the cover of $N$ corresponding to $\pi_1(S_{01})$ is geometrically finite since 
 $S_{01}$ does not admit a finite quotient bounding a degenerate end of $N$.
Finally,   $i(v_1)$ is a parabolic if and only if $i(v_{i_0})$ is homotopic into $S_{01}$ if and only if $i_0=1$. Similarly, $i(v_0)$ is a parabolic if and only if 
$i(v_{-j_0})$ is homotopic into $S_{01}$ if and only if $j_0=0$.}
\end{proof}

\section{The stairstep construction}\label{sec-stairstep} The rest of the paper is devoted to constructing $\pi_1-$injective tracks $\TT$ in surface bundles $M$ over graphs $\GG$ and proving that any elevation $\til{\TT}$ is quasiconvex   in the universal cover $\til M$ (see Appendix~\ref{sec-app}, particularly Definition~\ref{def-eiq} and Theorem~\ref{eiqimpliescube}).  

\begin{rmk}\label{rmk-horvert}
{A caveat about the usage of `horizontal' and `vertical'. For convenience of exposition below, we think of the fibers 
$S$ as horizontal below, and the tree $T$ in $S \times T$ as vertical.}
\end{rmk}

\subsection{The stairstep construction in 3-manifolds}\label{sec-stairstep3}
In this section, we motivate the general stairstep construction by describing it in the simpler setting of a 3-manifold $N$ fibering over the circle with fiber $S$.  

\begin{defn}\label{def-stairstep}{ A stairstep in $S \times [0,n]$ is constructed from the following:
	\begin{enumerate}
        \item A tight geodesic {$v_0, v_1, \dots, v_n$} in the curve graph  $\CC(S)$ so that the closed closed curves $\sigma_i$ on $S$ corresponding to $v_i$ are homologous to one another up to orientation.
		\item For each $i\in \{1,\ldots,n\}$ an essential subsurface $\tr_i\subset S\times \{i-\frac{1}{2}\}$ (called a \defstyle{tread}) with boundary equal to $(\sigma_{i-1} \cup \sigma_{i}) \times \{i-\frac{1}{2}\}$.
		\item For each $i\in \{0,\ldots,n\}$, an  annulus $\rs_i$ (called a \defstyle{riser})  given by $\sigma_{i} \times I_i$, where
                  \begin{equation*}
                    I_i =
                    \begin{cases}
                      [0,\frac{1}{2}] & i = 0\\
                      [i-\frac{1}{2},i+\frac{1}{2}] & 0<i<n\\
                      [n-\frac{1}{2},n] & i=n.
                    \end{cases}
                  \end{equation*}  
            \end{enumerate}
            }

	The union of the treads and risers $\cup_i \rs_i \bigcup \cup \tr_i$ will be referred to as a {\bf stairstep} in $S \times [0,n]$ and denoted as $\TT$.  See Figure \ref{fig:stairstep} for a schematic, where treads are horizontal and risers are vertical.

 Gluing $S \times \{0\}$ to $S \times \{n\}$ via a homeomorphism $\phi: S \to S$ taking $\sigma_0$ to $\sigma_n$ we obtain a surface (we will also call this a stairstep) $\TT_N$ in the mapping torus $N=S \times [0,n]/\sim_\phi$.  
\end{defn}

{
\begin{rmk}\label{rmk-tread}
The hypothesis that the $\sigma_i$'s are homologous to each other  (up to orientation) is what allows us to construct the treads in Item (2) in Definition
\ref{def-stairstep} above.
\end{rmk}}

\begin{figure}[H]
	\includegraphics[height=5cm]{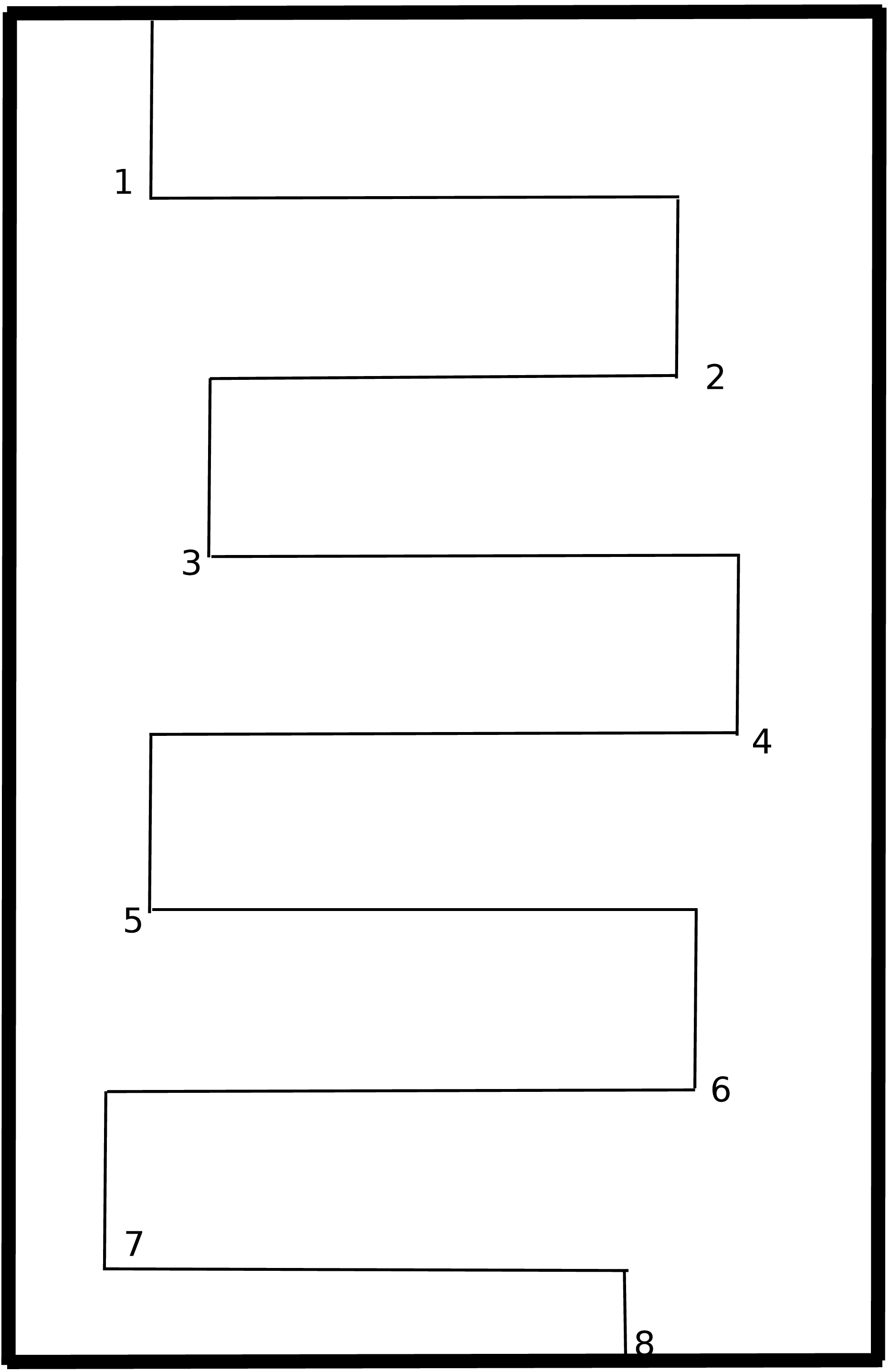}

	\medskip
	
	\caption{Stairstep in $S \times [0,n]$.}\label{fig:stairstep}
\end{figure}

\begin{eg} \label{eg-clr} An important motivating example is a geometrically finite surface constructed by
  Cooper-Long-Reid \cite[pp. 278-279]{clr}.  	In our language, what they build is a stairstep consisting of a single tread and riser.  	{
  Let $\phi: S \to S$ be a pseudo-Anosov map such  that there exists a nonseparating
  \emph{oriented} simple closed curve $\sigma$ satisfying the following:
  \begin{enumerate}
  \item $\sigma, \phi(\sigma)$ are  disjoint,
  \item $\sigma,  \phi(\sigma)$ denote the same nonzero class in $H_1(S)$.
  \end{enumerate}
  Set $\sigma_1=\sigma$ and 
$\sigma_0=\phi(\sigma)$. The tread in this context is the embedded essential subsurface of $S$
bounded by $\sigma_0$ and 
$\sigma_1$.}
	{
	As in Definition \ref{def-stairstep}, the union of the tread and riser in $S \times [0,1]$ is denoted as $\TT$ and referred to as a  stairstep  in the $I-$bundle. Here, $(\sigma_0, 0) = \TT \cap S \times \{0\}$ and $(\sigma_1, 1) = \TT \cap S \times \{1\}$. The gluing homeomorphism on the boundary of $S \times I$ is given   by (the identification map) $(\phi(x),0) \sim (x,1)$. Since  $\phi (\sigma_1) = \sigma_0$, the image of $\TT$ in $N$ is a stairstep surface $\TT_N$ given by identifying the boundary components of $\TT$ under the gluing homeomorphism.}
	
	{	Since $N$ is
	an oriented manifold, the gluing map from $S \times \{0\}$ to $S \times \{1\}$ must be orientation-reversing. Further, by requiring that  stairstep $\TT_N$ in the mapping torus $N$ is an oriented submanifold, the gluing map  from $\TT \cap (S \times \{1\})$ to $\TT \cap (S \times \{0\})$
	must also be orientation-reversing. This forces the co-orientation of $\TT$ to be preserved under the gluing map. Hence $\TT_N$ can be isotoped to be transverse to the suspension flow. 
Such a surface must be incompressible,
        as can be seen by lifting it to the universal cover and observing that any flowline can intersect the lift at most once.  (In fact, Cooper--Long--Reid prove not just incompressibility, but the stronger assertion that $\TT_N$ is Thurston norm minimizing).}

	{	To ensure  that the surface $\TT_N$ is geometrically finite (i.e.\ $\til{\TT_N}$ is quasiconvex in the universal cover $\til N$), \cite{clr} requires further that any elevation of $\TT_N$ to $\til N$ misses a flow-line.  That this suffices to show geometric finiteness uses some  machinery (either from the Thurston norm \cite[3.14]{clr} or from Cannon-Thurston maps); we refer to \cite{clr} for details.}
\end{eg}

\begin{rmk}\label{rmk-euler}
  Even if both the mapping torus $N$ and stairstep $\TT_N$ are orientable, it may not be possible in general to isotope $\TT_N$ to be transverse to the suspension flow, when there are multiple treads.  

  Indeed, let $e(T\FF) \in H^2(N)$ denote the Euler class of the tangent bundle to the foliation of $N$ by the fibers $S \times \{t\}$.  Fix an orientation on $\TT_N$. Let $\TT_N^+$ (resp. $\TT_N^-$) denote the collection of treads that are positively (resp. negatively) co-oriented with respect to the suspension flow in $N$.  Then  $\langle e(T\FF), \TT_N\rangle = \sum_{\tr_i \in \TT_N^+} \chi(\tr_i) - \sum_{\tr_i \in \TT_N^-} \chi(\tr_i)$ \cite{thurston-mams}.  If all the treads are not co-oriented in the same direction, it follows that $|\langle e(T\FF), \TT_N\rangle | < |\chi (\TT_N)| $ and hence $\TT_N$ cannot  be isotoped to be transverse to the suspension flow, cf.\ \cite[Expos\'e 14]{FLP}, especially the proof of  \cite[Theorem 14.6]{FLP} { due to Fried}. 

  We are sometimes able to use a different argument to prove incompressibility and quasi-convexity; see Sections \ref{sec-stairstepeg1} and \ref{sec-qctrack}. 
\end{rmk}

\subsection{The tree-stairstep}\label{sec-treestairstep}

\begin{defn}\label{def-tstairstep}  A tree-stairstep in the  topological model $M_T =S \times \BU(T)$ corresponding to a tight tree $T$ of homologous non-separating curves is built from the following data:
	\begin{enumerate}
		\item A tight tree $i:T\to \CC(S)$ of non-separating curves such that the simple closed curves $\{i(v): v \in V(T)\}$ on $S$ corresponding to $v$ are homologous to each other.
		\item For every pair of adjacent vertices $v, w$ of $T$, let $\tr_{vw}$ be an essential  subsurface  of the mid-surface $S_{vw}$ (cf.\ Definition \ref{def-topmodeltree}) with boundary equal to  $(i(v)\cup i(w))$. These  subsurfaces $\tr_{vw}$ shall be referred to as {\bf treads}. 
		\item For each $v\in T$, a Margulis riser $\rs_v$ in the topological building block $M_v$.
	\end{enumerate}
	The union  $\ttt$ of the treads and risers $\cup_v \rs_v \bigcup \cup \tr_{vw}$ will be referred to as a {\bf tree-stairstep} in $M_T =S \times \BU(T)$ associated to $i:T\to \CC(S)$.
\end{defn}

\subsection{Main Theorem}\label{sec-maintech}
The following is the main theorem of this paper and will be proven in Section \ref{sec-qctrack}.
\begin{restatable}{theorem}{maintech}\label{maintech} Given $R > 0$ and $V_0 \in \natls$, there exists $\delta, L_0, C \geq 0$ such that the following holds. Let $i:T \to \CC(S)$ be an $L-$tight  $R-$thick  tree of non-separating homologous curves  with $L\geq L_0$ such that the valence of any vertex of $T$ is at most $V_0$. Let $\ttt$ be a  tree-stairstep associated to $i:T\to \CC(S)$ and $\tttt$ be an elevation to $\til M_T$. Then
	\begin{enumerate}
		\item $(\til M_T,\dt)$ is $\delta-$hyperbolic.
		\item $\tttt$, {equipped with the induced path-pseudometric
		} is $C-$qi-embedded in $(\til{M_T},\dt)$.
		\item $\ttt$ is incompressible in $M_T$, i.e.\ $\pi_1(\ttt)$ injects into $\pi_1(M_T)$.
		\item If in addition there exists $L_1$ such that for every vertex $v$ of $T$ and for every pair of distinct vertices $u\neq w$ adjacent to $v$ in $T$, 
		\[d_{\CC(S \setminus i(v))} (i(u), i(w)) \leq L_1, \] then $\tmtdw$ is hyperbolic and $\tttt$ is quasiconvex in $(\til{M_T},\dw)$.
	\end{enumerate}
	
\end{restatable}

{
\begin{rmk}
In Item (2) of Theorem \ref{maintech} above, the ambient space  $(\til{M_T},\dt)$ is only a pseudometric space. Hence, in the presence of hyperbolicity, qi-embeddedness is a stronger condition than quasiconvexity. On the other hand, in 
Item (4) of Theorem \ref{maintech} above,
 $(\til{M_T},\dw)$ is a proper path metric space, and $\tttt$ is properly embedded. Hence, in the presence of hyperbolicity, qi-embeddedness coincides with quasiconvexity. This is the reason for the difference
in formulation of Items (2) and (4).
\end{rmk}}

\section{Examples of EIQ tracks}\label{sec-construction} In this section, we shall give some families of examples to which Theorem \ref{maintech} applies
(see also Definition~\ref{def-eiq} and Theorem~\ref{eiqimpliescube}).

We start with a  construction of embedded geometrically finite surfaces in hyperbolic 3-manifolds. We shall then construct EIQ tracks in complexes fibering over finite graphs and hence by Theorem \ref{eiqimpliescube} construct cubulable complexes whose fundamental groups $G$ are hyperbolic and fit into exact sequences of the form $$1\to \pi_1(S) \to G \to F_n\to 1.$$

\subsection{Stairsteps in 3-manifolds}\label{sec-stairstepeg1} Our aim here is to construct a hyperbolic 3-manifold fibering over the circle and a stairstep  in it. The first example of a stairstep we shall furnish has 2 treads and 2 risers. In a sense these are the simplest example of quasiconvex stairsteps. 
 \\  

{
	\noindent {\bf Marking graph $\MM(S)$:}  In the Lemma below, we shall use the 
the marking graph $\MM(S)$ from \cite{masur-minsky2}, where it is shown that the mapping class group acts properly, cocompactly by isometries on $\MM(S)$, and is
therefore quasi-isometric to $\MM(S)$ \cite[Section 7.1]{masur-minsky2}, \cite[p.1059-1060]{behrstock-minsky}. 

Recall first \cite[Section 2.5]{masur-minsky2} that a \emph{clean transverse curve}
$\beta$ to a simple closed curve $\gamma$ on a closed surface $S$ of genus greater than one is a simple closed curve 
in $S$ such that  
\begin{enumerate}
\item the subsurface $W$ of $S$ filled by $\gamma , \beta$ is either a one-holed torus or a 4-holed sphere,
\item  $d_{\CC(W)} (\gamma, \beta) =1$.
\end{enumerate}
We refer to $\gamma$ as the base curve of the clean transverse curve $\beta$
and denote it as $\rm{base} (\beta)$.

\begin{defn}\label{def-cm}\cite[p. 1059-1060]{behrstock-minsky}
 A \emph{clean marking} $\mu$ is a
pair $(\rm{base}(\mu), \rm{transversals})$ such that
\begin{enumerate}
	\item the \emph{base} $\rm{base}(\mu)$ of $\mu$ is a maximal clique in the curve graph
	$\CC(S)$,
	\item the transversals of $\mu$ consist of one clean transverse curve for each component of $\rm{base}(\mu)$. Further, each transversal $\beta$ is disjoint 
	from all curves of $\rm{base}(\mu)$ apart from $\rm{base} (\beta)$.
\end{enumerate}
\end{defn}

The vertices $\mu$ of the marking graph $\MM(S)$  are given by \emph{clean markings}. Edges of $\MM(S)$ are given by two kinds of elementary moves between markings, given by the \emph{twist} move and the  \emph{swap} move (see \cite[p. 1060]{behrstock-minsky} for details).

Let $A \subset S$ denote an essential annulus. The \emph{marking graph  $\MM(A)$ of the annulus $A$} is identified with the curve complex $\CC(S_A)$, where $S_A$ is the cover of $S$ corresponding to $\pi_1(A)$ \cite[Sections 2.4, 2.5]{masur-minsky2}. 
Definition~\ref{def-cm} differs slightly from, but is equivalent to, the original definition in \cite[Section 2.5]{masur-minsky2}.  Masur-Minsky define the marking,
not in terms of the clean transverse curve itself, but in terms of its projection to the annulus complex corresponding to its base.  Thus, clean markings  are given equivalently by
pairs $ (\gamma_i, \pi_{\gamma_i} (\beta_i))$, where
\begin{enumerate}
	\item $\{\gamma_i\}$ is a maximal clique in the curve graph
	$\CC(S)$,
	\item $ \pi_{\gamma_i}$ denotes the projection to the annulus complex corresponding to $\gamma_i$,
	\item $\beta_i$ is a clean transverse curve for $\gamma_i$,
	\item $\beta_i$ misses the other base curves of $\mu$
	\item the transversals to $\gamma_i$ are given by $\pi_{\gamma_i} (\beta_i)$.
\end{enumerate} }

{  We record the following Lemma for completeness and to set up the notation for Lemma~\ref{lem:exist_twists} below. {See \cite[Lemma 2.1]{behrstock-minsky} for a proof.}
 	\begin{lemma}\label{lem-qi2pdkt}
 	Let $MCG(S,\alpha)$ be the subgroup of $MCG(S)$ fixing a simple closed curve $\alpha$, i.e.\ $MCG(S,\alpha) = \, Stab(\alpha) \subset MCG(S)$.
 	Then $MCG(S,\alpha)$ is quasi-isometric to the product $MCG(S \setminus\alpha)\times \Z$, where $\Z$ is the mapping class group of an annular neighborhood $A$ of $\alpha$ (rel.\ boundary). 
 	\end{lemma}
 	}

For an essential annulus $A \subset S$,  let $d_A(. \, , \, .)$ denote distance between subsurface projections on $A$.

\begin{defn}\label{def-notwist} Let $R>0$.
For an element $\psi \in MCG(S,\alpha)$, the sequence $\psi^n$ is said to  
said to be {\bf renormalized by} $\tw_\alpha^{k_n}$ to  have {\bf  $R-$bounded Dehn twist} along $\alpha$ if, for all {$x \in \MM(S, \alpha)$}, 
$$d_{A} (x, \tw_\alpha^{k_n}\circ\, \psi^n(x)) \leq R.$$ 
\end{defn}

{
\begin{lemma}\label{lem:exist_twists}
 There exists $R_0 >0$ such that for all $R \geq R_0$ and $\psi \in MCG(S,\alpha)$,
 the following holds.\\ There exists a sequence
of renormalizing Dehn twists $\tw_\alpha^{k_n}$ as in Definition \ref{def-notwist} such that $\tw_\alpha^{k_n}\circ \psi^n(x)$
has  $R-$bounded Dehn twist along $\alpha$.
The same is true for $\tw_\alpha^{-k_n}\circ \psi^{-n}(x)$
\end{lemma}
\begin{proof}
Choose a complete clean marking $\mu \in \MM(S, \alpha)$. In particular, for every base curve $\gamma $ of $\mu$ other than $\alpha$, the transverse curve is disjoint from
$\alpha$, i.e.\ $\mu$ restricted to $S \setminus \alpha$ gives a clean marking on 
$S \setminus \alpha$. By definition of clean markings, the transversal 
$\beta$ to $\alpha$ misses every other base curve of $\mu$.

The marking $\mu_n=\psi^n (\mu)$ is again a clean marking satisfying the above conditions. We need to consider $\pi_{\alpha} (\mu_n)$. Since all other base curves
and their clean transverse curves 
are disjoint from $\alpha$, it follows that 
$$\pi_{\alpha} (\mu_n) = \pi_{\alpha} ( \psi^n(\beta)).$$
Next, note that $\langle\tw_\alpha\rangle$ acts {coboundedly} on $\MM(A)$,  where $A$ is the annular cover of $S$ corresponding to 
$\alpha$. Let $R_1$ denote the 
 diameter of the fundamental domain of the $\langle\tw_\alpha\rangle$-action on $\MM(A)$. {Note that
\cite[Lemma 2.3]{bkmm} $R_1$ may be chosen to be 2.}
Then for all $n \in \Z$, there exists $k_n \in \Z$   such that  
$$d_{A} ( \pi_{\alpha} (\beta), \tw_\alpha^{k_n}\circ\, \pi_{\alpha} (\psi^n(\beta)) \leq R_1.$$
Again, 
$$d_{A}(\tw_\alpha^{k_n}\circ\, \pi_{\alpha} (\psi^n(\beta)), 
\pi_{\alpha}  (\tw_\alpha^{k_n}\psi^n(\beta)) \leq R_1$$
for all $k_n \in \Z$. Choosing $R_0=2R_1$, we are through.

To prove the last assertion, observe that the action of 
$\partial \tw_{\alpha}^{-k_n}$ and $\partial \psi_A^{-1}$ pull back a point
in $\partial \tw_{\alpha}^{k_n} (\kappa_\psi^+) \cap \partial \psi_A (\kappa_\psi^+)$ to an intersection of base fundamental domains. The same argument
as above now furnishes the same conclusion for the inverses.
\end{proof}}

\begin{defn}\label{def-renpa} Let $R>0$.
	Let $\sigma \subset S$ be a multicurve.  Let $\psi:S \to S$ be a homeomorphism such that
	\begin{enumerate}
		\item $\psi$ fixes $\sigma$ point-wise.
		\item $\psi: W \to W$ is a pseudo-Anosov homeomorphism for every component $W$ of { $(S \setminus \sigma)$.}
		\item For every component $\sigma_i$ of $\sigma$, we choose $k_{in}$ such that $\tw_{\sigma_i}^{k_{in}}\, \circ\,\psi^n$ 
		and  $\tw_{\sigma_i}^{-k_{in}}\, \circ\,\psi^{-n}$ have $R-$bounded Dehn twist
		about $\sigma_i$.
	\end{enumerate}
Then {$\Pi_i \tw_{\sigma_i}^{k_{in}}\, \circ\,\psi^n$} is said to be a sequence of renormalized pseudo-Anosov homeomorphism of $S$ {\bf  in the complement of $\sigma$}, {where the product is taken over  all the components $\sigma_i$ of $\sigma$.}

Since an $R>0$ and a sequence $k_{in}$ as above always exists, we abbreviate this by saying that	 $\psi $ is  a  {\bf pseudo-Anosov homeomorphism of $S$  in the complement of $\sigma$}. Also, we shall denote the renormalized pseudo-Anosov homeomorphisms $\tw_{\sigma_i}^{k_{in}}\, \circ\,\psi^n$ by
$\overline{\psi^n}$.
\end{defn}

\begin{lemma}\label{lem-rthick} 
	Let $\sigma \subset S$ be a multicurve. 	
Let $\psi:S \to S$ be a   pseudo-Anosov homeomorphism of $S$   in the complement of $\sigma$, and let $A$ be an annular neighborhood of $\sigma$. Then there exists $R_0>0$ {and  renormalizations  $\overline{\psi^n}$} such that if {
\begin{enumerate}
	\item $U$ is a component  of $S \setminus A$,	
		\item  $\alpha$ is an essential simple closed curve on a geodesic in 
	$\CC(U)$ stabilized by $\psi|_U$,
	\item $W$ is
	\begin{itemize}
		\item either a  connected component of $A$,
	\item or a proper {non-annular} essential connected subsurface of $U$
	\item or an annular  essential subsurface of  $U$, not parallel to the boundary 
\end{itemize}
\item $n \in \Z$,
\end{enumerate}}
{then   $$d_W(\alpha, \overline{\psi^n}(\alpha)) \leq R_0,$$} {i.e.\ any geodesic in 
$\CC(U)$ stabilized by $\psi|_U$  is fully $R_0-$thick.
(Here, we assume, after passing to a finite power of $\psi$ ahead of time if necessary, that $\psi$ preserves every connected component of  $S \setminus A$.)}
\end{lemma}

\begin{proof}
For $W$ a  connected component of $A$, the existence 
of renormalizations $\bbar{\psi^n}$ is given by  Lemma~\ref{lem:exist_twists}.
This case now follows from the condition on $\bbar{\psi^n}$ that for every component $\sigma_i$ of $\sigma$, {the renormalizations $\overline{\psi^n}$} have uniformly bounded Dehn twist
about $\sigma_i$.

Next, let $W$ be
\begin{enumerate}
\item either a proper {non-annular} essential  subsurface of a component $U$ of $S \setminus A$, 
\item or an annular  essential subsurface of a component $U$ of $S \setminus A$, not parallel to the boundary of the component in which it lies.
\end{enumerate} {Here, $\psi|_U$ is a pseudo-Anosov homeomorphism.
This case now follows from Lemma \ref{lem-pA}.}
\end{proof}

\noindent {\bf Stairsteps in 3-manifolds: $2$ treads}\\
Let $\sigma_1, \sigma_2$ be homologous {distinct} non-separating simple closed curves on $S$. Let $\psi_i:S \to S$ be pseudo-Anosov homeomorphisms in the complement of $\sigma_i$, $i=1,2$. As in Definition \ref{def-renpa}, the renormalized pseudo-Anosov homeomorphisms will be denoted by $\overline{\psi_i^p}$.
{
For $i=1,2$, recall (Definition~\ref{def-renpa} and Lemma~\ref{lem:exist_twists}) that renormalizations satisfy the 
additional symmetric conditions
$$ \overline{\psi_i^{-p_i}}= (\overline{\psi_i^{p_i}})^{-1}.$$ 
For $p_1, p_2 \in \Z$,
define $$\Phi(p_1, p_2) =   \overline{\psi_2^{p_2}}.\overline{\psi_1^{p_1}}$$ and let $M(p_1, p_2)$ be the 3-manifold fibering over the circle with fiber $S$ and monodromy $\Phi(p_1, p_2)$. }

Since $\psi_i$ acts by 
a pseudo-Anosov homeomorphism on $(S \setminus \sigma_i)$, the translation length of $\overline{\psi_i^p}$ on the curve graph $\CC (S \setminus \sigma_i)$ is $O(p)$,
{i.e.\ there exists $C=C(\psi_1,\psi_2)$ such that for all $p \in \Z$ and $i=1,2$,
the translation length of $\overline{\psi_i^p}$ in $\CC (S \setminus \sigma_i)$  lies in $[|p|/C, C|p|]$. In particular, for any $u \in \CC (S \setminus \sigma_i)$, $d_{\CC (S \setminus \sigma_i)} (u, \overline{\psi_i^p} (u))$ is   $O(p)$. }

Further, by Lemma \ref{lem-rthick},
there exists $R$ such that any  geodesic in $\CC (S \setminus \sigma_i)$ preserved by $\psi_i$ (and hence by any of its powers or renormalized powers) is {fully $R-$thick.}

For  $i=1,2$, let $v_i$  be the vertex of $\ccs$ corresponding to $\sigma_i$. 
{\begin{lemma}\label{lemma-tg}
 Given $\psi_1, \psi_2$ as in the preceding discussion and any $L > 1$ there exists $p_0 > 1 $, such that
for all $p_1, p_2$ with $|p_1|, |p_2| \geq p_0$, the sequence $$T=\cdots, \Phi (p_1,p_2)^{-1}(v_1), \Phi (p_1,p_2)^{-1}(v_2), v_1, v_2, \Phi(p_1,p_2)(v_1), \Phi(p_1,p_2)(v_2),\cdots$$ is an $L-$tight  geodesic, where $\Phi(p_1,p_2)$ is defined as above.
\end{lemma}}
\begin{proof}
{	Since $v_1, v_2$ are adjacent vertices in $\CC(S)$, $[v_1, v_2]$  is a geodesic.
Since $\psi_2$ is a 
pseudo-Anosov in the complement of $\sigma_2$, for any $L\geq 3$,
there exists $p_0 > 1 $, such that
for all $ p_2$ with $|p_2| \geq p_0$, $\psi^{p_2}_2$ is an $L-$big rotation about $v_2$, i.e.\
$$d_{\CC(S \setminus v_2)} (v_1, \psi^{p_2}_2(v_1)) \geq L.$$
The sequence
 $v_1$, $v_2$, {$\psi^{p_2}_2(v_1)$} is  therefore $L-$tight. Consequently,  by Proposition \ref{isometrictighttree}, it is
	 an $L-$tight geodesic. }
	 
	{ Next,  there exists $p_0 > 1 $ (without loss of generality we can pick the same $p_0$), such that
	 for all $ p_1$ with $|p_1| \geq p_0$, $\psi^{p_1}_1$ is an $L-$big rotation about $v_1$, i.e.\
	 $$d_{\CC(S \setminus v_1)} (v_2, \psi^{-p_1}_1(v_2)) \geq L.$$}
	
{	Hence,  $ {\psi_1^{-p_1}}(v_2), v_1, v_2,$ {$\psi^{p_2}_2(v_1)$} is an $L-$tight sequence. By Proposition \ref{isometrictighttree} again, it is
	an $L-$tight geodesic. }
	
{Define $\Phi_0(p_1, p_2) =   {\psi_2^{p_2}}.{\psi_1^{p_1}}$ (without renormalization for convenience of notation), and
observe that $\Phi_0(p_1, p_2) ({\psi_1^{-p_1}}(v_2))={\psi_2^{p_2}}(v_2)=v_2$
and
	{$\Phi_0(p_1, p_2) (v_1)=\psi^{p_2}_2(v_1)$}, so that the sequence $ {\psi_1^{-p_1}}(v_2), v_1, v_2, \psi_2^{p_2}(v_1)$ equals (term-wise) the sequence
	$\Phi_0(p_1, p_2)^{-1} (v_2), v_1, v_2, \Phi_0(p_1, p_2) (v_1)$, where $\Phi_0(p_1, p_2)$ translates the first two vertices to the last two vertices.}

	{
	Finally, we observe that the above argument goes through verbatim with 
$	{\psi_i^{p_i}}$ replaced by $	\overline{\psi_i^{p_i}}$ and 
$\Phi_0(p_1, p_2)$ replaced by $\Phi(p_1, p_2)$.
Thus,  the sequence
$$T_0 = \Phi(p_1, p_2)^{-1} (v_2), v_1, v_2, \Phi(p_1, p_2) (v_1)$$ is an $L-$tight geodesic, where $\Phi(p_1, p_2)$ translates the first two vertices to the last two vertices. The union of translates $\Phi(p_1, p_2)^{j} (T_0): j \in \Z$ thus gives
the bi-infinite $L-$tight sequence $T$ in the statement of the lemma, where any subsequence of length 4 is an 
	$L-$tight geodesic. By Proposition \ref{isometrictighttree}, it is
	an $L-$tight geodesic. }
\end{proof}

{
	\subsection{$R-$thickness of $T$}
	
It remains to prove $R-$thickness of $T$. We shall draw heavily from the model manifold technology of \cite{minsky-elc1,minsky-elc2} and prove this in two steps:
\begin{enumerate}
\item In Lemma~\ref{lem-dgthick} and Corollary~\ref{cor-mppthick} we show the following. Recall the construction of  $M(p_1,p_2)$  in the discussion before Lemma~\ref{lemma-tg}.  Then the curves $\sigma_1, \sigma_2$ might well be realized by short geodesics as $p_1,p_2$ become large. We show that 
apart from Margulis tubes corresponding to these two curves,   injectivity radius for $M(p_1,p_2)$  is uniformly bounded away from zero independent of 
$p_1,p_2$, provided they are sufficiently large in absolute value. This involves using the geometry of the model manifold from  \cite{minsky-elc1,minsky-bddgeo}.
\item We then use the correspondence between short curves and large subsurface projections from \cite{minsky-elc1} along with the conclusion of step 1 to prove 
$R-$thickness of $T$.
\end{enumerate}.\\

\subsubsection{Model manifolds and injectivity radius}\label{subsub-model}
	We first observe
 that Lemma \ref{lem-pA} can be paraphrased as follows
	(and this is the formulation due to Minsky \cite{minsky-bddgeo}).
	Let $S$ be a surface possibly with punctures, and $\phi$ a pseudo-Anosov homeomorphism of $S$.
	Let  $\LL_+, \LL_-$ denote the stable and unstable laminations of $\phi$. Note that  $\LL_+, \LL_-$ are the ending laminations of a doubly degenerate hyperbolic
	3-manifold $N$ homeomorphic to $S \times \R$ obtained as a cover of a finite volume hyperbolic 3-manifold fibering over the circle with fiber $S$.
	Then there exists $R$ such that for any proper essential subsurface $W$ of $S$,
	$$d_W(\LL_+, \LL_-) \leq R.$$ This was generalized by Rafi \cite[Theorem 1.5]{rafi} to any 
	doubly degenerate hyperbolic 3-manifold $N$ homeomorphic to $S \times \R$ with ending laminations $\LL_+, \LL_-$. However, \cite[Theorem 1.1]{rafi} deals with a more general setup, where he establishes that `short curves in $N$ are exactly the boundary
	components of essential subsurfaces $Y \subset S$ such that projections of
	ending laminations $\LL_+, \LL_-$ to $Y$ are far apart
	in $\CC(Y)$, that is, $d_Y (\LL_+, \LL_-)$ is large.'

However, in the setup of Lemma~\ref{lemma-tg} above, involving $S, \sigma_1, \sigma_2, \psi_1, \psi_2$, we need \emph{uniform bounds} independent of $p_1, p_2$ provided they are large enough. 
This does not immediately follow from \cite{minsky-bddgeo} and \cite[Theorem 1.1]{rafi}. To obtain uniform bounds, we need to extract a geometric limit with $|p_1|, |p_2| \to \infty$.
This leads us to consider  hyperbolic 3-manifolds $N$ homeomorphic to $S \times \R$ with  end-invariants given as follows. Recall that $\psi_i$ is a 
	pseudo-Anosov on $S \setminus \sigma_i$. For $i=1,2$, let $\LL_{i+}, \LL_{i-}$ denote the stable and unstable laminations for $\psi_i$ thought of as supported in 
	$S \setminus \sigma_i \subset S$. Then the end-invariants of $N$ are defined  by
	\begin{enumerate}
		\item $\LL_-= \LL_{1-} \cup  \sigma_1$
		\item $\LL_+= \LL_{2+} \cup  \sigma_2$.
	\end{enumerate}
	Informally, we can think of $\LL_{1-}$ as $\psi_1^{-\infty} (\sigma_2)$ and 
	$\LL_{2+}$ as $\psi_2^{\infty} (\sigma_1)$. Thus, $N$ corresponds to $p_1 = -\infty,
	p_2 = \infty$ in the preceding notation. Then $N$ has two accidental parabolics corresponding to the curves $\sigma_1, \sigma_2$. Removing the (open) rank one cusps 
	corresponding to $\sigma_1, \sigma_2$ from $N$, we obtain $N_0$ with two ends
	$N_0^\pm$, where $N_0^+$ (resp.\ $N_0^-$) has the ending lamination $\LL_+$
	(resp.\ $\LL_-$). The existence of the complete hyperbolic structure on $N$ is really a consequence of the double limit theorem \cite{thurston-hypstr2} applied 
	to a pared manifold $(M, \PP)$, where $M=S \times [-1,1]$ and $\PP $ consists of
	annular neighborhoods $\AAA_1, \AAA_2$ of $\sigma_1 \times \{-1\}$ and $\sigma_2 \times \{1\}$
	in $S \times \{-1\}$ and $S \times \{1\}$ respectively. Choosing conformal structures
	$\tau_1, \tau_2$ on $S \times \{-1\} \setminus \AAA_1$ and $S \times \{1\} \setminus \AAA_2$ respectively,
	$N$ is the algebraic limit of the sequence of simultaneous uniformizations of 
	$(\psi_1^{-M} (\tau_1),\psi_2^{M} (\tau_2))$ furnished by the double limit theorem as $M \to \infty$. 

A more sophisticated description  yielding further details of the structure of $N_0$ can be obtained from the ending lamination theorem \cite{minsky-elc1,minsky-elc2} and 
	the combinatorial model therein. We refer the reader to the end of Section \ref{sec-hierlts}, where the hierarchy for a limiting split block was described. In the language of Section \ref{sec-hierlts}, both ends $N^+$ and $N^-$ are limiting split blocks with end-invariants $\LL_+$ and $\LL_-$ respectively.
	As in Section \ref{sec-hierlts},
	let $r_1=\{v_2, w_1, w_2, \cdots\}$ be a geodesic ray in $\CC(S \setminus v_1)$ joining $v_2$ to $\LL_{1-}$ and 
	$r_2=\{v_1, w_1', w_2', \cdots\}$ be a geodesic ray in $\CC(S \setminus v_2)$ joining $v_1$ to $\LL_{2+}$. Then the sequence of 1-simplices
	$$\{\cdots, (v_1,w_2),(v_1,w_1), (v_1,v_2),(w_1',v_2),(w_2',v_2),\cdots\}$$ 
	in the curve graph of $S$ gives rise to a hierarchy of geodesics in the sense of \cite{masur-minsky2}, where the base geodesic is given by $g=[v_1,v_2]$, and $r_1, r_2$ are subordinate to $g$ supported on $S \setminus v_1$,  $S \setminus v_2$ respectively. Completing $g, r_1, r_2$ to a hierarchy, a combinatorial model
	$M_0$ for $N_0$ can be built as in \cite{minsky-elc1}, and such a model is uniformly
	(depending only on the genus of $S$) bi-Lipschitz to $N_0$ (see Lemma \ref{lem-injradforLtRt} for instance, where it is  illustrated that the only short curves correspond to $v_1, v_2$). We note below that 
	$N_0$ has a lower bound on its injectivity radius.

\begin{lemma}\label{lem-dgthick}
Let $ N_0$ be as above. Then $N_0$ has injectivity radius uniformly bounded below.
\end{lemma}

\begin{proof}
 A  model closely related to the above model for $N_0$ can be built as follows.  Let $M^+$ be a singly degenerate manifold diffeomorphic to $(S \setminus
	\sigma_2) \times [1, \infty)$ so that its ending lamination is $\LL_{2+} $. Let 
	$M^+_0$ denote $M^+$ minus a neighborhood of the rank one cusps corresponding to
	$\sigma_2$. Note that there are two such cusps, since $(S \setminus
	\sigma_2)$ has two boundary curves. Similarly, let $M^-$ be a singly degenerate manifold diffeomorphic to $(S \setminus
	\sigma_1) \times  (-\infty, -1]$ so that its ending lamination is $\LL_{1-} $.
	Let $M^-_0$ denote $M^-$ minus a neighborhood of the rank one cusps corresponding to
	$\sigma_1$. Finally, let $B= S \times [-1,1]$. Note that we are yet to determine 
	a metric on $B$. We assume that $M^+$ and $M^-$ are smooth manifolds with boundary with metrics on the boundary induced by the restriction of hyperbolic metrics on \emph{complete}  singly degenerate manifolds.
	Glue $M^-_0$ (resp. $M^+_0$) to $S \times \{-1\}$ (resp.\ $S \times \{1\}$) via the identity on $(S \setminus
	N_\ep(\sigma_1)) \times \{-1\}$ (resp. $(S \setminus
	N_\ep(\sigma_2)) \times \{1\}$), where $N_\ep$ denotes  a small $\ep-$neighborhood. This induces a smooth metric  on $(S \setminus
	N_\ep(\sigma_1)) \times \{-1\}$ and $(S \setminus
	N_\ep(\sigma_2)) \times \{1\}$. Extend this smoothly to some metric on $B$ to obtain a smooth Riemannian metric on $M_0^\pm= M_0^- \cup B \cup M_0^+$. Fixing a base surface in $N_0$ and identifying it with $S \times \{0\} \subset B \subset 
	M_0^\pm$, the concluding argument in \cite[Section 6]{minsky-jams} using Sullivan's Theorem \cite{sullivan}, \cite[Theorem 6.1]{minsky-jams} shows that $M_0^\pm$ is bi-Lipschitz to $N_0$. Next $M_0^\pm$ clearly has injectivity radius uniformly bounded below, since its ends $M_0^-,  M_0^+$ do. Hence so does 
	$N_0$.
\end{proof}

\begin{cor}\label{cor-mppthick}
	Let $M(p_1,p_2)$ denote the  hyperbolic 3-manifold fibering over the circle with fiber $S$, and monodromy $\Phi(p_1,p_2)$ as in the discussion preceding Lemma~\ref{lemma-tg}. Let $M(p_1,p_2)_S$ denote the cover of $M(p_1,p_2)$ corresponding to $\pi_1(S)$. Then there exists $p_0>1$, and $\ep>0$,  such that
	for all $p_1,p_2$ with  $|p_1|,|p_2| > p_0$ the following holds. \\
	Let $T$ be the tight geodesic from Lemma~\ref{lemma-tg} and $w \in \CC(S) \setminus T$. Also, let $\sigma_w$ denote a simple closed curve on $S$ corresponding to $w$. Then the geodesic realization of  $\sigma_w$ in 
 $M(p_1,p_2)_S$ has length at least $\ep$. (We summarize this by saying that away from tubes around closed geodesics corresponding to the vertices of $T$, $M(p_1,p_2)_S$  has injectivity radius bounded below by $\ep$.)
\end{cor}

\begin{proof}
We can choose a base surface of bounded geometry in $M(p_1,p_2)$ and lift it to
a base surface in $M(p_1,p_2)_S$. We argue by contradiction, using a geometric limit argument. Suppose that the conclusion fails with 
$p_1, p_2 \to \infty$.

As
$p_1, p_2 \to \infty$, the model manifolds $M(p_1,p_2)_S$ converge geometrically to a model manifold  $M_0$  for $N_0$ constructed as in Lemma~\ref{lem-dgthick}. By Lemma~\ref{lem-dgthick}, $N_0$ has injectivity radius uniformly bounded below by some $\ep>0$.

Hence, there exists $p_0 > 0$ such that for all $p_1, p_2$ satisfying 
$|p_1|,|p_2| > p_0$ the following holds. 
 Except for the curves corresponding to translates of $v_1, v_2$ (under powers of $\Phi(p_1,p_2)$), all other closed geodesics in $M(p_1,p_2)_S$ have length at least $\ep$. This proves the Corollary. 
\end{proof}

\subsubsection{Injectivity radius and subsurface projections}
We now summarize the output from \cite{minsky-elc1,minsky-elc2} that we shall need. Below, we shall set $N_1$ equal to
\begin{enumerate}
\item either $N_0$ as in Lemma~\ref{lem-dgthick}, or
\item $M(p_1,p_2)_S$ as in Corollary~\ref{cor-mppthick}.
\end{enumerate}

In \cite[Lemma 10.1]{minsky-elc1}, Minsky shows that large subsurface projections give short curves:
\begin{lemma}\label{lem-largesubsurfimpliesshort}
Given $\ep > 0$, there exists $k_0 \geq 0$ such that if at least one  of the following hold:
\begin{enumerate}
\item there exists an essential non-annular subsurface $W$ of $S$ such that $d_W(\LL_+,\LL_-) \geq k_0$, where $\LL_+,\LL_-$ are as in Section~\ref{subsub-model}, and further, $v$ is a boundary component of $W$,
\item Let $\AAA$ denote the annulus corresponding to $v$, and 
$d_\AAA(\LL_+,\LL_-) \geq k_0$
\end{enumerate} then $v$ corresponds to a geodesic of length at most
$\ep$ in $N$.
\end{lemma}

In \cite[Theorem 9.11]{minsky-elc1}, Minsky shows the converse. We state a weaker version of this theorem below.
\begin{theorem}\label{thm-shortimpliessubsurf}
Given $k_0 \geq 0$  there exists $\ep > 0$, such that if both  of the following hold:
\begin{enumerate}
	\item for any non-annular subsurface $W$ of $S$ such that  $v$ is a boundary component of $W$, we have $d_W(\LL_+,\LL_-) \leq k_0$, 
	\item Let $\AAA$ denote the annulus corresponding to $v$, and 
	$d_\AAA(\LL_+,\LL_-) \leq k_0$
\end{enumerate} then $v$ corresponds to a geodesic of length at least
$\ep$ in $N$.
\end{theorem}

 We now have the following modified versions of Lemma \ref{lem-pA}.

{  \begin{cor}\label{cor-thick} Let $N_0, \sigma_1,\sigma_2$ be as in the discussion preceding Lemma~\ref{lem-dgthick}. 
		There exists $R$ such that
		for any proper essential subsurface $W$ of either 
		$(S \setminus
		\sigma_1)$ or $(S \setminus
		\sigma_2)$, $d_W(\LL_+, \LL_-) \leq R$.
\end{cor}}

\begin{proof}{  Let $W$ be as in
		the statement of the Corollary. If $W$ has a boundary component $\alpha$ other than $ \sigma_1, \sigma_2$, then  $\alpha$ cannot be a short curve by Lemma~\ref{lem-dgthick}. Hence, by Theorem~\ref{thm-shortimpliessubsurf} , i.e.\ \cite[[Theorem 9.11]{minsky-elc1}, the meridinal coefficients associated to a solid torus neighborhood of  $\alpha$ must be bounded, forcing subsurface projections to $W$ to be bounded. These bounds depend only on $S$.}
	
	{ It remains to consider $W$ a connected component of $ S \setminus
		(\sigma_1 \cup \sigma_2)$. There are finitely many such components. Since neither $\LL_+, \LL_-$ has a leaf in $W$, 
		$d_W(\LL_+, \LL_-)$ is finite in this particular case. Taking a maximum of such $d_W(\LL_+, \LL_-)$ over the finitely many components of $ S \setminus
		(\sigma_1 \cup \sigma_2)$, we are done.}
\end{proof}

{ We finally return to proving $R-$thickness of $T$.}
{\begin{lemma}\label{lem-2stepthick}
Let $T(=T(p_1,p_2))$ be as in Lemma \ref{lemma-tg}.  There exists $R$ and $p_0'$ such that for all $p_1, p_2$ satisfying $|p_1|, |p_2| \geq p_0'$, 
$T$ is $R-$thick.
\end{lemma}}

\begin{proof}{
Let $N_1= M( p_1,p_2)_S$ be as in Corollary~\ref{cor-mppthick}. Then  $T$ is the 
 base geodesic of the hierarchy corresponding to the model manifold  for $N_1$. The first set of
subordinate geodesics are translates (under powers of $\Phi(p_1,p_2)$) of geodesics in $\CC(S\setminus v_1)$ (resp.\ $\CC(S\setminus v_2)$) joining $v_2, \overline{\psi_1^{-p_1}} (v_2)$ (resp.\ $v_1, \overline{\psi_2^{p_2}} (v_1)$). We note that
since $\psi_1$ and $\psi_2$ are pseudo-Anosovs on $S \setminus i(v_1)$ and
$S \setminus i(v_2)$ respectively, 
\begin{enumerate}
\item the sequence $\{v_2, \psi_1^{-1}(v_2), \cdots,
\psi_1^{-p_1} (v_2)\}$ is a uniform (independent of $p_1$) quasigeodesic in  $\CC(S \setminus v_1)$
\item  the sequence $\{v_1, \psi_2(v_1), \cdots,
\psi_2^{p_2} (v_1)\}$ is a uniform (independent of $p_2$) quasigeodesic in  $\CC(S \setminus v_2)$.
\end{enumerate}
In particular,  these sequences lie in a uniformly bounded neighborhood of the geodesic rays $r_1, r_2$ in 
$\CC(S \setminus v_1)$ and $\CC(S \setminus v_2)$ described earlier. Here `uniform' means independent of $p_1, p_2$.}

By Corollary~\ref{cor-mppthick}, closed geodesics corresponding to vertices of $\CC(S)$ \emph{not} in $T$ have length at least $\ep$. 
Hence, by Theorem~\ref{thm-shortimpliessubsurf}, there exists  $R$ and $p_0'$ such that for all $p_1, p_2$ satisfying $|p_1|, |p_2| \geq p_0'$, $T$ is $R-$thick. 
\end{proof}

\begin{rmk} It was pointed out to us by the anonymous referee that 
Lemma~\ref{lem-2stepthick} admits an alternate proof using the hierarchical hyperbolicity/coarse median structures from \cite{bhs}.
This proof circumvents the use of the model manifold in Section~\ref{subsub-model}. Since we are interested also in identifying the short geodesics, we  have retained the proof here.
\end{rmk}
}

{
Lemma \ref{lemma-tg} and Lemma \ref{lem-2stepthick} combine to show that $T$ is 
$L-$tight and $R-$thick for all large enough $|p_1|,|p_2|$ so that Theorem \ref{maintech} applies.
To simplify notation, let us denote the cover of   $M(p_1, p_2)$ corresponding to $\pi_1(S)$ by $M_T$. }
Let $\ttt$ denote the associated stair-step in $M_T$ constructed equivariantly with respect to the $\Z-$action generated by  $\Phi(p_1,p_2)$. Let $\TT$ denote the quotient stairstep in  $M(p_1, p_2)$. Then, by Theorem \ref{maintech}   $\TT$ is geometrically finite in  $M(p_1, p_2)$.\\

We summarize the conclusion of the above  construction as follows:

\begin{prop}\label{qchierin3m}
	Let $v_1,v_2 \in \ccs$ be a pair of adjacent vertices in $\ccs$ such that $v_1, v_2$ correspond to homologous simple closed non-separating curves $\sigma_1, \sigma_2$. For $i = 1, 2$, let $\psi_i:S \to S$ be  a pseudo-Anosov  homeomorphism in the complement of $\sigma_i$.
{	There exists $p_0 > 1 $, such that
	for all $p_1, p_2$ with $|p_1|, |p_2|\geq p_0$, the following holds.}
	
	Let
	$\Phi(p_1, p_2)=\overline{\psi_2^{p_2}}.   \overline{\psi_1^{p_1}}$ and let $M(p_1,p_2)$ be the 3-manifold fibering over the circle with fiber $S$ and monodromy $\Phi(p_1, p_2)$. Then the associated stairstep surface $\TT$  in $M(p_1,p_2)$  is geometrically finite.
\end{prop}

\begin{rmk}\label{rmk-qchierin3m} The existence of a quasiconvex hierarchy for such $M(p_1,p_2)$
	now follows. Indeed,  Thurston's theorem 
{(see	\cite[Theorem 2.3]{thurston-bams} or \cite[p. 83]{otal}) along with a covering theorem  (see \cite[Theorem 9.2.2]{thurstonnotes} or \cite{canary-cover})} shows that the Haken hierarchy is quasiconvex provided that the first surface is so.  Cubulability of such manifolds 
	now follows from Wise's Theorem \ref{wise-hierarchy}). Theorem \ref{maintech2}  will
	generalize Theorem \ref{maintech} to the case of a tree of
	homologous curves (possibly separating). Thus Proposition \ref{qchierin3m} can be extended to prove the existence of a geometrically finite surface in a class of fibered hyperbolic   3-manifolds with first Betti number one.
 In \cite{sisto-jta}, 	Sisto found sufficient conditions on Heegaard splittings of rational homology 3-spheres $M$ to guarantee that they were Haken. The  construction of Proposition \ref{qchierin3m} may be regarded as an analog 
	of the main theorem of  \cite{sisto-jta} in the context of fibered manifolds $M$. 
\end{rmk}

\subsection{Tree  stairsteps}\label{sec-treestairstepeg} We will now generalize the above construction to surface bundles over graphs. Given $R>0$, choose
$L\gg 1$ as in Theorem \ref{maintech}. For  $j=1, \cdots, n$,
let $$\gamma_j:= \cdots, v_{-1,j}, v_0, v_{1,j}, \cdots$$ be  $R-$thick $L-$tight geodesics of homologous non-separating curves in $\ccs$  invariant under  pseudo-Anosov homeomorphisms $\Phi_j$ respectively (constructed as in Section \ref{sec-stairstepeg1}). Note that $v_0 $ belongs to $\gamma_j$ for all $j=1, \cdots, n$.
Since there is a unique non-separating curve on $S$ up to the action of the mapping class group $MCG(S)$ any tight geodesic of homologous non-separating curves may be translated by an element of $MCG(S)$ so that the image contains $v_0$ (and thus this is a rather mild condition).

For $j=2, \cdots, n$, and  $\Psi_j \in MCG(S,v_0)$,  
$\Psi_j \Phi_j \Psi_j^{-1}$ stabilize $\Psi_j (\gamma_j)$.
We think of $\Psi_j$'s as {\it rotations} about $v_0$ (in the spirit  of \cite{dgo}).

{
\begin{rmk}\label{rmk-dgo}
Though we never formally use any output of \cite{dgo}, Definitions 5.1 and 5.2 
and the ideas of Chapter 5 of \cite{dgo} have influenced this paper. The lengths of subsurface projections below replace the large rotation angles of   \cite[Chapter 5]{dgo}. The terminology of large rotations has been adopted in Definition \ref{def-largerot} below.
\end{rmk}}

 For convenience of notation, let $\Psi_1 \in MCG(S,v_0)$ be the identity map.
 { We shall choose the $\Psi_j$'s so that  the subsurface projections of geodesics $\Psi_j (\gamma_j)$ to $S \setminus i(v_0)$ are large. Further, let $\beta_1, \beta_2$ be any pair of distinct geodesic rays starting at $v_0$ and proceeding along any two distinct tight geodesics $\Psi_j (\gamma_j)$ with any orientation.
 Then we shall demand that $\beta_1\cup \beta_2$ is a tight geodesic whose
 subsurface projection to $S \setminus i(v_0)$ is also large.}
 
 See diagram below for $n=2$, where the thick vertical and horizontal lines indicate tight geodesics passing through the origin $v_0$. The thin broken lines indicate long tight geodesics in $\CC(S\setminus v_0)$ joining $v_{\pm 1, 1} \, , \, v_{\pm 1, 2}$.
{Their lengths  give a measure of  the subsurface projections.}
The  dotted lines indicate that all the vertices on the thin broken lines are at distance one from $v_0$ in $\CC(S)$ (we have drawn these only in the positive quadrant to prevent cluttering up the figure).

\begin{figure}[ht]
	
	\centering
	
	\includegraphics[height=6cm]{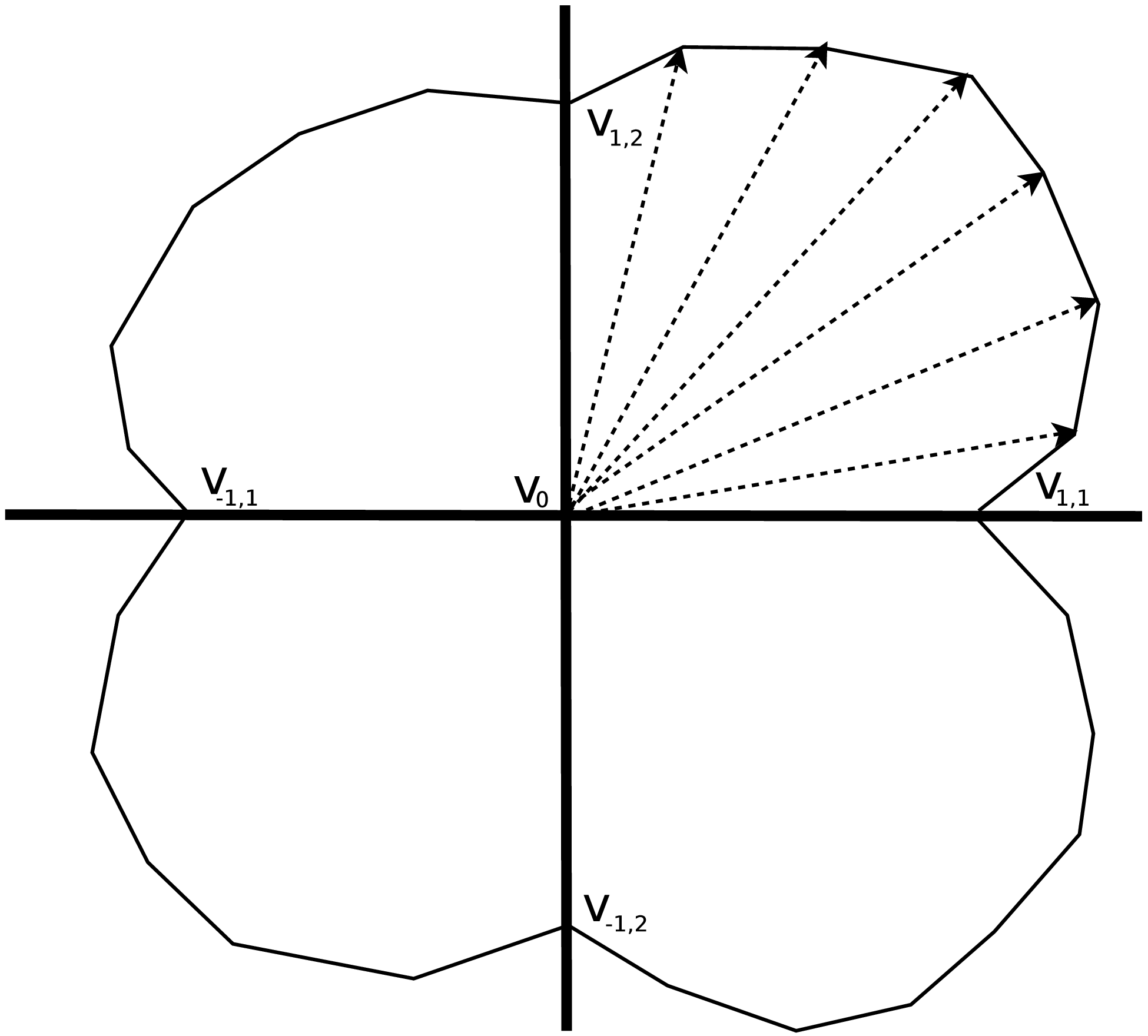}

	\bigskip
	
	\caption{Two tight geodesics meeting with large 
		{subsurface projections}}
	
\end{figure}

We describe the construction more precisely.
\begin{defn}\label{def-largerot} Suppose that   $\gamma_j = \{v_{ij}\}$, 
	$i \in \Z$, $j=1, \cdots , n$ is a finite collection of $L-$tight $R-$thick geodesics, all passing through $v_0 = v_{0j}$. 
	A collection $\Psi_j \in MCG(S,v_0)$, $j= 1, \cdots , m$ is said to be 
	a family of $(L,R)-${\bf large rotations} about $v_0$ for $\gamma_j, \,j=1, \cdots , n$ if
	\begin{enumerate}
	\item 	{ for $r=\pm 1$,  and  any  distinct $v, w \in \bigcup_j \{\Psi_j v_{r,j}\}$, we have
  $d_{\CC(S\setminus v_0)} (v,w ) \geq L$.
\item   for $r=\pm 1, \pm 2$,   any proper essential subsurface $W \subset
  S\setminus v_0$ and  any  distinct $v, w \in \bigcup_j \{\Psi_j v_{r,j}\}$, we have that
  $ d_W(v,w) \leq R$.
  \item Let $y,z$ be  vertices at distance one from $v_0$ on 
  $\gamma, \gamma' \in  \{\Psi_j (\gamma_j)\}$ respectively with $\gamma\neq \gamma'$. Further, let
  $x \in \gamma$ be such that $y$ lies between $v_0$ and $x$, and $d(x, v_0) $ equals 2 or 3. Then, for  any proper essential subsurface $W \subset
  S\setminus y$, $d_W(x, z) \leq R$.}
	\end{enumerate}
\end{defn}

\begin{prop}\label{ttforfn} There exists a universal constant $M$ such that for  $R>0$, we can choose $L_0\geq 3$ satisfying the following:\\
	For $L\geq L_0$, and for  $j=1, \cdots, n$, let $\gamma_j$  be  an $R-$thick $L-$tight-geodesic of homologous non-separating curves in $\ccs$  invariant under  pseudo-Anosov homeomorphisms $\Phi_j$. Further, assume that $\gamma_j/\langle \Phi_j\rangle$ has at least { 3 vertices (e.g.\ by choosing
	$\Phi_j$ of the form $\Phi(p_1,p_2)^2$, where $\Phi(p_1,p_2)$ 
	is as in in Proposition \ref{qchierin3m})}.
	Let $\Psi_j$'s be  $(L,R)-$large rotations about $v_0$ for $\gamma_j, \,j=1, \cdots , n$.

	Let $Q$ be the group generated by $\Psi_j \Phi_j \Psi_j^{-1}$, $j=1, \cdots, n$ (recall that $\Psi_1$ is the identity). Then the union of the  $Q-$translates of $\Psi_j (\gamma_j)$ is an $L-$tight $(R+M)-$thick tree $T$. Associated to $T$ is a $Q-$invariant track $\ttt$ in $M_T$ such that $\TT:= \ttt/Q$ is an EIQ track (cf.\ Definition~\ref{def-eiq}).
\end{prop}

\begin{proof} By choosing $L$ sufficiently large (depending only  the genus of $S$),
	the finite tree with vertices $\Psi_j (v_{-1j}), \Psi_j (v_{0j}), \Psi_j (v_{1j})$, $j=1, \cdots, n$ is a star $\star_1$ consisting of a single vertex $v_0=
	\Psi_j (v_{0j})$ (for all $j$) adjacent to $2n$ leaves $\Psi_j (v_{-1j}), \Psi_j (v_{1j})$, $j=1, \cdots, n$. Since 
	\begin{enumerate}
	\item each $\gamma_j$ is $L-$tight, each sequence $\Psi_j (v_{-1j}), \Psi_j (v_{0j}), \Psi_j (v_{1j})$ is $L-$tight for $j=1, \cdots, n$.
	\item each $\Psi_j$ is an $(L,R)-$large rotation any triple $w_{-1}, v_0, w_1$
	with $w_{-1} \in \{\Psi_k (v_{-1k}),  \Psi_k (v_{1k})\}, 
	w_{1} \in \{\Psi_l (v_{-1l}),  \Psi_l (v_{1l})\}$ with $k \neq l$ is  $L-$tight.
	\end{enumerate}
Let $\beta_j =[v_{0j}, v_{f_jj}]\subset \gamma_j$ denote a fundamental domain for the action of the cyclic group $\langle \Phi_j \rangle$ on $\gamma_j$. Then 
$$\star_e:= \bigcup_j \Psi_j (\Phi_j^{-1} (\beta_j) \cup \beta_j)$$ again has a distinguished
central vertex $v_0$ and $L-$tight geodesic segments corresponding to 
$\Psi_j (\Phi_j^{-1} (\beta_j))$ and $\Psi_j (\beta_j)$ emanating from $v_0$. 
We refer to $\star_e$ as the \emph{extended star about $v_0$}. Since $\star_1$ is
$L-$tight as also each $\gamma_j$ (and hence each $\Psi_j(\gamma_j)$), the 
extended star $\star_e$  is also $L-$tight. Translating $\star_e$ by the group $Q$, 
we observe that $L-$tightness is preserved. (It suffices to check this at each terminal vertex of $\star_e$, where checking under the generators of $Q$ suffices). By Proposition \ref{isometrictighttree}, this  furnishes an isometrically embedded
tree $T$ in $\CC(S)$. Hence,
	   $Q$ is a convex cocompact free subgroup of $MCG(S)$ of rank $n$ preserving 
	   the isometrically embedded $L-$tight tree $T$. 
	   
	   We now invoke the hypothesis that $\gamma_j/\Phi_j$ has at least 3 vertices
	   and prove $R'-$thickness for some $R'$ depending on $R$ and $S$.
	   Two cases arise:
	   \begin{enumerate}
	   \item[Case 1:] vertices of valence $2n$,
	   \item[Case 2:] vertices of valence $2$.\\
	   \end{enumerate}
	   
	  \noindent {\it Case 1:} By $Q-$invariance, it suffices to check $R-$thickness for $\star_e$.  {By Remark \ref{rmk-local}, the Bounded Geodesic Image theorem yields a universal constant $M$, such that } $(R+M)-$thickness of $\star_e$ follows from $R-$thickness of the $\Psi_j (\gamma_j)$'s and the hypothesis that $\Psi_j$'s are $(L,R)-$large rotations about $v_0$ for $\gamma_j, \,j=1, \cdots , n$.\\
	  
	   { \noindent {\it Case 2:} Let $v$ be a vertex of valence $2$.  By $Q-$invariance, we may assume that $v$ lies on $\gamma_j$
	    strictly between $v_{0j}
	    $ and $\Phi_j(v_{0j})$ for some $j$. By Remark \ref{rmk-local}, it suffices to check
	   $R-$thickness for a pair of distinct vertices $v_-, v_+$ on $T \subset \CC(S)$ at 
	   distance at most 2 from $v$. At least one of these vertices must lie on 
	   $\gamma_j$ since $\gamma_j/\Phi_j$ has at least 3 vertices.
	   If both vertices are on $\gamma_j$, then $R-$thickness follows from
	   thickness of $\gamma_j$. Else, without loss of generality, assume that
	   \begin{enumerate}
	   \item  $v_- \in \gamma_j$, $v_+ \in \gamma_i$ with $i \neq j$, and
	   \item $d_{\CC(S)} (v, v_0) =1$.
	   \end{enumerate}
We can assume further that $v_-=w_{-3}, w_{-2} , w_{-1}=v, v_0, w_1=v_+$ are  5 vertices in order with $w_{-3}, w_{-2} , w_{-1}=v, v_0$ on $\gamma_j$, and 
$v_0, w_1$ on $\gamma_i$.
	   We know from the \emph{last} hypothesis on $(L,R)-$large rotations about $v_0$, that
	   for any proper essential subsurface $W\subset S \setminus v$,
	   $d_W(z, w_1) \leq R$ for $z =w_{-3}, w_{-2}$. $R-$thickness at $v$ follows.}

	The last statement now follows from Theorem \ref{maintech}.
\end{proof}

{
\begin{rmk}\label{rmk-nonreg}
Note that $T$ is not regular. This is because $v_0$ has valence $2n$, while every
other vertex in $\star_e$ that is not a terminal vertex has valence 2.
\end{rmk}}

We now have the following application of Theorem \ref{maintech}. It gives one of the main new examples of cubulable groups provided by this paper:

\begin{theorem}\label{maineg1} Let $Q$ be a convex cocompact free subgroup of $MCG(S)$ of rank $m$ as in Proposition \ref{ttforfn} above.
	Let $$1 \to \pi_1(S) \to G \to Q_m \to 1$$ be the induced exact sequence of hyperbolic groups. Then $G$ admits a quasiconvex hierarchy and  is cubulable and virtually special. 
\end{theorem}

\begin{proof}Proposition \ref{ttforfn} above furnishes
	the existence of an EIQ track $\TT$. Theorem \ref{eiqimpliescube} now shows that $G$ admits a quasiconvex hierarchy. Hence, by
	Wise's	
	Theorem \ref{wise-hierarchy}, $G$   is cubulable and virtually special.
\end{proof}

{\subsection{An example}\label{sec-eg} To conclude this section, we shall give an example satisfying the hypotheses of 
Proposition~\ref{ttforfn}. The idea is best explicated in the simplest case with $n=1$,
so that there exist
\begin{enumerate}
\item  one $L-$tight $R-$thick geodesic $\gamma$ stabilized by
a pseudo-anosov $\Phi$,
\item  and one rotation $\Psi$ about $v_0 \in \gamma$. 
\end{enumerate}
The crucial hypothesis to check is that of the $(L,R)-$large rotation $\Psi$. While
the construction is quite direct, the proof is indirect and invokes model manifolds \cite{minsky-elc1} again.

Let $\gamma$ be an $L-$tight $R-$thick geodesic stabilized by a pseudo-anosov $\Phi=\Phi(p_1,p_2)$ constructed as in Lemma~\ref{lemma-tg}. Recall that 
$\Phi(p_1, p_2) =   \overline{\psi_2^{p_2}}.\overline{\psi_1^{p_1}}$, and we assume without loss of generality that $\psi_1$ is a pseudo-anosov on $S\setminus v_0$. Let $\ell_{v_0}$ denote the  tree-link of $v_0$ in $\gamma$, i.e.\ when $v_0$ is thought of as a vertex of the tight geodesic
$\gamma$. 

Recall now from Lemma~\ref{lem-tlink} the 
construction of the tree-link $T_{v_0}$ of $v_0$. We describe the construction of 
$T_{v_0}$ in the present special case.
 Let the restriction $\psi_1|(S\setminus v_0)$ be denoted by $\theta_1$.
Note that $\ell_{v_0}$ is a subsegment of a bi-infinite tight geodesic $\ell \subset \CC (S\setminus v_0)$ stabilized under the action of  $\theta_1$ on $\CC (S\setminus v_0)$.
 Then,  $T_{v_0}$ is a tree approximating  the convex hull of $\ell_{v_0} \cup \Psi (\ell_{v_0})$ by Lemma~\ref{lem-tlink}. There are four terminal vertices
 of $T_{v_0}$:
 \begin{enumerate}
 \item the image under $\P_v$ (see Lemma~\ref{lem-tlink}) of the terminal vertices $\ell_{v_0}^\pm$ of $\ell_{v_0}$ that we denote
 by  $\ell_{0}^\pm$ for convenience. Let $\ell_0$ denote the geodesic in 
 $T_{v_0}$ joining these two points.
 \item the image under $\P_v$ (see Lemma~\ref{lem-tlink}) of the terminal vertices $\Psi (\ell_{v_0})^\pm$ of $\Psi (\ell_{v_0})$ that we denote
 by  $\ell_{1}^\pm$ for convenience. Let $\ell_1$ denote the geodesic in 
 $T_{v_0}$ joining these two points.
 \end{enumerate}
 
 As usual, there are two possibilities that could arise:
 \begin{enumerate}
 \item $\ell_0 \cap \ell_1 \neq \emptyset$, where $\ell_0, \ell_1$ \emph{cross}.
 In this case, let $\ell_0 \cap \ell_1 = \ell_X$.
 \item $\ell_0 \cap \ell_1 = \emptyset$, where $\ell_0, \ell_1$ \emph{do not cross}. In this case, let $\ell_X$ denote the shortest geodesic in $T_{v_0}$ joining $\ell_0, \ell_1 $.
 \end{enumerate}

\noindent {\bf Basic condition:} Let $\theta_2$ denote the conjugate of $\theta_1$ by
the restriction $\Psi|(S\setminus v_0)$, so that $\theta_2$ stabilizes the 
tight geodesic $\Psi(\ell) \subset \CC (S\setminus v_0)$.

\begin{condn}\label{cond-coco}
We choose 
$\Psi$ and $p_1$ such that the subgroup  $\Theta = \langle \theta_1^{p_1}, \theta_2^{p_1} \rangle$ generated by 
$\theta_1^{p_1}, \theta_2^{p_1}$ is a convex cocompact subgroup isomorphic to the free group
$\mathbb{F}_2$ freely generated by $\theta_1^{p_1}, \theta_2^{p_1}$.
\end{condn}

Note that any $\Psi$  not lying in the normalizer of $\theta_1$ conjugates it to 
a pseudo-anosov $\theta_2$ with different attracting and repelling laminations.
Hence for all $|p_1|$ sufficiently large, Condition~\ref{cond-coco} is satisfied
as the group generated by a pair of pseudo-anosov homeomorphisms with distinct 
attracting and repelling laminations satisfy the Tits' alternative, and large enough powers $\theta_1^{p_1}, \theta_2^{p_1}$ of $\theta_1, \theta_2$ give the
required convex cocompact subgroup (see \cite{kl-coco,farb-coco,hamen} for 
background on convex cocompact subgroups).

Such a choice of $p_1$ is consistent with Proposition~\ref{qchierin3m}. In the argument below, we might 
need to choose $|p_1|$ still larger. We remind the reader of the notion of split geometry from
Definition \ref{def-splsplit} and Figure~\ref{schematic2},  the notion of a topological building block from Section~\ref{sec-topbb} and the model metric
$\dw$ from Theorem~\ref{model-str}. \\

\noindent {\bf Metric $\dw$ on $ M_{v_0}$:}
 Recall that the topological
building block corresponding to the $T_{v_0}$ is given topologically by
$T_{v_0}\times S = M_{v_0}$. Since $\ell, \Psi(\ell) \subset \CC (S\setminus v_0)$
are stabilized by pseudo-anosov homeomorphisms, there exist thick Teichm\"uller
geodesics  $\ell_T, \Psi(\ell_T) \subset \Teich (S\setminus v_0)$ having the same ending laminations as $\ell, \Psi(\ell) \subset \CC (S\setminus v_0)$ respectively.

Hence, there exist $R_1$ such that for any essential proper subsurface 
$W \subset  (S\setminus v_0)$, the subsurface projections $d_W(\tau_1, \tau_2)
\leq R_1$, where $\tau_1, \tau_2$ are the markings corresponding to any two 
of the points $\ell_0^\pm, \ell_1^\pm$. Note that $\ell_0^\pm, \ell_1^\pm$ are a priori defined as points in $\CC (S\setminus v_0)$, but since 
$\ell_T, \Psi(\ell_T) \subset \Teich (S\setminus v_0)$  are thick Teichm\"uller
geodesics, they give coarsely well-defined points of $ \Teich (S\setminus v_0)$.
Note also that $\ell_0^\pm, \ell_1^\pm$ tend to the  ending laminations giving the ideal end-points of $\ell, \Psi(\ell) $ as $|p_1| \to \infty$. Thus, the 
topological
building block $M_{v_0}$ corresponding to the tree-link $T_{v_0}$ inherits a natural metric given as follows. Let $\RR_{v_0}$ denote the Margulis riser.
Then  $M_{v_0}\setminus \RR_{v_0}$ is 
 bi-Lipschitz homeomorphic to the thick part of the universal curve over $T_{v_0}$ (thought of as embedded in $\Teich (S\setminus v_0)$. The bi-Lipschitz constants 
 are dependent only $R_1$ and the surface $S$.

Let  $\beta \subset T_{v_0}$ be any geodesic segment joining any two of 
$\ell_0^\pm, \ell_1^\pm$. Let $(M_\beta,\dw)$ be the metric on  
the special split block $S \times \beta$ given by the restriction of the bundle 
$M_{v_0}$ to $\beta$. Then $M_\beta \cap \RR_{v_0}$ is the standard annulus
$A_\beta$
in the split block $(M_\beta,\dw)$. Further, 
the injectivity radius of $(M_\beta \setminus A_\beta)$ is bounded below in terms of $R_1$ and the surface $S$.
\\

\noindent {\bf Metric  associated to rays $\gamma \setminus v_0$, and 
	$\Psi(\gamma \setminus v_0)$:} \\ Let $\cdots,
v_{-1}, v_0, v_1, \cdots$ be the tight sequence of vertices defining the tight geodesic $\gamma$.  By hypothesis, there exist $L, R$ such that $\gamma$ is $L-$tight, $R-$thick. Let
$(M_\gamma,\dw)$ denote the special split geometry model given by Theorem~\ref{model-str}. Then there exists a special split block 
$(M_{\gamma,v_0},\dw)$ corresponding to
$v_0 \in \gamma$ such that the markings on (the two components of) its boundary $\partial M_{\gamma,v_0}$
are at a uniformly bounded distance from $\ell_0^\pm$ independent of $p_1$. This follows from the fact that $\bbar{\psi_1^{p_1}}$ have been chosen to be renormalized in Lemma~\ref{lemma-tg}.
Then $(M_\gamma \setminus \, {\rm Int}(M_{\gamma,v_0}) )$ consists of two model
manifolds $M_\gamma^\pm$, where $\partial M_\gamma^+$ (resp.\ $\partial M_\gamma^-$) has a marking at a uniformly bounded distance from $\ell_0^+$
(resp.\ $\ell_0^-$) independent of $p_1$.

A similar argument shows that 
\begin{enumerate}
\item there exists a special split block 
$(M_{\Psi(\gamma),v_0},\dw)$ corresponding to
$v_0 \in \gamma$ such that the markings on (the two components of) its boundary $\partial M_{\Psi(\gamma),v_0}$
are at a uniformly bounded distance from $\ell_1^\pm$ independent of $p_1$. 
\item $(M_{\Psi(\gamma)} \setminus \, {\rm Int}(M_{\Psi(\gamma),v_0}) )$ consists of two model
manifolds $M_{\Psi(\gamma)}^\pm$, where $\partial M_{\Psi(\gamma)}^+$ (resp.\ $\partial M_{\Psi(\gamma)}^-$) has a marking at a uniformly bounded distance from $\ell_1^+$
(resp.\ $\ell_1^-$)  independent of $p_1$.
\end{enumerate}

We are finally in a position to prove:

\begin{prop}\label{prop-lrrotexist}
Let $N \in \natls$ be such that $\Phi, \Psi$ satisfy Condition~\ref{cond-coco} for all 
$|p_1| \geq N$. Then there exists $M \geq N$. $L \geq 3, R \geq 1$ such that for all $|p_1| \geq M$,  $\Psi$ is an $(L,R)-$large rotation about $v_0$.
\end{prop}

\begin{proof}
Let $M_\gamma^\pm$, $M_{\Psi(\gamma)}^\pm$ be as above. For any pair $M^+,
M^- \in \{M_\gamma^\pm, M_{\Psi(\gamma)}^\pm\}$, there exists $\beta\subset T_v$
such that the markings on the boundary components $M_\beta^+$, $M_\beta^-$ 
comprising $\partial M_\beta$ lie at a bounded distance from the markings on
$\partial M^+,
\partial M^-$ respectively, where the bounds depend on $R_1$ defined above
and the surface $S$.

Thus, the union $M^+ \cup  M_\beta \cup M^-$ admits a combinatorial split geometry model, such that away from the Margulis risers, the injectivity radius is bounded below. 
Hence, all hierarchy paths supported on proper essential subsurfaces of $S\setminus
v$ for $v$ a vertex of $\gamma \cup \Psi(\gamma)$ (including $v_0$) are uniformly
bounded by some $R$ as required by Definition~\ref{def-rthick}. 

In the proof of Proposition~\ref{ttforfn}, we have already 
established   the existence of $M$ such that $\gamma \cup \Psi(\gamma)$
is  $L-$tight   for all $|p_1| \geq M$. The Proposition follows.
\end{proof}

\begin{rmk}\label{rmk-resolution}
The model geometry of a doubly degenerate 3-manifold contains the information of a hierarchy path \emph{along with a resolution of the hierarchy} in the sense of \cite{minsky-elc1}. In the proof above, we have implicitly used this 
information about the resolution by extracting it from the model geometries of
$M_\gamma$, $M_{\Psi(\gamma)}$. What the above proof does is `splice' two such resolutions together along a hierarchy path for $\beta$. 
\end{rmk}
 }
 \section{Quasiconvexity of treads}\label{sec-treads} For the purposes of this section, we fix a closed surface $S$ of genus at least 2. The main theorem of this section is the following:

 \begin{restatable}{theorem}{treadsunifqc}\label{treadsunifqc}
	Given $R>0$, $V_0 \in \natls$, there exists $C \geq 0$ such that the following holds.\\
	Let $T$ be an $L-$tight $R-$thick tree of homologous non-separating curves in $\ccs$ with valence at most $V_0$ and $L\geq 3$.
	 Let  ${\ttt}$  be a tree-stairstep (cf.\ Definition \ref{def-tstairstep}) in $M_T$. Let $v,w$ be a pair of adjacent vertices of $T$,  $S_{vw}$ be the corresponding mid-surface, and $\tr_{vw}$ be a tread with boundary  $i(v)\cup i(w)$. Then
	any elevation $\ttrvw$ is $C-$quasiconvex in $\tmtdt$.
\end{restatable}

The proof of Theorem \ref{treadsunifqc} will occupy this section. We will need to recall  some technology first.
\subsection{Laminations, Cannon-Thurston Maps and quasiconvexity}\label{sec-qcct}
\begin{defn} Let $H$ be a hyperbolic subgroup of a hyperbolic group $G$. 
	Let $\Gamma_H, \Gamma_G$ denote  Cayley graphs of $H, G$ with respect to finite generating sets. Assuming that  the
	generating set of $G$ contains the generating set of $H$, let
	$i : \Gamma_H \rightarrow \Gamma_G$  denote the inclusion
	map. Let $\widehat{\Gamma_H}$ and
	$\widehat{\Gamma_G}$ denote the Gromov compactifications of $\Gamma_H, \Gamma_G$.

	A {\bf Cannon-Thurston map} for the pair
	$(H,G)$ is a map 
	$\hat{i} : \widehat{\Gamma_H} \rightarrow \widehat{\Gamma_G}$ 
	which is a continuous
	extension of $i$. \end{defn}

We shall denote the Gromov boundaries of $H, G$ by  $\partial{H}$,
$\partial{G}$  respectively. Note that these are independent of the choice of finite generating sets.

\begin{theorem}\cite[p. 527]{mitra-ct}\label{ctexist}
	Let $G$ be a hyperbolic group and let $H$ be a hyperbolic normal subgroup
	of $G$. Let 
	$i : \Gamma_H\rightarrow\Gamma_G$ be the inclusion map.
	Then a Cannon-Thurston map exists for the pair $(H,G)$, i.e.\ $i$ extends to a continuous
	map $\hat{i}$ from
	$\widehat{\Gamma_H}$ to $\widehat{\Gamma_G}$.
\end{theorem}

\begin{defn}
	An {\bf algebraic lamination} \cite{bfh-lam, chl07,  kl10, kl15, mitra-endlam} for a hyperbolic group $H$ is an
	$H$-invariant, flip invariant, closed subset $\LL \subseteq \partial^2
	H =(\partial H \times \partial H \setminus \Delta)$, where $(x,y)\to (y,x)$ is called the flip, and $\Delta$ is the diagonal in $\partial H \times \partial H$.
\end{defn}

\begin{defn} Suppose that a Cannon-Thurston map exists for the pair $(H,G)$.
	Let
	$\Lambda_{CT} = 
	\{ (p,q) \in \partial^2{H} \mid
	\hat{i} (p) = \hat{i} (q) \}$.
	It is easy to see that $\Lambda_{CT}$ is an algebraic lamination. We call it the
	\defstyle{Cannon-Thurston lamination}.
\end{defn}

\begin{lemma}\cite[Lemma 2.1]{mitra-pams}\label{qccrit} Let $G$ be a hyperbolic group and let $H$ be a hyperbolic  subgroup. Then
	$H$ is quasiconvex in $G$ if and only if a  Cannon-Thurston map exists for the pair $(H,G)$ and $\Lambda_{CT}  = \emptyset$.
\end{lemma}

\begin{theorem}\label{kldmr1}\cite[Theorem 1.3]{kld-coco}	Let $$ 1 \to H \to G \to Q \to 1$$ be an exact sequence of hyperbolic groups, where $H$ is  a  surface group.
	Let $L$ be  a finitely generated  infinite index subgroup  of $H$. Then $L$ is quasiconvex in $G$. 
\end{theorem}

For a convex cocompact subgroup $Q$ of $MCG(S)$, it was shown in \cite{farb-coco,kl-coco} that $\partial Q$ embeds canonically in the projectivized measured lamination space $\PML(S) = \partial \Teich(S)$ \cite[Theorem 1.1]{farb-coco} and also in the ending lamination space $\EL(S)=\partial \CC(S)$
\cite[Theorem 1.3]{kl-coco}. The associated  map from $\partial Q$, thought of as a subset of $ \PML(S)$, to $\partial Q$, thought of as a subset of $  \EL(S)$, simply forgets the measure. Thus, $\partial Q \subset \EL(S)$ parametrizes a family of ending laminations. The ending lamination corresponding to $z \in \partial Q$ will be denoted as $\Lambda_z$.
\begin{theorem}\label{ctstr}\cite[Theorem 3.5]{mahan-rafi}
	Let $1\to H \to G \to F_n \to 1$ be
	an exact sequence
	of hyperbolic groups, with $H$ a closed surface group.
	Then a Cannon-Thurston map exists for the pair $(H,G)$ and the Cannon-Thurston lamination is given as a  union of  laminations $\Lambda_z$ parametrized by $\partial F_n$: 
	$$\Lambda_{CT} = \bigcup_{z \in \partial F_n} \Lambda_z,$$
	{where $\Lambda_z$ as  a relation on $\partial H$ is given by the transitive closure of the relation induced by an ending lamination. Thus, $\Lambda_z$ consists of an ending lamination along with all diagonal leaves in any complementary ideal polygon.}
\end{theorem}

\subsection{Uniform quasiconvexity of treads in  split geometry manifolds}  Let  $l$ be  an $L-$tight $R-$thick tree whose underlying topological space is homeomorphic to $\R$.
Recall from Section { \ref{sec-splitgeo} }  that   the associated doubly degenerate 3-manifold  $N_l$ is of special split geometry. As in Section \ref{sec-geolts}, let $\{ B_i \}$ denote split blocks, $S_i = B_{i-1} \cap B_i$ denote split surfaces, and $\tau_i(+), \tau_{i}(-)$ be the geodesics on $S_i$ corresponding to the core curves of the splitting tubes $\T_i, \T_{i-1}$. For convenience of notation, we shall also refer to the splitting tubes $\T_i, \T_{i-1}$ as 
$\T_i(+), \T_{i}(-) $ respectively. Also  recall that there exists $D \geq 1$ such that any split surface is of $D-$bounded geometry.

Let $S_{c,i} \subset S_i \setminus
(\tau_i(+) \cup \tau_{i}(-))$ denote a component of the surface $S_i$ cut open along the curves  $\tau_i(+), \tau_{i}(-)$. Let $\til{S_{c,i}} $ be an elevation of ${S_{c,i}} $ to the universal cover $\til N_l$. Let $\SSSS_i \subset \til N_l$ denote the set obtained by adjoining to
$\til{S_{c,i}} $ all the elevations of $ \T_i (+)$ and  $ \T_{i}(-)$ that abut $\til{S_{c,i}} $. Recall that $l_i$ is the `height' of the standard annulus  in the welded split block $B_i$. 
By the construction of tree-links, $l_i$ is approximately equal to $d_{\CC(S\setminus v_i)} (v_{i-1},v_{i+1})$ (up to a constant depending on $R$ alone). 

{It will be helpful to look at the cover $N_l(i)$ of $N_l$ corresponding to $\pi_1(S_{c,i})$. Then there is a distinguished elevation 
	$\widehat{S_{c,i}} $ of 
${S_{c,i}} $ in  $N_l(i)$, inducing an isomorphism of fundamental groups. Further, there exist two elevations each of $ \T_i (+)$ and  $ \T_{i}(-)$ that abut 	$\widehat{S_{c,i}} $ (since the curves $\tau_i(\pm)$ are non-separating). Let $\SSS_i$ denote the set in  $N_l(i)$ obtained by adjoining these elevations of $ \T_i (+)$ and  $ \T_{i}(-)$ to $\widehat{S_{c,i}} $. Then $\SSS_i$ is the image in  $N_l(i)$ of the unique $\pi_1(S_{c,i})-$invariant translate of $\SSSS_i $ in $\til N_l$. Henceforth, we shall refer to this particular translate of $\SSSS_i $ in $\til N_l$ as the base elevation 
$\SSSS_i $. Further, let $CH(\SSSS_i)$ denote the convex hull of the limit set of  $\SSSS_i $ in $\til N_l$,
and $CC(\SSS_i)$ denote the convex core of  $N_l(i)$. Note that  $CC(\SSS_i)$  is the quotient of 
$CH(\SSSS_i)$ under $\pi_1(S_{c,i})$.}

\begin{rmk}[Dependence of constants]\label{rmk-dep}
  Before we state the next Lemma, we briefly recount for the convenience of the reader, the implicit dependence of constants involved. Uniformity of the quasiconvexity constant $C$ in Lemma \ref{qcMl} below is crucial in our argument. The statement of Lemma \ref{qcMl} shows that it depends on $R$  (the parameter determining $R-$thickness). It also depends on the genus of the surface $S$; but this has been fixed at the outset. In the final proof of Theorem \ref{treadsunifqc},
    the quasiconvexity constant $C$ will depend also on the valence of the tree $T$.
	The constant $C$ certainly depends on the parameters $D, \ep$ of special split geometry in Definition \ref{def-splsplit}; but as shown in Proposition \ref{prop-splsplit}, the  parameters $D, \ep$ depend in turn only on $R$ (and implicitly on the genus of $S$). 
\end{rmk}
We shall prove that

\begin{lemma}\label{qcMl} Given $R>0$, there exists $C \geq 0$ such that the following holds.\\
	Let $l$ be an $L-$tight $R-$thick tree of homologous non-separating curves in $\ccs$ whose underlying topological space is homeomorphic to $\R$ and let $N_l$ be the corresponding doubly degenerate manifold of special split geometry. Let $\SSSS_i$ be as above.
	Then, for all $i$, $ \SSSS_i$ is $C-$quasiconvex in $\til N_l$.
\end{lemma}

\begin{proof} Note that quasiconvexity of $\SSSS_i$ was already known by Theorem \ref{kldmr1}. The effective dependence of $C$ on $R$ is what we establish now.
		We argue by contradiction.
	Suppose that the statement is not true. 
	We carry out a geometric limit argument { and use Lemma \ref{lem-cutsurfacegf} to arrive at a contradiction. }
	
	Let $\{(N_m, x_m)\}$ be  a sequence of worse and worse counterexamples, where we assume that $x_m$ lies on a  split surface $S(m) \subset N_m$ (we use the notation $S(m)$ to distinguish from a sequence $S_i$ of split surfaces exiting the end of a fixed $N_m$). {More precisely, after passing to a subsequence if necessary, we can assume that for every $m$, there exists $w_m \in CH(\SSSS_m)$, such that the distance between $w_m$ and $\SSSS_m$ in $N_m$ is at least $m$. Equivalently, there exists $z_m \in CC(\SSS_m)$ such that the distance between $z_m$ and $\SSS_m$ in the cover $N_m(m)$ corresponding to $\pi_1(\SSS_m)$ is at least $m$.  }

	Then, (after passing to a subsequence if necessary), a geometric limit  $N_\infty$ exists, $S(m)$ converges to a split surface $S(\infty) \subset N_\infty$ (see paragraph following Definition \ref{def-ltsplit} for terminology). Let $\tau(m+), \tau(m-)$ denote
	the distinguished curves on $S(m)$ homotopic to core curves of splitting tubes $\T(m+), \T(m-)$ abutting $S(m)$ (recall that splitting tubes are neighborhoods of Margulis tubes). Let $B(m+), B(m-)$ denote the split blocks of $N_m$ containing $\T(m+), \T(m-)$ respectively. Also, let $l(m+), l(m-)$ denote the heights of $B(m+), B(m-)$ respectively.

	 A connected component  $S_{c} (m)$ of $ S(m) \setminus \tau(m+) \cup \tau(m-)$ converges to a connected subsurface $S_c$ of  $S(\infty)$. 
	Let $\til{S_c}$ denote an elevation  of $S_c$ to $\til N(\infty)$. Let $\T(\infty,+), \, \T(\infty,-)$ denote the geometric limits of $\T(m+), \T(m-)$ respectively in $N_\infty$. Then $\T({\infty,-}), \, \T({\infty,+})$ are either splitting tubes in split blocks or rank one cusps in limiting split blocks (Definition \ref{def-ltsplit}). Construct $\SSSS$ from
	$\til{S_c}$ by adjoining abutting elevations of $\T({\infty,-}), \, \T({\infty,+})$ as in the discussion preceding the Lemma. Note also that the pair of simple  closed curves $\tau(m+), \tau(m-)$  converge to a pair of simple  closed curves $\tau(+), \tau(-)$ on $S(\infty)$ corresponding to the core curves of $\T({\infty,-}), \, \T({\infty,+})$. Also, $\tau (+) \cup \tau(-)$ are the boundary curves of $S_c$. As elements of the curve graph $\ccs$,  $\tau(m+)=\tau (+), \tau(m-)=\tau(-)$ for $m$ large enough. 
	
	{Let $N_\infty(S_c)$ denote the cover of $N_\infty$ corresponding to $\pi_1(S_c)$ and $CC(S_c)$ denote its convex core. There exists  a unique $\pi_1(S_c)-$invariant translate of $\SSSS_\infty$, that we denote
by $\SSSS_\infty$ for convenience. Let 		 
		$\SSS_\infty$
		denote the image of  $\SSSS_\infty$ in $N_\infty(S_c)$ under the covering projection from $\til N_\infty$ to $N_\infty(S_c)$.
		Then $CC(\SSS_m)$
	converges geometrically to $CC(S_c)$. From the second paragraph of this proof, there exist $y_m \in CC(S_c)$ such that the distance from $y_m $ to 	$\SSS_\infty$ is at least $m$ in $N_\infty(S_c)$. Hence, $N_\infty(S_c)$ is either geometrically infinite or contains a contains a cusp corresponding to a simple closed curve $\beta$ other than $\tau (+)$ and $ \tau(-)$. This contradicts Lemma \ref{lem-cutsurfacegf}.}

\end{proof}

Lemma \ref{qcMl} establishes uniform quasiconvexity (see Remark \ref{rmk-dep} for dependence of constants) of the $\SSSS_i$'s in
$\til{N_l}$. We would like now to transfer this quasiconvexity to the universal covers
$(\til{M_l}, \dt)$. Recall that we have fixed a bi-Lipschitz homeomorphism between $N_l^0$ and $M_l^0$ after Theorem \ref{model-str}.
Recall also that in $(M_l,\dw)$, a splitting tube $T$ of  $N_l$ gets replaced by a standard annulus $A$. Further, (see Definition \ref{tubeel}) let $(M_l,\dt)$
 denote the tube-electrified metric obtained by electrifying the $\R-$direction of the universal cover $\til A = \R \times I$ of $A$.
Lemma \ref{weldandgeoltcommute} allows us to pass  between $N_l$ and the corresponding metric surface bundle $(M_l,\dw)$. 
Let $S_{c,i}$ be as in Lemma \ref{qcMl} and the discussion preceding it. {The bi-Lipschitz homeomorphism between $N_l^0$ and $M_l^0$ and Lemma \ref{weldandgeoltcommute} now allow us to transfer $S_{c,i}$  to $S_{c,i,M}$ contained in $M_l$.}
 Let $\til{S_{c,i,M}}$ be an elevation of $S_{c,i,M}$ to $\til{M_l}$. Let $\{ A_{ij}\}$ denote the collection of elevations of  standard annuli abutting  $\til{S_{c,i}}$.
Abusing notation slightly, let $$(\SSSS_i,\dw) = (\til{S_{c,i,M}}\bigcup \cup_j \til A_{ij},\dw), $$
and let 
$$(\SSSS_i,\dt) = (\til{S_{c,i,M}}\bigcup \cup_j \til A_{ij},\dt), $$
denote the corresponding tube-electrified metric.

We shall need a slightly modified version of \cite[Lemma 8.3]{mahan-split} that says that
partial electrification preserves quasiconvexity. Since this argument will appear again in the proof of Theorem \ref{maintech} in Section \ref{sec-qctrack}, we briefly discuss the transfer of information  from  relatively hyperbolicity  to tube-electrified spaces in our context by specializing the general discussion in \cite{mahan-pal}.  The spaces we shall be dealing with are as follows:
\begin{enumerate}
\item $\tmtdw$, which is $\delta-$hyperbolic relative to the collection
  $\tmr$ of elevations of risers to $\til{M_T}$ by Theorem \ref{mainprel}.
\item $\EE(\tmtdw, \tmr)$, which is the electric space obtained from $\tmtdw$ by electrifying the elements of $\tmr$. 
\item $\EE(\tmtdt, \tmr)$, which is the electric space obtained from $\tmtdt$ by electrifying the elements of $\tmr$. 
Note that $\EE(\tmtdw, \tmr)$ is quasi-isometric to $\EE(\tmtdt, \tmr)$.
\end{enumerate}

The following Lemma is now a special case of \cite[Lemmas 1.20, 1.21]{mahan-pal}:
\begin{lemma}\label{lem-alltrack} Given $\delta_1, D_1 > 0$, there exists 
	$\alpha_1 > 0$ such that the following holds. Suppose that $\EE(\tmtdw, \tmr)$ is $\delta_1-$hyperbolic,
	and that any two distinct elements of $\tmr$ are at distance at least
	$D_1$ from each other in $\tmtdw$.
Let $\gamma$ be a geodesic in {$\tmtdt$} joining $a, b$. Let $\gamma_e^1$ be an electric geodesic in $\EE(\tmtdw, \tmr)$ joining $a,b$.  Let $\gamma_e^2$ be an electric geodesic in $\EE(\tmtdt, \tmr)$ joining $a,b$.
Let $\gamma_h$ be a geodesic in  $\tmtdw$ joining $a,b$. Then $\gamma$, $\gamma_e^1$, $\gamma_e^2$ and $\gamma_h$ track each other away from $\tmr$ with  tracking constant $\alpha_1$.
\end{lemma}

Let $\gamma$ and $\gamma_e^1$ be as in Lemma \ref{lem-alltrack}.
We shall now modify $\gamma_e^1$ to a more canonical representative.
 Let $x_v, y_v$ be the entry and exit point of the electric geodesic $\gamma_e^1$ for an elevation $\til{v \times T_v}$ of a Margulis riser. Interpolating a geodesic segment $\eta_v$ in $\til{v \times T_v}$ between $x_v, y_v$, for every  $\til{v \times T_v}$ that $\gamma_e^1$ meets, we obtain a path 
$$\bbar{\gamma} = (\gamma_e^1 \setminus \bigcup_{\tmr} \til{v \times T_v}) \bigcup_v 
\eta_v.$$ where the first union ranges over all elevations of risers in $\tmr$ and the second union ranges over all $v$ such that $\gamma_e^1$ meets the Margulis riser $\til{v \times T_v}$. 

Since $\gamma_e^1$ and $\gamma_h$ track each other away from $\tmr$, we have the following consequence of Lemma \ref{lem-alltrack}:

\begin{cor}\label{cor-alltrack} Given $\delta_1, D_1 > 0$, there exists 
	$\alpha_1 > 0$ such that the following holds. Suppose that $\EE(\tmtdw, \tmr)$ is $\delta_1-$hyperbolic,
	and that any two distinct elements of $\tmr$ are at distance at least
	$D_1$ from each other in $\tmtdw$.
	Let  $\gamma_h$ be as in Lemma \ref{lem-alltrack} and $\bbar{\gamma}$ be as in the discussion preceding the Corollary. Then  $\bbar{\gamma}$ and $\gamma_h$ track each other away from $\tmr$  with  tracking constant $\alpha_1$.
\end{cor}
 Lemma \ref{lem-alltrack} and Corollary \ref{cor-alltrack} above recast the ``Bounded Coset Penetration'' property of relative hyperbolicity \cite{farb-relhyp} in our context.

\begin{lemma}\label{tubelpresqc} Given $R >0$ there exists $\delta \geq 0$, $C>0$ such that the following holds.
	Let $N$ be a doubly degenerate manifold of special split geometry corresponding to an $L-$tight $R-$thick tree with underlying space $\reals$. 
	Let $(M,\dw)$ (resp.\ $(M,\dt)$) denote the corresponding metric surface bundle with the welded (resp.\ tube-electrified) metric. Let $\tmr$ denote the collection of elevations of Margulis risers. Then 
	\begin{enumerate}
	\item $(\til{M},\dt)$ is $\delta-$hyperbolic and $(\til{M},\dw)$ is strongly $\delta-$hyperbolic relative to the collection $\tmr$.
	\item $(\SSSS_i,\dt)$ is $C-$qi-embedded in  $(\til{M},\dt)$.
	\end{enumerate}
\end{lemma}

\begin{proof}
	The first conclusion follows from Theorem \ref{mainprel}.
	The proof of the second conclusion
	follows that of \cite[Lemma 8.3]{mahan-split} and we  indicate the slight modification to the setup we need.
	Let $\TT$ denote the collection of splitting tubes in $N$. 
	Uniform separation of splitting tubes (see Proposition \ref{prop-splsplit}, where the separation constant $d_0$ is shown to depends only on $R$) shows that $\til N$ is strongly hyperbolic relative to the collection $\til \TT$ of elevations of $\T \in \TT$, with constant of relative hyperbolicity depending only on $d_0$ and the constant $D$	
 appearing in Proposition \ref{prop-splsplit} defining the relationship between Margulis tubes and splitting tubes (note that $D$ depends only on $R$).
 
	Let $N_0$ denote $N\setminus \bigcup_{\T \in \TT} \T^0$, where $\T^0$ denotes the interior of $\T$. Let $\til N_0$ denote the elevation of $N_0$ to $\til N$ and let $\partial \til \TT$ denote the collection of elevations of $\partial \T$ for $\T \in \TT$. Equipping $N_0$ with the induced path metric $d_p$, it follows  that $\til N_0$
	is strongly hyperbolic relative to the collection $\partial \til \TT$ (with the same constant as the strong relative hyperbolicity constant of $\til N$ relative to $\til \TT$). Since $(\til M, \dt)$ can be obtained by tube electrification of $(\til N_0, d_p)$,   Lemma \ref{lem-alltrack} and Corollary \ref{cor-alltrack} show that
	geodesics in $\tmtdw$ and $\tmtdt$ track each other away from the elements of
	$\tmr$.
	
	Next observe that $(\SSSS_i,\dt)$ in  $(\til{M},\dt)$ is obtained by the tube-electrification procedure applied to  $\SSSS_i$ in $\til N$. By Lemma \ref{qcMl}, there exists $C_0$ depending on $R$ alone such that $\SSSS_i$  is $C_0-$qi-embedded in $\til N$. Identifying 
	$\til N_0$ and $\til M_0$ via the bi-Lipschitz homeomorphism of Theorem \ref{model-str}, the previous
	paragraph now shows that geodesics in $\til N$ and those in 
	$(\til M, \dt)$ track each other away from the elements of $\til \TT$ and
	$\tmr$.
	 Hence, under tube-electrification,
	  quasigeodesics in $\SSSS_i \subset \til N$ go to
	 quasigeodesics in $(\SSSS_i,\dt)$ proving the second assertion. 
	 {(This last piece of the argument is as in the proof of  \cite[Lemma 8.3]{mahan-split}, where the quasiconvex set $A$ takes the place of $\SSSS_i$, and a partially electrified metric takes the place of $\dt$.) }
\end{proof}

\begin{prop}\label{qcMlel}  Given $R>0$, there exists $C \geq 0$ such that the following holds.\\
	Let $l$ be an $L-$tight $R-$thick tree in $\ccs$ whose underlying topological space is homeomorphic to $\R$ and let $(M_l,\dw)$ be the corresponding metric surface bundle.
	Let $\tmldt$ denote the corresponding tube-electrified metric. Let  $S_{c,i}$ be as above. Then 
	any elevation $\til{S_{c,i}}$ is $C-$quasiconvex in $\tmldt$.
\end{prop}

\begin{proof} During the course of the proof of this proposition, `uniform' will mean `depending only on $R$'.
Observe that in the (pseudo)metric space $\tmldt$, each elevation $(A_{ij},\dt)$ of a standard annulus  is uniformly quasi-isometric to the interval $[0,l_i]$, where $l_i$ is its height.  Since each  $\{ A_{ij}\}$ is uniformly quasi-isometric to an interval, $(\til{S_{c,i}},\dt)$ is uniformly quasiconvex in $(\SSSS_i,\dt)$ (in fact, it is a uniform Lipschitz retract obtained by projecting each  $\til A_{ij} = \R \times [0,l_i]$ onto $\R \times \{0\}$). Hence by the second conclusion of Lemma 
\ref{tubelpresqc},  $(\til{S_{c,i}},\dt)$ is uniformly quasiconvex in  $\tmldt$.
\end{proof}

\subsection{Uniform quasiconvexity of treads in $\tmtdt$: Proof of Theorem \ref{treadsunifqc}} We now turn our attention to $P: \tmtdt \to \but$ with $T$ an $L-$tight $R-$thick tree.

We restate the main theorem of this section for convenience.
\treadsunifqc*

\begin{proof}
	Proposition \ref{qcMlel} shows that there exists $C_0$ such that $\ttrvw$ is $C_0-$quasiconvex in $\tmldt $ for any bi-infinite geodesic $l$ in $\but$  passing through the midpoint vertex $vw$.
	
	The converse direction of Proposition \ref{effectiveqc} now shows that there exists $K$ such that  $\ttrvw$ flares in all directions with parameter $K$. The forward direction of Proposition \ref{effectiveqc} finally shows that there exists $C$ such that  $\ttrvw$ is $C-$quasiconvex in $\tmtdt$.
\end{proof}

\section{Quasiconvexity of the track: Proof of Theorem \ref{maintech}}\label{sec-qctrack}
We are now in a position to prove the main technical Theorem of this paper, Theorem \ref{maintech}, which we restate for convenience:
\maintech*

\begin{proof}
The first conclusion of the theorem follows from 
Theorem \ref{mainprel}. 

We turn now to the second assertion.
Note that any elevation $\tttt$ of 
the  track $\ttt$ to $\tmtdt$ consists of
\begin{itemize}
	\item elevations $\ttrvw$ of treads that are  (uniformly) $C_0-$quasiconvex in $\tmtdt$ by 
	Theorem \ref{treadsunifqc} ($C_0$ depends only on $R$ and $V_0$),
	\item attached to one another via elevations  of  Margulis risers $(\til{v \times T_v}, \dt)$. Note that these are  uniformly
	(depending only on  $V_0$) {quasi-isometrically embedded  copies} of the tree-link $T_v$ since the {circle direction} corresponding to $v$ (and hence its universal cover $\R$) is electrified in the $\dt$ metric. {In particular, the  elevations  of  Margulis risers are uniformly
		(depending only on  $V_0$) quasiconvex.}
\end{itemize}

Note also that by the hypothesis of $L-$tightness, the distance between any two terminal vertices of $T_v$ (for any $v$) is at least $L$.
Thus, we have a collection of  $C_0-$quasiconvex treads meeting elements of $\tmr$ at  large distances (at least $L$, up to an additive constant depending on $V_0$) from each other.
Thus any geodesic $\gamma$ in $\tttt$ is built up of alternating segments of the following types:
\begin{itemize}
	\item geodesics in (the intrinsic metric on) elevations $\ttrvw$   of  treads $\trvw$,
	\item geodesics in Margulis risers $(\til{v \times T_v}, \dt)$ of length at least $L$ (up to an additive constant depending on  $V_0$).
\end{itemize}

Recall
the local-global principle for quasigeodesics in hyperbolic spaces (see {\cite[p. 405]{brid-h} for instance, where local \emph{geodesics} are shown to be global quasigeodesics}): Given $\delta>0$ and $C' \geq 1$, there exists $\Upsilon > 0$ and $C \geq 1$ such that if a parametrized path $\beta$ in a $\delta-$hyperbolic metric space satisfies the property that each subpath of $\beta$ of length $\Upsilon$ is a $C'-$quasigeodesic, then $\beta$ is a $C-$quasigeodesic. {A closely related fact that may be deduced
	from the local-global principle gives sufficient conditions for a broken geodesic to be a global quasigeodesic \cite[Lemma 9.3]{minasyan-thesis}. It says the following: Let $\{x_0, \cdots, x_n\}$ be a sequence of points in
	a $\delta-$hyperbolic $( X,d)$ such that $d(x_{i-1},x_i) \geq C_1$,
$\forall i = 1,\cdots,n$, and $(x_{i-1}, x_{i+1})_{x_i} \leq C_0$ $\forall i = 1,\cdots, n-1$, where $C_0 \geq 14 \delta$, and
	$C_1 > 12(C_0 + \delta)$. Then the
	broken geodesic given by the concatenation $\bigcup_i [x_{i-1},x_i]$ is
	a $C_2(=C_2(C_0, \delta))-$ quasigeodesic  contained in the  $2C_0-$neighborhood  of the
	geodesic segment $[x_{0},x_n]$.}

By Theorem \ref{mainprel} or the first statement of the theorem, 
$\tmtdt$ is $\delta-$hyperbolic. {By Theorem \ref{treadsunifqc}, treads are uniformly quasiconvex in $\tmtdt$.
There exist $K, \ep$ such that for any  tread $\tr_{vw}$, and any abutting riser ${v \times T_v}$, the latter is a $(K, \ep)-$quasi-isometric section from some interval in $\but$
to  $\tmtdt$. Thus, any elevation  $(\til{v \times T_v}, \dt)$ is    $(K, \ep)-$quasi-isometrically embedded in $\tmtdt$. Let $\Pi: \tmtdt \to \but$ represent the projection as before. Then, for any $x \in \til{\tr_{vw}}$ and
$y\in \til{v \times T_v}$  the distance $\dt (x, y)$ is 
comparable to $d_\but (\Pi(x), \Pi(y))$. Hence,
for any triple $x,y, z$ with $x \in \til{\tr_{vw}}$, $y\in \til{v \times T_v}$, 
and $z \in \til{\tr_{vw}} \cap  \til{v \times T_v}$,   the Gromov inner product $(x,y)_z$ is uniformly bounded above.
Hence, (this is slightly stronger than Lemma \ref{tubelpresqc})  any tread along with abutting risers is also $C-$qi-embedded in $\tmtdt$ as is any riser.}
Hence there exists $C'$ depending on $L$ such that  any subpath of $ \gamma$ of length $\Upsilon$ is  contained in 
\begin{enumerate}
\item a tread along with abutting risers, or
\item  a riser.
\end{enumerate}
Hence, by the local-global principle for quasigeodesics
 $\gamma$ is a $C-$quasigeodesic provided $L$ is sufficiently large. 
 This proves the second conclusion of the theorem. \\

We now prove the  third assertion of Theorem \ref{maintech}. Fix an elevation $\tttt$ of $\ttt$. We denote the elevations $\ttrvw$ contained in 
$\tttt$ by $\til{Tr_i}$ where $i$ ranges over some countable set. 
Since any two elevations $\til{Tr_1}, \til{Tr_2}$ lying over distinct mid-point vertices of $\but$ are separated by at least $L$, and since each $\til{Tr_i}$ is simply connected, any closed essential loop $\sigma$ in  $\tttt$ must  have at least one geodesic segment in the elevation  $(\til{v \times T_v}, \dt)$ of  a riser. 
We can thus put $\sigma$ in standard form so that it is a union of geodesic segments in ${Tr_i}$  and geodesic segments in  elevations  of   risers. 
But then the local-global principle for quasigeodesics in hyperbolic spaces  again shows that $\sigma$ is a $C-$quasigeodesic for $L$   sufficiently large; in particular, if $L \gg C$, it cannot begin and end at the same point.
This contradiction proves the third conclusion of the theorem.\\

To prove the last assertion, assume now that there is a uniform upper bound $L_1$ on the diameters of tree-links $T_v$. Observe first that $\tmtdw$ is $\delta-$hyperbolic relative to the  elements of $\tmr$ by
Theorem \ref{mainprel}. The elements of $\tmr$ are now uniformly hyperbolic, since the upper bound $ L_1$  furnishes uniform quasi-isometries of $\til{v\times T_v}$ with $\R$. Hence, 
$\tmtdw$ is hyperbolic by \cite{bowditch-relhyp} (see also \cite[Proposition 2.9]{mahan-relrig}  or \cite[Proposition 6.1]{mahan-dahmani}).

For $a, b \in \tttt$, let $\gamma_h$ be a geodesic in $\tmtdw$ joining $a, b$.  Also, let $\bbar{\gamma}$ be constructed from the geodesic in $\tmtdt$ joining $a, b$ as in the discussion preceding Corollary \ref{cor-alltrack}. From the second assertion of this theorem, it follows that $\tttt$ is $C-$qi-embedded in $\tmtdt$. Hence $\bbar{\gamma}$  lies in a bounded neighborhood of $\tttt$ in $\tmtdt$.
By  Lemma \ref{lem-alltrack} and Corollary \ref{cor-alltrack},
geodesics in $\tmtdw$ and geodesics in $\tmtdt$ track each other away from
elements of $\tmr$; also, $\bbar{\gamma}$ and $\gamma_h$ 
track each other away from
elements of $\tmr$. Using strong relative hyperbolicity of $\tmtdw$ relative to the collection $\tmr$, it follows that
 $\gamma_h$ lies in a bounded neighborhood of $\tttt$ in $\tmtdw$.
This  proves the fourth assertion of the theorem and completes the proof
of Theorem \ref{maintech}. 
\end{proof}

\begin{rmk}
	A part of the fourth assertion of Theorem \ref{maintech}, viz.\ quasiconvexity of $\tttt$ in $\tmtdw$ actually holds (as the proof above shows) without the extra hypothesis of the existence of $L_1$. However, in the absence of such an upper bound $L_1$,  $\tmtdw$ is no longer hyperbolic, but only strongly hyperbolic relative to the collection $\tmr$.
\end{rmk}

\section{Generalization: separating curves}\label{sec-genlzns}

The purpose of this section is to summarize some notions from \cite{mahan-hyp} and extend the discussion in  Section \ref{sec-tt} to allow the possibility of 
multicurves,
as well as
separating curves. {The setup in this section is somewhat more general, and allows for the construction of new examples not covered by  Section \ref{sec-tt}. However, dealing with this level of generality at the outset would have cluttered the exposition considerably. We have therefore opted to only indicate the
modifications necessary in proving Theorem \ref{maintech2} below.}

\subsection{Tight trees}
The collection of complete graphs {with one or more vertices} in the curve graph $\CC(S)$ will be denoted as $\ccd(S)$. (Equivalently, $\ccd(S)$ is the {union of 0-cells of the curve complex along with} the collection of 1-skeleta of simplices in the curve complex. However, since we have only used the curve graph and not the curve complex in this paper we opt for the earlier point of view.) 
Let $\gamma =\{\cdots, v_{-1}, v_0, v_1, \cdots\}$ be a 
geodesic (finite, semi-infinite, or bi-infinite)  in a tree $T$ and  $i: V(T) \to \ccd (S)$ a map.  A path in $\ccs$ {\bf induced by $\gamma$} is a choice of  simple closed curves $\sigma_i \in i(v_i)$.
The map $i$  will be called an \defstyle{isometric embedding} if any path induced in $\ccs$ by a geodesic $\gamma$ in $T$ is a geodesic in $\ccs$.
We now generalize the notion of tight trees of non-separating curves (Definition \ref{def-tighttree}) to multicurves.

\begin{defn}\label{def-tighttree-sep}\cite{mahan-hyp}
	An \defstyle{$L-$tight tree} in the curve graph $\CC(S)$ consists of a (not necessarily regular) simplicial tree $T$ of bounded valence and a  map $i: V(T) \to \ccd (S)$ such that 
	\begin{enumerate}
		\item 
		for every vertex $v$ of $T$,
		$S\setminus i(v)$ consists of exactly one or two components. Further, if $S\setminus i(v)$ consists of  two components and $i(v)$ contains more than one simple closed curve, then each component of $i(v)$ is individually non-separating. If $S\setminus i(v)$ consists of  two components, $v$ is called a \defstyle{separating vertex} of $T$.
		\item for every pair of adjacent vertices $u \neq v$   in $T$, and any vertices $u_0, v_0$ of the simplices $i(u), i(v)$ respectively,
		\[d_{\ccs} (u_0, v_0) =1.\]
		\item There is a \defstyle{distinguished component} $Y_v$ of $S\setminus i(v)$ such that for any vertex $u$ adjacent to $v$ in $T$, $i(u) \subset Y_v$ (automatic if $i(v)$ is non-separating).
		For $i(v)$ separating, we shall refer to $Y_v':=S \setminus Y_v$ as the \defstyle{secondary} component for $v$.  
		\item  for every pair of distinct vertices $u \neq w$ adjacent to $v$ in $T$, and any vertices $u_0, w_0$ of the simplices $i(u), i(w)$ respectively,
		\[d_{\CC(Y_v)} (u_0, w_0) \geq L.\]
	\end{enumerate}	
	An $L-$tight tree for some $L\geq 3$ will simply be called a tight tree.
\end{defn}

We recall the following from \cite{mahan-hyp} due to Bromberg.
\begin{prop}\label{isometrictighttree-sep} There exists $L > 1 $ such that
	if $S$ is a closed surface of genus at least $2$, and  $i: V(T)\to \ccd(S)$ defines an $L$--tight tree,  then $i$ is an isometric embedding.
\end{prop}

In fact, as shown in \cite{mahan-hyp}, $L\geq \max{(2M, 4D)}$ suffices, where $M$ is a constant given by the Bounded Geodesic Image Theorem
\cite{masur-minsky2} and $D$ is the Behrstock constant
{(see \cite{behrstock} and  \cite[Proposition 2.12]{mahan-hyp})}. We shall not need this.

\subsection{Balanced trees} A special class of tight trees will be required to generalize Theorem \ref{maintech}. For $i: V(T) \to \ccd(S)$ a tight-tree, tree-links $T_v$ are defined as in Definition \ref{def-treelink} with the qualifier  that for $v$ a separating vertex, the weak  hull $CH(\ilkv)$ is constructed in the curve graph $\CC(Y_v)$ of the distinguished component $Y_v$ of $S\setminus i(v)$. The ``balanced'' condition we shall introduce now essentially guarantees that for $v$ a separating vertex, $T_v$ serves as the tree-link of the secondary component $Y_v'$ as well. For $v$ a separating vertex, the tree link $T_v$  furnishes, a priori, a way of constructing a model geometry for $Y_v \times T_v$. To extend this to a model geometry for $S \times T_v$, we need the balanced condition below.

For $w$ adjacent to $v$ let $T_w'$ denote the connected component of $T\setminus \{v\}$ containing $w$. Let $\Pi'_v (T_w')$ denote the subsurface projection of $i(V(T_w'))$ onto $\CC(Y_v')$.

\begin{defn}\label{def-balancedtree}\cite[Definition 2.18]{mahan-hyp}
	A tight tree  $i: T \to \CC(S)$ is said to be a \defstyle{balanced} tree with parameters $D,k$ if 
	\begin{enumerate}
		\item For every separating vertex $v$ of $T$, and every adjacent vertex $w$, \[\d(\Pi'_v (T_w')) \leq D.\]
		\item For the secondary component $Y_v'$, let $\sec(v) \subset \CC(Y_v')$ denote the collection of  curves in $ \Pi'_v (T_w')$ as $w$ ranges over all vertices adjacent to $v$ in $T$. Let $CH(\sec(v))$ denote the weak  hull of $\sec(v) $ in $\CC(Y_v')$. We demand that
		there exists a  surjective $k-$quasi-isometry \[\P_v' : CH(\sec(v)) \to T_v\] to the tree-link $T_v$, such that for any vertex $w$ of $T$ adjacent to $v$, \[\P_v'( \Pi'_v (T_w')) = \P_v (w),\] (where $\P_v$ is the projection defined in Definition \ref{def-treelink}).
	\end{enumerate} 
	
\end{defn}

The notions of topological building block $M_v$ (Definition \ref{def-topbb}), blow-up $\but$, and topological model (Definition \ref{def-topmodeltree}) now go through exactly as in Section \ref{sec-topbb}.
Definition \ref{def-balancedtree} guarantees that the weak  hulls $CH(\ilkv) \subset \CC(Y_v)$ and $CH(\sec (v)) \subset \CC(Y_v')$ are coarsely quasi-isometric to each other and to the tree-link $T_v$. {
Finally, in the hypothesis of Theorem \ref{maintech2}, we shall impose the condition 
that that  if $v, u$ are adjacent in $T$, then there exists a subsurface of $S$ with boundary $i(v) \cup i(u)$. These subsurfaces can now be used to construct treads,
and hence a
 tree-stairstep as in  Definition \ref{def-tstairstep}. As before, we denote the
  tree-stairstep corresponding to such a balanced tree by $\TT_T$.}
With these constructions, the proof of Theorem \ref{maintech} goes through  and we have:

\begin{theorem}\label{maintech2}
	Given $R, D, k > 0$, $V_0 \in \natls$ there exists $\delta_0, L_0, C \geq 0$ such that the following holds. Let $i:T \to \ccd(S)$ be an $L-$tight  $R-$thick balanced tree  with parameters $D, k$ such that 
	\begin{enumerate}
	\item  {If $v, u$ are adjacent in $T$, there exists a subsurface of $S$ with boundary $i(v) \cup i(u)$. Let $\TT_T$ be the tree-stairstep constructed out of these subsurfaces (see the discussion preceding the statement of this theorem).}
	\item The valence of any vertex of $T$ is at most $V_0$.
	\end{enumerate} Then, for 
	$L\geq L_0$,
	\begin{enumerate}
		\item $(\til M_T,\dt)$ is $\delta_0-$hyperbolic.
			\item $\tttt$ is $C-$qi-embedded in $(\til{M_T},\dt)$.
		\item $\ttt$ is incompressible in $M_T$, i.e.\ $\pi_1(\ttt)$ injects into $\pi_1(M_T)$.
		\item If in addition there exists $L_1$ such that for every vertex $v$ of $T$ and for every pair of distinct vertices $u\neq w$ adjacent to $v$ in $T$, 
		\[d_{\CC(S \setminus i(v))} (i(u), i(w)) \leq L_1, \] then $\tttt$ is quasiconvex in $(\til{M_T},\dw)$.
	\end{enumerate}
\end{theorem}

Item (1) was proven in {\cite[Theorem 3.36]{mahan-hyp}}.  The proof of Theorem \ref{maintech} in Sections \ref{sec-treads} and \ref{sec-qctrack} goes through mutatis mutandis to establish the remaining conclusions. We briefly refer the reader to the relevant sections in the text:\\
1) For the construction of a topological building block $M_v$ associated to a  tree-link $T_v$, see Definition \ref{def-topbb}. For the blown-up tree $\but$, and the whole topological model $S \times \but$ (Definition \ref{def-topmodeltree})   see Section \ref{sec-topbb}.\\
2) The construction of a tree-stairstep  (Definition \ref{def-tstairstep})  goes through as the metric on the riser $\RR_v$ is coarsely well-defined thanks to
the tree-link $T_v$ being coarsely well-defined.\\
3) We now turn to the proof of quasiconvexity of $\tttt$ in Theorem \ref{maintech2}. Quasiconvexity of treads (Theorem \ref{treadsunifqc}) follows as before.
There are two points to note. First, since the tree-links $T_v$ are now only coarsely well-defined, the parameters $D,k$ of Definition \ref{def-balancedtree} will be involved in Remark \ref{rmk-dep}. Second, the geometric limit argument in Lemma \ref{qcMl} really only used the fact that
no component of the ending laminations of the geometric limit can be contained in the subsurface $S_c$ obtained by cutting $S$ along $\tau(\pm)$.
This allows Theorem \ref{treadsunifqc} to go through.\\
4) Finally, with Theorem \ref{treadsunifqc} in place, Section \ref{sec-qctrack} goes through without change.

\begin{rmk}\label{qchierin3msep}
With Theorem \ref{maintech2} in place, Propositions \ref{qchierin3m} and \ref{ttforfn} now generalize in a straightforward way simply by dropping the hypothesis that the curves are non-separating. We illustrate this first by the construction of a stairstep in a 3-manifold with two treads. Note first that Lemma \ref{lem-rthick} allows for $\sigma$ to be separating. We now let $\sigma_1, \sigma_2$ be a pair of disjoint null-homologous simple closed curves. Let $S\setminus \sigma_i = W_{i1} \cup W_{i2}$. Define $\psi_i : S \to S$ such that  $\psi_i$ is a pseudo-Anosov homeomorphism in the complement of $\sigma_i$ (see Definition \ref{def-renpa} and the conventions therein). 
Then for large enough $p_1, p_2$ the 3-manifold with monodromy  $\overline{\psi_2^{p_2}}.\overline{\psi_1^{p_1}}$ is hyperbolic. Further it admits a stairstep surface $\TT$ with two risers corresponding to $\sigma_1, \sigma_2$ and two treads.  Theorem \ref{maintech2} now shows that $\TT$ is incompressible and geometrically finite. 

To extend this to stairsteps as in Section \ref{sec-treestairstepeg}, it remains only to check the balanced condition. (The large rotations argument is as before.) This needs to be checked  locally. Let $\sigma$ be a separating simple closed curve corresponding to a vertex $v_0 \in \ccs$. 
Let $S \setminus \sigma = W_1 \sqcup W_2$. It suffices  to construct  (for $i=1,2$)
$\Phi_i = \phi_{i1}\sqcup \phi_{i2}$--two pseudo-Anosov homeomorphisms  in the complement of $\sigma$ ensuring the balanced condition. Thus $\phi_{ij}: W_j \to W_j$ are pseudo-Anosov homeomorphisms and we renormalize by Dehn twists about $\sigma$ if necessary to ensure that $\Phi_i$ has uniformly bounded Dehn twists about $\sigma$. Let $\gamma_{ij} \subset \CC(W_j)$ be the axes of $\phi_{ij}$. Translating $\gamma_{ij}$ suitably by a power of an auxiliary pseudo-Anosov homeomorphism of $W_j$, we can ensure that $d_{\CC(W_1)} (\gamma_{11}, \gamma_{21})$ and $d_{\CC(W_2)} (\gamma_{12}, \gamma_{22})$ are comparable:
$$\frac{1}{2} d_{\CC(W_2)} (\gamma_{12}, \gamma_{22}) \leq d_{\CC(W_1)} (\gamma_{11}, \gamma_{21}) \leq d_{\CC(W_2)} (\gamma_{12}, \gamma_{22}).$$
We can also ensure that the shortest geodesic in $\CC(W_j)$ between $\gamma_{1j}, \gamma_{2j}$ realizing $d_{\CC(W_j)} (\gamma_{1j}, \gamma_{2j})$ is thick (since we have used a power of an auxiliary pseudo-Anosov homeomorphism of $W_j$ to translate one away from the other).
Finally, suppose that the translation lengths of $\phi_{ij}$ are within a multiplicative factor $k \geq 2$ of each other. Let $T_v$ denote the (primary) tree-link
constructed for $W_1$ from  $\phi_{11}^{p}, \phi_{21}^{p}$. Assuming that $N\gg p$, $T_v$ looks like the  letter $H$, where the horizontal bar has length approximately $d_{\CC(W_1)} (\gamma_{11}, \gamma_{21})$ and the two verticals have length of the order of $p$ (up to a multiplicative factor $k$). For the secondary subsurface $W_2$, the secondary weak hull $CH(\sec(v))$
constructed for $W_2$ from  $\phi_{12}^{p}, \phi_{22}^{p}$ is then $2k-$bi-Lipschitz homeomorphic to $T_v$ ensuring the balanced condition.
\end{rmk}

\subsection{Virtual algebraic fibering}\label{vfiber}
We say that a group $G$ {\bf virtually algebraically fibers} if {a finite index subgroup of $G$} admits
a surjective homomorphism to $\Z$ with finitely generated kernel.
We recall the following Theorem of Kielak \cite{kielak}:

\begin{theorem}\cite{kielak}\label{kielak}
Let $G$ be cubulable. Then $G$ virtually algebraically fibers if and only if the first $\ell^2-$betti number $\beta_1^{(2)} (G)$ vanishes.
\end{theorem}

As a consequence we have:
\begin{prop}\label{prop-vfiber} Given $R >0$, there exists $L \geq 3$ such that the following holds.
Let $G$ be a hyperbolic group admitting an exact sequence $$1 \to H \to G \to Q \to 1, $$ where $H=\pi_1(S)$ is the fundamental group of a closed surface and $Q$ acts freely and cocompactly by isometries on an $R-$thick,
$L-$tight tree of non-separating homologous curves.
Then $G$ virtually algebraically fibers. A similar statement holds for 
separating homologous curves.
\end{prop}

\begin{proof}
$G$ is cubulable by Proposition \ref{ttforfn} and Remark \ref{qchierin3msep} (the latter for separating curves). Since $G$ contains an infinite index finitely generated normal subgroup $H$, $\beta_1^{(2)} (G)=0$ (see \cite[Th\'eor\`eme 6.8]{gaboriau} or \cite[Theorem 5.12]{thom}).
Hence $G$ virtually algebraically fibers by Theorem \ref{kielak}.
\end{proof}
 
 \begin{rmk}\label{rmk-incoh}
Recent work of Kropholler, Vidussi and Walsh \cite{kw} (see also \cite{kvw}) establishes that $G$ as in
 Proposition \ref{prop-vfiber} is incoherent provided $b_1(G)$ is strictly greater than $b_1(Q)$. We expect that the finitely generated
 normal subgroup of $G$ that Proposition \ref{prop-vfiber} furnishes
 is not finitely presented and hence the hypothesis $b_1(G) > b_1(Q)$
 in  \cite{kw}
 may not be necessary.
  \end{rmk}

\appendix
\section{Tracks and cubulations}\label{sec-app}
\centerline{\textsc{by Jason Manning, Mahan Mj, and Michah Sageev}}

\medskip

The main purpose of this Appendix is to reduce the problem of
cubulating the fundamental group of a hyperbolic surface bundle over a graph (in the sense of Section~\ref{sec-graphs-of-spaces}) to the construction of a track satisfying some conditions.
After providing some
background on  tracks, we show  that in order to cubulate a hyperbolic surface bundle over a graph
it suffices to construct an essential incompressible quasiconvex (EIQ) track (see Definition~\ref{def-eiq} and
Theorem~\ref{eiqimpliescube} below).

\subsection{Tracks}\label{sec-track}
The following definition is not standard, but it is the most useful for our purposes.
\begin{defn}\label{def-track}
	Let $X$ be a {cell complex}.  A closed connected subset $\TT\subset X$ is a \defstyle{track} if there is a closed $I$--bundle neighborhood $N$ of $\TT$ in $X$. Thus, there is a deformation retraction $N\to \TT$ whose point-preimages are intervals. A  closed (not necessarily connected) subset $\TT\subset X$ is a \defstyle{pattern} if there is a closed $I$--bundle neighborhood $N$ of $\TT$ in $X$.
	
	A track is said to be \defstyle{$2$-sided} if the $I$--bundle structure is trivial; otherwise it is said to be \defstyle{$1$-sided}.
	Note that a	 $2$-sided track $\TT$ has a neighborhood homeomorphic to $\TT\times [0,1]$.
	
	{We now specialize to the surface bundle case. Let $\GG$ be a graph without self-loops (this can always be arranged by subdividing 
		edges of graphs if necessary).
		Let $\Pi:X\to \mathcal{G}$ be a surface bundle over $\GG$. Note that 	every closed edge $e_j$ with the induced topology from $\GG$ is homeomorphic to a closed interval $I_j$. Hence,  $\Pi^{-1}(e_j)$ is homeomorphic to $S \times I_j$.
		 A track $\TT\subset X$ is a \defstyle{surface bundle track} if  {$\TT\cap \Pi^{-1}(e_j)$} is a properly embedded surface in $\Pi^{-1}(e_j)$ for every closed edge $e_j$.
			}
	\end{defn} 

{
	Whenever $v\in \VV(\GG)$, we refer to the surfaces of the
	form $\Pi^{-1} (v)$ as {\bf vertex fibers.} Thus if a vertex fiber is contained in $ \Pi^{-1}(e_j)$, for some $i \in J$, then  
	 it is necessarily a boundary component of   $ \Pi^{-1}(e_j)$.
}

\begin{rmk}(Historical remarks)
  Traditionally, the setting of tracks is simplicial complexes.  {In that setting a track is a connected subset of a simplicial complex whose intersection with each closed simplex is equal to the intersection of that simplex with a 
  a disjoint union of finitely many hyperplanes missing the vertices.
  (We think of each $n$-simplex as a subset of some copy of $\R^{n+1}$.)}  In particular, a track does not intersect the $0$-skeleton, intersects each edge in a finite set of points and intersects each 2-simplex in a finite disjoint union of arcs, each of which has its endpoints in the 1-skeleton.  We will refer to a track in this sense as a \defstyle{combinatorial track}.

	(Combinatorial) tracks were introduced  by Dunwoody \cite{Dunwoody} for 2-complexes and subsequently studied by him and others (see, for example,  \cite{bowditch-cutpts,delzant,dunwoodysageev,ss-ast,fujiwarapapa}). In the case that $X$ is a 3-dimensional manifold, combinatorial tracks are known as normal surfaces, {and have been studied extensively over the past century} (see \cite{kent-normal} for a survey).
	
	A track in our sense is more general, even in the presence of a simplicial structure.  For example, the intersection with a simplex need not be composed of contractible pieces.
	
\end{rmk}

Tracks are not always $\pi_1$-injective, but the aim in the main part of this paper is to construct ones that are. 
In that setting, $X$ is a surface bundle over a graph, which decomposes as a union of copies of $S_g\times I$ glued together along their boundaries.  The intersection of a track in $X$ with a copy of $S_g\times I\subset X$ is necessarily a properly embedded surface.  For the track to be $2$-sided it is necessary but not sufficient that these surfaces are $2$-sided.
If we can show a track is both $2$-sided and $\pi_1$--injective, we obtain a decomposition of $\pi_1(X)$ as a free product with amalgamation or HNN-extension.
The following lemma is exploited throughout the work.

\begin{lemma}[Dunwoody]\label{lem:dunwoody}
	If $\tau$ is a {$2$-sided} track in a simply connected {cell-complex} $X$, then $X\setminus\tau$ has two path components. 
\end{lemma}
\begin{proof}
	(See \cite{DicksDunwoody} for a proof in the case of $2$-dimensional simplicial complexes.  The proof in our setting is the same as Dicks--Dunwoody's.)
	Since $X$ and $\tau$ are path-connected and $\tau$ has a product neighborhood, it is clear that $X\setminus \tau$ has at most two path components.  Suppose there is only one.
	
	There is a closed neighborhood $N$ of $\tau$ and a homeomorphism $\phi: \tau\times[0,1]\to N$.  By shrinking $N$ if necessary we can assume that $X\setminus \mathrm{Int}(N)$ is path-connected.  Projection onto $[0,1]$ and identifying the endpoints gives a continuous surjection $\pi:X\to S^1$.
	
	Fix $x\in \tau$, and let $\sigma_1$ be the path $\phi(x\times[0,1])$.  Since $X\setminus \mathrm{Int}(N)$ is path connected, there is a path $\sigma_2$ joining the endpoints of $\sigma_1$ in the complement of $N$.  Putting $\sigma_1$ and $\sigma_2$ together we get a map $\psi: S^1\to X$.  The composition $\pi\circ\psi:S^1\to S^1$ has degree one, so $\pi_1X$ surjects $\Z$, in contradiction to the assumption that $X$ was simply connected.
\end{proof}

\subsection{Essentiality}

A track is said to be \defstyle{inessential} if the associated graph of groups decomposition is a trivial splitting as a free product with amalgamation: namely $G=A*_C B $ where one of the maps $C\to A$ or $C\to B$ is an isomorphism.  {(If the track $\tau$ is $2$--sided the edge group $C$ is $\pi_1\tau$; if it is $1$--sided then $C$ is the fundamental group of the $2$--sided double cover which forms the boundary of an $I$--bundle neighborhood of $\tau$.)}
The reason such a splitting is considered trivial is that every group has one. Namely if $G$ is any group and $H$ is any subgroup then $G\cong G*_H H$. Consequently it does not provide any new information about the group (again, see Scott and Wall \cite{scott-wall}). 

In our setting, $\tau$ will be a track in a compact space, namely a  surface bundle $X$ over a {finite} graph. In this proper, cocompact setting, essentiality will correspond to the following: every elevation of $\tau$ to the universal cover $\til  X$ will separate $\til  X$ into two components each of which contains points arbitrarily far away from $\tau$. (Here for the metric we may use any proper $G$-equivariant metric on $X$.) Such components are referred to as {\bf deep components}.

\begin{defn}\label{def-eiq}
	Let $X$ be a compact complex with hyperbolic fundamental group.
	A track $\TT\subset X$ is called an an \defstyle{essential, incompressible quasiconvex track} or \defstyle{EIQ track} if it is essential and if the induced map on fundamental group is injective with quasiconvex image.
\end{defn}

Since $X$ is compact with hyperbolic fundamental group, its universal cover $\tilde{X}$ endowed with any $\pi_1X$--equivariant geodesic metric is Gromov hyperbolic.  Its Gromov boundary $\partial \til{X}$ is a compact metrizable space, and any subset $Y\subseteq \til{X}$ has a well-defined limit set $\Lambda(Y)\subseteq \til{X}$.  The next lemma explains how the limit set of an elevation of a track cuts the boundary in two.
\begin{lemma}\label{lem:trackfacts}
	Let $X$ be a compact complex with hyperbolic fundamental group, and let $\TT\subseteq X$ be an EIQ track.  Let $\til{X}$ be the universal cover of $X$, and let $\til{\TT}$ be an elevation of $\TT$ to $\til{X}$.
	\begin{enumerate}
		\item\label{twocomponents} The complement $\til{X}\setminus \til{\TT}$ consists of two components, $H_1$ and $H_2$.
		\item\label{limitintersection} The limit set $\Lambda(\til{\TT})=\Lambda(H_1)\cap \Lambda(H_2)$.
		\item\label{nonemptyinterior} For $i\in\{1,2\}$, $\mathrm{Int}(\Lambda(H_1)) = \Lambda(H_1)\setminus\Lambda(\til{\TT})$ is nonempty.
	\end{enumerate}
\end{lemma}
\begin{proof}
	An elevation of a track is a track, so \eqref{twocomponents} follows directly from Lemma \ref{lem:dunwoody}.
	
	Since $\til{\TT}\subset \overline{H_i}$, $i=1,2$, it is clear that $\Lambda(\til{\TT})\subset \Lambda(H_1)\cap \Lambda(H_2)$. We shall now prove the reverse inclusion. Let $p \in \Lambda(H_1)\cap \Lambda(H_2)$. Since $\til{\TT}$ is quasiconvex,
	each $H_i$ must also be quasiconvex.  Fixing a base-point $o \in \til{\TT}$,  there exists a geodesic ray $\gamma$ based at $o$ and tending to $p$ which lies in a bounded neighborhood of both halfspaces.  Any point at bounded distance from both halfspaces is also bounded distance from their intersection $\tilde{\TT}$, so there is a sequence of points on $\til{\TT}$ converging to $p$.
	This proves $ \Lambda(H_1)\cap \Lambda(H_2) \subset \Lambda(\til{\TT})$ and establishes \eqref{limitintersection}.
	
	Fix $i\in\{1,2\}$.  The track $\TT$ is essential, which implies  that there are points $\{x_j\}_{j\in\bN}$ in $H_i$ so $d(x_j,\til{\TT})\to \infty$.  The stabilizer of $\til{\TT}$ acts cocompactly on it. Thus $\{x_j\}$ can be chosen so that the closest point projections to $\til{\TT}$ lie in  some compact set $K\subset\til{\TT}$.  This implies that there exists an upper bound $C_0$ on the Gromov inner product of any $x_j$ and any point in $\til \TT$ with respect to some base-point in $K$. This means that, up to passing to a subsequence, $x_j$ converges to a point $p$ whose Gromov  inner product with any point in $\til \TT$ is at {most} $C_0$. Therefore $p$ is contained  in an open subset of $\Lambda(H_i)$  missing $\Lambda(\TT)$.   This establishes \eqref{nonemptyinterior}.
\end{proof}

\subsection{EIQ track implies cubulable}\label{sec-eiqtrack1} 
{
	\begin{defn}\label{def-fi}
		We say that a track $\TT$ is freely indecomposable if $\pi_1(\TT)$ is infinite and
		freely indecomposable.
\end{defn}}

In this subsection we use Wise's Quasiconvex Hierarchy Theorem \ref{wise-hierarchy} to prove:
\begin{theorem}\label{eiqimpliescube} 
	Let $M$ be a closed surface bundle over a finite graph $\GG$, so that $\pi_1M$ is hyperbolic. Suppose that $M$ contains an EIQ freely indecomposable surface bundle track $\TT$. Then $\pi_1M$ 
	{admits a quasiconvex hierarchy and is therefore} cubulable.
\end{theorem}

Let $M$ be a surface bundle over a finite graph $\GG$ with fiber $S$.  
Also, let $\pi_1(M)=G$.
The key point we want to show that is that the vertex groups coming from cutting along an EIQ track are themselves cubulated.

\begin{lemma}\label{trackintersectsfiber}
	Let $0 \in \GG$ be an arbitrary point and
	let $S_0$ be the fiber over $0 \in \GG$.
	If $\TT$ is an EIQ in $M$, then $\TT\cap S_0$ is nonempty.
\end{lemma}
\begin{proof} We use the observations in Lemma \ref{lem:trackfacts} to establish this lemma.
	
  Suppose not. Then $S_0 \subset M \setminus \TT$. Let $M'$ be the component of $M \setminus \TT$ containing $S_0$. Since $\TT$ is an EIQ track, its quasiconvex universal cover $\til{\TT}$ separates the universal cover $\til M$ into {two components $N_1, N_2$ by Lemma \ref{lem:trackfacts} \eqref{twocomponents}.  Without loss of generality we assume that an elevation $\til{M'}$ of $M'$ is contained in $N_1$.
    By Lemma \ref{lem:trackfacts} \eqref{limitintersection}
          the limit set $\partial \til{\TT}$ equals the 
		intersection $\Lambda(N_1) \cap \Lambda (N_2)$.
		Further, by Lemma \ref{lem:trackfacts} \eqref{nonemptyinterior},  $\Lambda(N_2) \setminus \partial \til{\TT}$
		is non-empty, and so \begin{equation}\label{eq:alldG}\Lambda(N_1)\ne\partial G.\end{equation}
                However $\pi_1(S_0)$ is normal in $G$, and hence its limit set is $\partial G$.  An elevation of $S_0$ is contained in $\til{M'}\subset N_1$, and is preserved by $\pi_1(S_0)$, so $\Lambda(N_1) = \partial G$, contradicting~\eqref{eq:alldG}.}
\end{proof}

Since $0\in \GG$ was arbitrary, Lemma \ref{trackintersectsfiber} shows that $\TT$ cuts every fiber $S$ of $M$. To cut $M_c = M \setminus N_\ep(\TT)$ further along quasiconvex tracks, we shall need the following refinement of a theorem of Scott and Swarup \cite{scottswar} due to Dowdall, Kent and Leininger \cite{kld-coco} (see also \cite{mahan-rafi}).

\begin{theorem}\label{kldmrlemma}
	Let $1 \to \pi_1S \to G \to F\to 1$ be an exact sequence of hyperbolic groups, where $F$ is free.  Let $H <\pi_1S$ be a finitely generated infinite index subgroup of $\pi_1S$. Then $H$ is quasiconvex in $G$.
\end{theorem}

{
	\begin{defn} { Let  $\TT\subset M$ be a 
	  disjoint union of a finite number of surface bundle tracks}. Say that $\TT\subset M$ is {\bf vertex-essential} if $\TT$ intersects each vertex fiber in a union of essential simple closed curves.  The \defstyle{complexity} of $\TT$ is the number of components of $\Pi^{-1}(\VV)\cap \TT$, where $\VV$ is the vertex set of the $1$-complex $\GG$.
	\end{defn}
	\begin{lemma}[Surgery on surface bundle tracks]\label{lem:surgery}
		Let $\TT\subset M$ be a freely indecomposable EIQ surface bundle track of minimal complexity among those carrying the same fundamental group.  Then $\TT$ is  vertex-essential.
	\end{lemma}
	\begin{proof}
		Choose some $v\in \VV(\GG)$ and let $S_v = \Pi^{-1}(v)$ be the vertex fiber at $v$. We argue by contradiction. Suppose $\TT\cap S_v$ contains an inessential curve $\sigma$.  Without loss of generality, we may suppose that $\sigma$ is innermost, so it bounds a disk $D_\sigma \subset S_v$ which is otherwise disjoint from $\TT$.  Let $\Star(v)$ be a small neighborhood of $v$ in $\GG$, chosen so that we can identify $\Pi^{-1}(\Star(v))$ with $S_v\times \Star(v)$, and so that under this identification $D_\sigma\times \Star(v)$ meets $\TT$ exactly in $\partial D_\sigma\times \Star(v)$.  We surger $\TT$ by removing $\partial D_\sigma\times \Star(v)$ and adding in $D_\sigma \times \partial \Star(v)$.  Let $\TT^0$ be the resulting pattern (which may be disconnected) and note that the complexity of $\TT^0$ is one less than the complexity of $\TT$. 
		
		 The fundamental group of $\TT$ is the fundamental group of a graph of groups with vertex groups equal to the fundamental groups of the components of $\TT^0$, together with a trivial group representing $\TT\cap S_v$, with an edge joining the trivial vertex group to one of the other vertex groups for each edge of $\Star(v)$. Since $\pi_1(\TT)$ is freely indecomposable, there exists a distinguished component
	$\TT_0^0$ of $\TT^0$,  such that the vertex group $\pi_1(\TT_0^0)$ equals $\pi_1(\TT)$.
  The complexity of $\TT_0^0$ is at most the complexity of $\TT^0$, which is strictly less than the complexity of $\TT$.  This contradicts the assumption that $\TT$ was minimal complexity among those tracks carrying the same fundamental group.
%
	\end{proof}
	
	\begin{proof}[Proof of Theorem~\ref{eiqimpliescube}]
		We are given $M$ a closed surface bundle over $\GG$ so that $\pi_1M$ is hyperbolic and $M$ contains an EIQ freely indecomposable surface bundle track $\TT$.  If $\TT$ is $1$-sided, we replace it by the $2$-sided track which forms the boundary of a regular neighborhood of $\TT$ in $M$.  The fundamental group of this new track is index $2$ in the fundamental group of the old track, so $\pi_1$--injectivity, quasiconvexity, and free indecomposability are clearly preserved.
			{By} Lemma~\ref{lem:surgery}, we may further assume that $\TT$ is vertex-essential.
		
		Since $\TT$ is $2$-sided and $\pi_1$--injective, $\pi_1M$ has a one-edge splitting (either an amalgam or HNN extension) with edge group $\pi_1\TT$.  Let $C$ be one of the components of $M\setminus N(\TT)$ where $N(\TT)$ is {an open product neighborhood} of $\TT$.  The quasi-convexity of $\pi_1\TT$ implies that $\pi_1C$ is also quasi-convex.  
		By the Quasi-convex Hierarchy Theorem~\ref{wise-hierarchy}, it suffices to show that $\pi_1C$ {has a quasiconvex hierarchy}.
		
		Let $v$ be a vertex of $\GG$ and let $S_v = \Pi^{-1}(v)$ be the corresponding vertex fiber.  Since $\TT$ is vertex-essential, the intersection of $C$ with $S_v$ is a union of proper essential subsurfaces $F_v^1\ldots F_v^{n(v)}$.  (They are proper subsurfaces by Lemma~\ref{trackintersectsfiber}.)  By Theorem~\ref{kldmrlemma}, each $\pi_1F_v^i$ is quasiconvex in $\pi_1M$.  {An open product neighborhood} of $S_v$ meets $C$ in a product of these subsurfaces with $\Star(v)$.  Each component of the boundary in $C$ of this neighborhood is a ($2$-sided) copy of some $F_v^i$.  Taken together these surfaces with boundary give a decomposition of $\pi_1C$ as a bipartite graph of groups.  All edges and the vertices of one color are labeled by the quasi-convex free subgroups $\pi_1F_v^i$; the vertices of the other color are labeled by fundamental groups of $3$--manifolds with boundary sitting inside $S\times I$ regions.  These $3$--manifold groups are infinite index in the fiber group, so (again applying Theorem~\ref{kldmrlemma}) they are quasi-convex free subgroups as well.  Since all the vertex groups in this graph of groups are free, and all the edge groups are quasi-convex, we may apply the Quasi-convex Hierarchy Theorem~\ref{wise-hierarchy} to conclude that $\pi_1C$ (and hence $\pi_1M$) is virtually special.
	\end{proof}
	
}

\noindent {\bf Acknowledgments:} MM thanks Jason Manning and Michah Sageev for agreeing to write Appendix~\ref{sec-app}  jointly,
and for their insights and participation during the course of the project. This paper would not have been possible without their  contribution. He gratefully acknowledges
several extremely helpful conversations with Dani Wise who generously shared many of his ideas on
cubulations with him 
and for the hospitality the McGill Mathematics Department extended during May 2015. He thanks Chris Leininger for telling him the proof of Proposition \ref{isometrictighttree} and Ken Bromberg for helpful correspondence and 
the proof of Proposition \ref{isometrictighttree-sep}. 

We are all grateful to the anonymous referees of this paper for going through an earlier draft with much 
	care and patience, and for making many perceptive comments that have improved the quality of the paper. Their diligence and attention to detail ensured that we flesh out and correct several gaps in an earlier version.

{ JM, MM, MS were supported in part by  NSF 
	Grant No. DMS-1440140, during a Fall 2016 program in Geometric Group Theory at MSRI, Berkeley, and
	Grant No.\ 346300 for IMPAN from the Simons Foundation and the matching 2015-2019 Polish MNiSW fund. MM and MS acknowledge ICTS for a workshop on Groups, Geometry and Dynamics in November 2017.}

\end{document}